\documentclass[reqno]{amsart}

\calclayout

\usepackage{amssymb,amsmath,array,multirow,makecell,blindtext,amsthm,enumitem,mathtools,amscd,tikz-cd,dsfont,hyperref,url,caption,float,placeins,bm}
\usepackage[nobysame, alphabetic]{amsrefs}
\usepackage[font=small, labelfont=bf]{caption}
\hypersetup{
    colorlinks=true,
    linkcolor=blue,
    citecolor=black,
    filecolor=magenta,      
    urlcolor=cyan,
}

\numberwithin{equation}{section}

\theoremstyle{plain}
\newtheorem{theorem}{Theorem}[section]
\newtheorem{corollary}[theorem]{Corollary}
\newtheorem{lemma}[theorem]{Lemma}
\newtheorem{remark}[theorem]{Remark}
\newtheorem{proposition}[theorem]{Proposition}
\newtheorem{definition}[theorem]{Definition}

\newtheorem{conjecture}[theorem]{Conjecture}
\newtheorem{example}[theorem]{Example}
\newtheorem{lthm}{Theorem} % theorems with letters (for intro)
\newtheorem{corl}[lthm]{Corollary}

\newcommand{\Z}{\mathbf{Z}}
\newcommand{\Q}{\mathbf{Q}}
\newcommand{\R}{\mathbf{R}}
\newcommand{\C}{\mathbf{C}}
\newcommand{\F}{\mathbf{F}}
\newcommand{\Pp}{\mathfrak{p}}
\newcommand{\Oo}{\mathcal{O}}
\newcommand{\T} {\mathcal{T}}
\newcommand{\cI}{\mathcal{I}}

\newcommand{\cB}{\mathcal{B}}

\newcommand{\cL}{\mathcal{L}}
\newcommand{\Gg}{\mathcal{G}}
\newcommand{\Div}{\operatorname{Div}}
\newcommand{\Frac}{\operatorname{Frac}}
\newcommand{\Gal}{\operatorname{Gal}}
\newcommand{\SL}{\operatorname{SL}}
\newcommand{\sgn}{\operatorname{sgn}}

\newcommand{\ord}{\operatorname{ord}}
\newcommand{\cor}{\operatorname{cor}}

\newcommand{\Symb}{\operatorname{Symb}}

\newcommand{\Hom}{\operatorname{Hom}}
\newcommand{\im}{\operatorname{im}}
\newcommand{\NP}{\operatorname{NP}}
\newcommand{\GL}{\operatorname{GL}}

\newcommand{\Sym}{\operatorname{Sym}}

\newcommand{\nf}{\normalfont}

\newcommand{\e}{\varepsilon}

\newcommand{\Mod}[1]{\ \mathrm{mod}\ #1}

\title[Iwasawa Invariants of Modular Forms with $a_p=0$]{Iwasawa Invariants of Modular Forms with $\pmb{a_p=0}$}
\author{Rylan Gajek-Leonard}

\address[]{Department of Mathematics, Union College, Schenectady, NY}
\email[Rylan Gajek-Leonard]{gajekler@union.edu}


\begin{document}

\begin{abstract} Fix a prime $p$ and a cuspidal newform $f$ of level coprime to $p$ with $a_p=0$. Attached to $f$ are signed $p$-adic $L$-functions $L_p^\pm(f)$ and Mazur-Tate elements $\theta_n(f)$, both of which encode arithmetic data about $f$ along the cyclotomic $\Z_p$-extension of $\Q$. We compute the Iwasawa invariants of Mazur-Tate elements in terms of the corresponding invariants of the signed $p$-adic $L$-functions. As corollaries, we determine the $p$-adic valuation of critical values of the $L$-function of $f$, and describe a relation between the Iwasawa invariants of $p$-congruent modular forms of weights  2 and $p+1$. Our results provide an asymptotic method for computing the signed Iwasawa invariants attached to newforms of any weight $k\geq 2$ with $a_p=0$.
\end{abstract}

\maketitle

%%%%%%%%%%%%%%%%%%%%%%%%%%%%%%%%%%%%%%%%%%%%%%%%%%%%%%%%%%%%%%%%%%%%%%%%%%%%%%%%%%%%%%%%%%%%%%%%%%%%%%%%%%%%%%%%
\section{Introduction}

Fix a prime $p$ and a cuspidal newform $f=\sum_{n\geq 1} a_nq^n$ of weight $k\geq 2$ and level $\Gamma_1(N)$, where we assume $p\nmid N$.  Let $K$ denote the finite extension of $\Q_p$ generated by the images of $a_n$ for all $n$ under a fixed choice of embedding $\overline{\Q} \hookrightarrow\overline{\Q}_p$. Let $\Oo$ denote the valuation ring of $K$. For simplicity, we assume throughout this introduction that $p$ is odd.  

Attached to $f$ is a sequence $\theta_n(f)\in \Oo[\Gamma_n]$ of \emph{Mazur-Tate elements}, where $\Gamma_n\cong \Z/p^n\Z$ is the Galois group of the $n$th-layer of the cyclotomic $\Z_p$-extension of $\Q$. These elements were first introduced by \cite{MTT} and \cite{MT}  in order to formulate $p$-adic analogues of conjectures of Birch and Swinnerton-Dyer. 
Like $p$-adic $L$-functions, Mazur-Tate elements interpolate special values of the complex $L$-function of $f$ and have associated Iwasawa $\lambda$ and $\mu$-invariants.  Understanding the behavior of these invariants is a central goal in Iwasawa theory: various `main conjectures' (for $p$-adic $L$-functions) and `refined conjectures' (for Mazur-Tate elements) relate the Iwasawa invariants of these objects to the arithmetic structure of certain Selmer groups attached to $f$.

In this article, we study the relationship between the Iwasawa invariants of  $p$-adic $L$-functions and those attached to Mazur-Tate elements. If $f$ is ordinary at $p$ (i.e., $\ord_p(a_p)=0$) or if $f$ is non-ordinary of weight $k=2$, it is known that the Iwasawa invariants of Mazur-Tate elements recover corresponding invariants coming from a $p$-adic $L$-function.  For instance, under some mild assumptions we have for $n\gg0$ that
\begin{equation}\label{ord1}
\lambda(\theta_n(f))=
\begin{cases}
\lambda\big(L_p(f)\big) & \text{if $f$ is ordinary,}\\
\lambda\big(L_p^{\sharp/\flat}(f)\big)+q_n &\text{if $f$ is non-ordinary and $k=2$,}
\end{cases}
 \end{equation} 
where $L_p(f),L_p^{\sharp/\flat}(f)\in \Oo[[T]]\otimes K$ are the $p$-adic $L$-functions of \cite{visik76,amicevelu75} and \cite{sprung17}, respectively, and  
$$
q_n=\begin{cases}p^{n-1}-p^{n-2}+\cdots +p-1 & \quad\text{if $n$ is even,}\\
p^{n-1}-p^{n-2}+\cdots +p^2-p & \quad\text{if $n$ is odd.}
\end{cases}
$$
See \cite[Proposition 3.7]{PW} and \cite[Theorem 4.1]{PW} (together with \cite[Corollary 8.9]{sprung17}) for precise statements of this result in the ordinary and non-ordinary cases, respectively. 

When $f$ is non-ordinary of weight $k>2$, the relationship between the Iwasawa invariants of Mazur-Tate elements and  $p$-adic $L$-functions is not well-understood.
Seminal work of Pollack \cite{pollack03} associates signed $p$-adic $L$-functions $L_p^\pm(f)$ to $f$ of any weight assuming $a_p(f)=0$. (Sprung's $\sharp/\flat$ $p$-adic $L$-functions are only defined when $k=2$ and $\ord_p(a_p)>0$, but they agree with Pollack's plus/minus $p$-adic $L$-functions when $a_p=0$.) Our first theorem shows that, up to certain explicit polynomials, the signed $p$-adic $L$-functions can be viewed as lifts of the Mazur-Tate elements. This generalizes a result of Pollack
\cite[Proposition 6.18]{pollack03} to higher weight modular forms.  
In what follows, we let $\e_n$ denote the sign of $(-1)^n$.

%%%%%%%%%
\begin{lthm}[Theorem \ref{lifts}] \label{lifts1} Let $f\in S_k(\Gamma_1(N))$ be a newform with $a_p(f)=0$. There are units $u_{n}\in \Z_p[\e(p)]^\times$ such that for all $n\geq 0$ we have
$$
\theta_{n}(f) \equiv u_n\log_{k,n}^{-\e_n}  L_{p}^{-\e_n}(f)\Mod \omega_n, 
$$
where $\omega_n=(1+T)^{p^n}-1$ and $\log_{k,n}^{\pm}\in\Z_p[T]$ (see \S\ref{section_lifts}) are explicit products of cyclotomic polynomials.
\end{lthm}

Theorem \ref{lifts1} forms the basis of our main result, which relates the Iwasawa invariants of $\theta_n(f)$ to those of the signed $p$-adic $L$-functions, thereby generalizing \eqref{ord1} to higher weights.  To state this result, we first introduce some notation. Let $\lambda^\pm(f)$ and $\mu^\pm(f)$ denote the Iwasawa invariants of $L_p^\pm(f)$ and define 
\begin{align*}
\nu&=\bigg\lfloor\frac{k-p-2}{p^2-1}\bigg\rfloor,\qquad  \nu^- =\nu(p-1)+1,\qquad \nu^+ = p\nu^{-},\\
\iota^\pm(f)&=\begin{cases}
0 &\text{if $2\leq k\leq p+1$,}\\
\big\lfloor\frac{p(k-2)-1}{p^2-1}\big\rfloor&\text{if  $k\geq p+2$ and $k\not\equiv p+2 \Mod (p^2-1)$,}\\
\frac{p(k-2-p)}{p^2-1}&\text{if  $k\equiv p+2 \Mod (p^2-1)$ and $\lambda^\pm(f) <\nu^{\pm}$,}\\
\frac{p(k-2)-1}{p^2-1}& \text{if  $k\equiv p+2 \Mod (p^2-1)$ and $\lambda^\pm(f) \geq \nu^{\pm}$.}
\end{cases}
\end{align*}
 
\begin{lthm}[Theorems \ref{theorem_smallweights} and \ref{main}] \label{main1} Let $f\in S_k(\Gamma_1(N))$ be a newform with $a_p(f)=0$. 
For all $n\gg0$ we have
\begin{align*}
 \mu\big(\theta_n(f)\big)&=\mu^{*}(f)+\iota^{*}(f),\\
 \lambda\big(\theta_n(f)\big) &=\lambda^{*}(f)+(k-1)q_n-\iota^{*}(f)(p^n-p^{n-1}),
\end{align*}
where $*=-$ if $n$ is even and $*=+$ if $n$ is odd. 
\end{lthm}

Because Mazur-Tate elements are easily computable, Theorem \ref{main1} provides an asymptotic method for computing the signed Iwasawa invariants attached to newforms with $a_p=0$, as seen in the following example.  
 
 %%%%%%
\begin{example}\label{compiwex}\nf Let $p=3$ and consider the weight $6$ newform 
$$
f=q - 4q^2 + 16q^4 - 14q^5 +O(q^6)\in S_6(\Gamma_0(26)),%- 170q^7 - 64q^8+\cdots.
$$ 
with LMFDB label \href{https://www.lmfdb.org/ModularForm/GL2/Q/holomorphic/26/6/a/a/}{\texttt{26.6.a.a}}. Theorem \ref{main1} implies 
$$
\lambda(\theta_n(f))=\lambda(L_3^\pm(f))+5q_n-(3^n-3^{n-1}),\qquad n\gg0.
$$
The following table contains the first eight terms of $\lambda(\theta_n(f))$ and $5q_n-(3^n-3^{n-1})$. 
\begin{center}
\begin{tabular}{|c||c|c|c|c|c|c|c|c|}
\hline
 $n$ & 1 & 2 &3&4&5&6&7&8\\  \hline
 $\lambda(\theta_n(f))$ & 1 & 5 &17&47&143&425&1277&3827\\  \hline 
$5q_n-(3^n-3^{n-1})$& $-2$ & 4 &12&46&138&424&1272&3826\\ 
\hline
\end{tabular}
\end{center}
This suggests that $\lambda(L_3^+(f))=5$ and $\lambda(L_3^-(f))=1.$
\end{example}

\begin{remark} \nf \phantom{}
\begin{enumerate}
\item In Appendix \ref{tablessection}, we present tables containing over two hundred pairs of invariants $\lambda(L_p^\pm(f))$ predicted by Theorem \ref{main} at the primes $p\in\{2,3,5\}$. When $p$ is odd, our data suggests that both $\lambda(L_p^+(f))$ and $\lambda(L_p^-(f))$ increase with the weight of $f$, but at different rates depending on the sign. At weight $p+1$, the difference between the sizes of the signed $\lambda$-invariants can be explained by Corollary \ref{coro_C} below. 
\item The behavior of $\lambda(\theta_n(f))$ in the ordinary (or non-ordinary weight $k=2$) case can be made effective: one can show that if $\mu(\theta_n(f))$ vanishes for some $n$ then equation  \eqref{ord1} holds at this particular value of $n$, meaning that the $\lambda$-invariant of the corresponding plus/minus $p$-adic $L$-function can be determined exactly. When $k\leq p+1$, Theorem \ref{main1} can be made effective (see Theorem \ref{theorem_smallweights}), although determining the smallest $n$ which guarantees the behavior requires first knowing the values $\lambda^\pm(f)$. At higher weights, we do not give explicit lower bounds on $n$, however examples suggests that such bounds may also depend on the size of $\lambda^\pm(f)$ -- see Remark \ref{remark_effective}. 
\end{enumerate}
\end{remark}

Let $\F$ be the residue field of $K$ and let $\overline{\rho_f}:\Gal(\overline{\Q}/\Q)\rightarrow \GL_2(\F)$ denote the residual Galois representation attached to $f$, which we assume to be irreducible.  In the pioneering work of Pollack and Weston \cite{PW}, the authors study the Mazur-Tate elements of modular forms $f$ of weight $k>2$ and level $\Gamma_0(N)$ under the assumption that $\overline{\rho_f}$ has Serre weight two (this implies $\overline{\rho_f}\cong\overline{\rho_g}$ for some weight two eigenform $g$). By using a mod $p$ multiplicity one result due to Ribet \cite{ribet91}, Pollack and Weston compute (under constraints on either the weight or slope of $f$) the Iwasawa invariants of $\theta_n(f)$ in terms of invariants attached to a weight two  eigenform $g$ whose residual representation is isomorphic to $\overline{\rho_f}$.
In particular, when $f$ is non-ordinary of weight $p+1$ then $f$ has Serre weight two and the results of \cite{PW} apply, allowing us to write $\lambda(\theta_n(f))$ in two ways: the first in terms of $\lambda^\pm(f)$ (from Theorem \ref{main1}) and the second in terms of $\lambda^{\sharp/\flat}(g)$ (from \cite{PW}). Combining these descriptions yields the following corollary. 

%%%%%%%
\begin{corl}[Corollary \ref{coro_p1}]\label{coro_C} Let $f\in S_{p+1}(\Gamma_0(N))$ be a newform with $a_p(f)=0$. Then there exists a $p$-non-ordinary eigenform $g\in S_2(\Gamma_0(N))$ with $\overline{\rho_f}\cong \overline{\rho_g}$ such that  if  $\mu\big(L_{p,\psi}^{\sharp}(g)\big) =\mu\big(L_{p,\psi}^{\flat}(g)\big) =0$ then 
\begin{align*}
\lambda^+(f)&= \lambda^{\flat}(g)+p-1, \quad \text{and}\\
\lambda^-(f)&=\lambda^\sharp(g).
\end{align*}
\end{corl}

This type of result falls into the same category as the work of \cite{EPW} and \cite{HL19} as it asserts a relation between the Iwasawa invariants of two modular forms whose residual Galois representations are isomorphic. Numerical examples are given \S\ref{pminvsection}.

 Theorem \ref{main1} together with results of \cite{pollack03} also yields a concise description of the $p$-adic valuation of twists of critical $L$-values for $f$ in terms of the Iwasawa invariants of Mazur-Tate elements. The valuation at $s=1$ is stated below; see Corollary \ref{ordLval} for a more general statement at the other critical values $s=j+1$ with $j\in \{0,\dots, k-2\}$.
   
 %%%%%%%
\begin{corl} (Corollary \ref{ordLval}) \label{p1two} Let $f\in S_k(\Gamma_1(N))$ be a newform with $a_p(f)=0$. Then for $n\gg0$ we have 
\begin{align*}
\ord_p\bigg(\frac{p^{n+1}\tau(\chi)}{\Omega_f^{\sgn(-1)^k}}L(f_{\chi^{-1}},1)\bigg)&=
\mu\big(\theta_{n}(f)\big)+\frac{\lambda\big(\theta_{n}(f)\big)}{p^n-p^{n-1}}\\
&=\mu(L_{p}^*(f))+\frac{(k-1)q_n+\lambda(L_{p}^{*}(f))}{p^n-p^{n-1}},
\end{align*}
where $\Omega^\pm_f\in \C^\times$ are cohomological periods (see \S \ref{section_coho}), $\tau(\chi)$ is a Gauss sum, and $*=-$ if $n$ is even and $*=+$ if $n$ is odd. 
\end{corl}

 In the elliptic curve case, these valuations are known to be closely related (via the Birch and Swinnerton-Dyer conjecture) to the growth of Shafarevich-Tate groups along the finite layers of the cyclotomic $\Z_p$-extension of $\Q$ (see \cite{pollack03,sprung17,PR,kurihara02}). Assuming the Bloch-Kato conjecture \cite[5.15.1]{blochkato}, we suspect that the valuations above might be analogously related to the growth of Bloch-Kato-Shafarevich-Tate groups for modular forms along the cyclotomic line.
 
 Finally, we use Theorem \ref{main1} together with computational evidence to formulate Conjecture \ref{conj_congruence}, which describes a possible relationship between  signed $\lambda$-invariants attached to pairs of cuspforms (of the same level but possibly different weights) with $a_p=0$ and isomorphic residual representations.

  %%%%
 \subsection{Sketch of proof of Theorem \ref{main1}}  In light of 
Theorem \ref{lifts1}, the main task is to relate the Iwasawa invariants of the lifts $\log_{k,n}^{\pm}  L_{p}^{\pm}(f)\in \Oo[[T]]\otimes K$ to those of  the Mazur-Tate elements $\theta_n(f)$. It is well-known (see \cite[Remark 4.3]{pollack04} or Lemma \ref{small_lam}) that if $F\in \Oo[[T]]\otimes K$ and $\lambda(F)<p^n$ then the Iwasawa invariants of $F$ and $\pi_n(F):= F\Mod \omega_n$ agree. Directly from the definitions of $\log_{k,n}^\pm$, we have (see Lemma \ref{omeganinv})
\begin{align*}
\lambda(\log_{k,n}^{-\e_n}  L_{p}^{-\e_n}(f))&=(k-1)q_n+\lambda^{-\e_n}(f),\\
\mu(\log_{k,n}^{-\e_n}  L_{p}^{-\e_n}(f))&=\mu^{-\e_n}(f).
\end{align*}
In particular, if $k\leq p+1$ then $(k-1)q_n+\lambda^\pm(f)<p^n$ for $n\gg0$, from which Theorem \ref{main1} at weights $k\leq p+1$ follows by combining Theorem \ref{lifts1} with the statements of the previous two sentences.
For weights $k\geq p+2$, it can happen that $(k-1)q_n+\lambda^\pm(f)\geq p^n$, in which case the Iwasawa invariants of $\log_{k,n}^{\pm}  L_{p}^{\pm}(f)$ and those of $\pi_n(\log_{k,n}^{\pm}  L_{p}^{\pm}(f))$ need not agree.

We consider this problem in isolation in \S\ref{section_finite-layer}: for fixed $n$ and $F$ with $\lambda(F)\geq p^n$, how are the Iwasawa invariants of $F$ and $\pi_n(F)$ related? We answer this question for a specific class of functions that we call \emph{$p$-large at $n$} (see Definition \ref{pldef}). Roughly speaking, these are functions $F$ with $\lambda(F)\geq p^n$ and large $p$-adic valuation in their first $\lambda(F)-1$ coefficients. We prove the following theorem, which allows one to pass from invariants `upstairs' in the full Iwasawa algebra to those `at layer $n$'.  Note the visual similarity between this theorem and the statement in Theorem \ref{main1}.

%%%%%%
\begin{lthm}[Theorem \ref{mainpl}]\label{mainpl1} If $F\in \Oo[[T]]\otimes K$ is $p$-large at $n$ then
\begin{align*}
\mu(\pi_n(F))&=\mu(F)+I_n(F), \quad \text{and}\\
\lambda(\pi_n(F))&=\lambda(F)-I_n(F)(p^n-p^{n-1}),
\end{align*}
where
$$
I_n(F)=\bigg\lfloor\frac{\lambda(F)-p^n}{p^n-p^{n-1}}\bigg\rfloor+1.
$$
\end{lthm}
 The proof of this theorem involves a careful study of the remainders of power series upon division by $\omega_n=(1+T)^{p^n}-1$, which may be of independent interest. Our basic technique is to define, for each integer $i$, a sequence $(c_N^{(i)})_{N\geq0}$ (see Definition \ref{def_cN}) with the property that if $F=\sum_{n\geq 0}b_iT^i\in \Oo[[T]]\otimes K$ then
\begin{equation}\label{proj}
F\equiv  \sum_{i=0}^{p^n-1}\bigg(b_i+\sum_{N=0}^\infty b_{p^n+N}c_N^{(i)}\bigg)T^i\Mod \omega_n. 
\end{equation}
This allows us to use the sequence $(c_N^{(i)})$ to study the behavior of $F$ mod $\omega_n$. 

We then show (using a Newton polygon argument) that $L_p^\pm(f)\log_{k,n}^\pm$ are $p$-large at $n\gg0$ when $k\geq p+2$ (see Proposition \ref{lpl}), thereby allowing us to use Theorem \ref{mainpl1} to compute  the Iwasawa invariants of $\pi_n(\log_{k,n}^{\pm}  L_{p}^{\pm}(f))$ and complete the proof of Theorem \ref{main1}. 

 \subsection{Outlook} We remark that the $a_p=0$-hypothesis is fairly restrictive at weights $k>2$, though it is satisfied for any $p$-non-ordinary CM modular form. We need this assumption in order to define the signed $p$-adic $L$-functions of Pollack \cite{pollack03}. The construction of Pollack's signed $p$-adic $L$-functions has been generalized to higher weight non-ordinary modular forms in \cite{LLZ0} using $p$-adic Hodge theory.
 It would be interesting to know whether relations analogous to those in Theorem \ref{main1} hold between the Iwasawa invariants of Mazur-Tate elements and the $p$-adic $L$-functions of \cite{LLZ0}. 
 
 Indeed, computations suggest that Mazur-Tate elements with sufficiently large slope also follow the behavior of Theorem \ref{main1}. For example, the $a_p=0$-type behavior is consistent with a `strange example' of \cite[\S7]{PW}: there is a 3-non-ordinary newform $f\in S_{18}(\Gamma_0(17))$  %(LMFDB label \href{https://www.lmfdb.org/ModularForm/GL2/Q/holomorphic/17/18/a/b/}{\texttt{17.18.a.b}}) 
with $\ord_3a_3(f)=5$ and Mazur-Tate elements satisfying  $\lambda(\theta_n(f))=3^{n}-3^{n-2}+q_{n-2}$. Note that, despite the fact that $a_3(f)\neq 0$, these $\lambda$-invariants follow the same pattern of Theorem \ref{main} since 
$$
3^{n}-3^{n-2}+q_{n-2}=(18-1)q_n-5(3^n-3^{n-1})+\begin{cases}
4 & \text{$n$ even,}\\
12 & \text{$n$ odd.}\\
\end{cases}
$$
This suggests that the constants 4 and 12 may be the $\lambda$-invariants attached to a pair of $p$-adic $L$-functions for $f$ -- perhaps those of \cite{LLZ0}. The construction of the $p$-adic $L$-functions of \cite{LLZ0} requires a choice of basis of the Wach module attached to $\rho_f$, for which an explicit such choice exists (due to Berger, Li, and Zhu -- see \cite{bergerlizhu04} and \cite[\S4]{LLZ0}) when $f$ has sufficiently large slope $\ord_p(a_p)>\lfloor\frac{k-2}{p-1}\rfloor$.\footnote{While the example of Pollack and Weston does not satisfy this bound, it is known that this bound need not be optimal (see \cite[Remark 4.1.2]{bergerlizhu04}).} In particular, the Berger--Li--Zhu basis involves the same signed logarithms used to define $\log_{k,n}^\pm$ in Theorem \ref{main1}, which could explain why similar behavior is seen in the Mazur-Tate elements across modular forms with large slope.

%%%%%%%%%%%%%%%%%%%%%%%%%%%%%%%%%%
\subsection{Acknowledgements} The author would like to thank Robert Pollack and Tom Weston for their guidance and support; Tori Day, Jeffrey Hatley, Antonio Lei, Florian Sprung, and Siman Wong for many helpful conversations; and Vicen\c{t}iu Pa\c{s}ol and Alexandru Popa for generously sharing their Magma package on period polynomials (see \cite{PasolPopa}), without which many of the computations in this paper would not be possible. We also thank the anonymous referees for their helpful comments.

%%%%%%%%%%%%%%%%%%%%%%%%%%%%%%%%%%
\subsection{Notation}

Fix throughout a prime $p\geq 2$ and an embedding $\iota:\overline{\Q}\hookrightarrow \overline{\Q_p}$. Let $\ord_p$ denote the unique valuation on $\overline{\Q_p}$ satisfying $\ord_p(p)=1$. For convenience, whenever $x\in \overline{\Q}$ we omit the embedding $\iota$ and simply write $\ord_px$ to denote the $p$-adic valuation of $\iota(x)$. 
For a finite extension $K/\Q_p$ with valuation ring $\Oo\subseteq K$, let $\Lambda=\Oo[[T]]$, $\Lambda_n=\Oo[T]/(\omega_n)$, $\Lambda_K=\Lambda\otimes K$, and $\Lambda_{K,n}=\Lambda_n\otimes K\cong K[T]/(\omega_n)$, where $\omega_n=(1+T)^{p^n}-1$. 

Let $S_k(N,\e)$ denote the space of cuspforms of level $\Gamma_1(N)$ and nebentype $\e$. When presenting examples, we often identify a modular form using its Magma \cite{Magma} label. (One can use this label to recover the modular form in question by entering the command \texttt{Newform("label")} in Magma.)
 We say that an eigenform $f=\sum a_nq^n$ is ordinary at $p$ if $\ord_p(\iota(a_p))=0$ and non-ordinary at $p$ if $\ord_p(\iota(a_p))>0$. 
We always write $\e_n$ to denote the sign of $(-1)^n$.  In order to deal with the prime $p=2$ simultaneously, we adopt the notation
$$
\star = \begin{cases}+&\quad\text{if $p>2$,} \\
- &\quad\text{if $p=2$.}
\end{cases}
$$
Let $\omega:(\Z/2p\Z)^\times \rightarrow \Z_p^\times$ denote the mod $p$ cyclotomic character, which (for odd primes) sends $a\in (\Z/p\Z)^\times$ to the unique $(p-1)$st  root of unity $\omega(a)\in \Z_p^{\times}$ with $\omega(a)\equiv a\Mod p$. By precomposing with reduction modulo $p$, we also view $\omega$ as a character on $\Z_p^\times$ and let $\langle \cdot \rangle:\Z_p^\times\rightarrow 1+2p\Z_p$ denote the projection $x\mapsto x\omega^{-1}(x)$. Unless otherwise stated, $\gamma$ denotes a fixed topological generator of $1+2p\Z_p$ and $\log_\gamma$ is the function $\log_p/\log_p(\gamma)$, where $\log_p$ is the usual $p$-adic logarithm. 

\setcounter{tocdepth}{1}
\tableofcontents

 %%%%%%%%%%%%%%%%%%%%%%%%%%%%%%%%%%%%%%%%%%%%%%%%%%%%%%%%%%%%%%%%%%%%%%%%%%%%%%%%%%%%%%%%%%%%%%%%%%%%%%%%%%%%
\section{Preliminaries}\label{section_prelim}

In this section we recall the definitions and basic results on Iwasawa invariants, Mazur-Tate elements, and Pollack's signed $p$-adic $L$-functions. 

\subsection{Iwasawa invariants} \label{section_iwinv}
Let $K/\Q_p$ denote a fixed finite extension with valuation ring $\Oo$ and uniformizer $\varpi$. 

\begin{definition} \nf The \emph{Iwasawa invariants} attached to a nonzero power series $F=\sum_{i=0}^\infty a_iT^i \in \Lambda_K$ are defined by
\begin{align*}
\mu(F)&= \min\{\ord_p (a_i)\mid i\geq 0\}\in e^{-1}\Z,\\
\lambda(F) &= \min \{\, i \, \mid \, \ord_p (a_i)=\mu(F)\}\in \Z_{\geq 0},
\end{align*}
where $e$ is the ramification index of $K/\Q_p$. Define $\mu(0)=\lambda(0)=\infty$. 
\end{definition}
 
By the Weierstrass preparation theorem, any $F\in \Lambda_K$ has a unique decomposition 
$$
F=p^{\mu(F)}(T^{\lambda(F)}+\varpi F_0)U,
$$
where $F_0\in \Oo[T]$ is a polynomial of degree $<\lambda(F)$ and $U\in \Lambda^\times$. 

\begin{lemma}\label{lem_basicIw} Let $x\in \overline{K}$ be nonzero and let $F,G\in \Lambda_K$.
\begin{enumerate}
\item$ \lambda(FG)=\lambda(F)+\lambda(G)$.
\item $\mu(FG)=\mu(F)+\mu(G)$.
\item$ \lambda(xF)=\lambda(F)$.
\item $\mu(xF)=\ord_px+\mu(F)$.
\end{enumerate}
\end{lemma}
\begin{proof} Parts (1) and (2) follow from the uniqueness part of the Weierstrass preparation theorem, and parts (3) and (4) are special cases of (1) and (2), respectively.
\end{proof}

\begin{definition} \nf Define the Iwasawa invariants of an element $\bar F\in \Lambda_{K,n}$ to be those of the unique smallest degree polynomial lying in the fiber of $\bar F$ under the map 
$$
\pi_n: \Lambda_{K}\twoheadrightarrow \Lambda_{K,n},\qquad F\mapsto F\Mod \omega_n.
$$
Note that $\lambda(\pi_n(F))<\deg(\omega_n)=p^n$ by definition. 
\end{definition}

%%%%%%%%%
\begin{example}\label{tpn} \nf Since $T^{p^n}\equiv -\sum_{i=1}^{p^{n}-1}\binom{p^n}{i}T^i \Mod \omega_n$ (expand $(1+T)^{p^n}-1$) and $\ord_p\binom{p^n}{i}=n-\ord_p(i)$, we have 
\begin{align*}
\mu\big(\pi_n(T^{p^n})\big)&=\mu(\textstyle\sum_{i=1}^{p^{n}-1}\binom{p^n}{i}T^i)=1,\\
\lambda\big(\pi_n(T^{p^n})\big)&=\lambda(\textstyle\sum_{i=1}^{p^{n}-1}\binom{p^n}{i}T^i)=p^{n-1}.
\end{align*}
Compare this to the fact that in $\Lambda_K$, we have $\mu(T^{p^n})=0$ and $\lambda(T^{p^n})=p^n$.
\end{example}

\begin{remark} \nf By fixing an isomorphism $\Oo[\Gamma_n]\cong \Lambda_n$, where $\Gamma_n\cong \Z/p^n\Z$, one can show using Hensel's lemma and the Weierstrass preparation theorem that the definition of Iwasawa invariants above agrees with that of \cite{PW}. See \cite[Proposition 2.1]{GCMB}. 
\end{remark}

Directly from the definitions, we have
\begin{align}
\mu(\pi_n(xF))&=\ord_px +\mu(\pi_n(F)),\label{nliw1}\\
\lambda(\pi_n(xF))&=\lambda(\pi_n(F)). \label{nliw2}
\end{align}
for any nonzero $x\in \overline{K}$. However, properties (1) and (2) of Lemma \ref{lem_basicIw} need not hold in $\Lambda_n$. For instance, using Example \ref{tpn}, we have 
\begin{align*}
\mu(\pi_1(T^p))=1 &\neq 0=\mu(\pi_1(T))+\mu(\pi_1(T^{p-1})),\\
\lambda(\pi_1(T^p))=1&\neq p =\lambda(\pi_1(T))+\lambda(\pi_1(T^{p-1})).
\end{align*}

We also see from Example \ref{tpn} that the Iwasawa invariants of $F$ and $\pi_n(F)$  need not agree in general. The relationship between these Iwasawa invariants will be examined in detail in \S\ref{section_finite-layer}. For now, observe that since $F=\lim\pi_n(F)$ one clearly has $\lambda(F)=\lambda(\pi_n(F))$ and $\mu(F)=\mu(\pi_n(F))$ for $n\gg0$. The smallest $n$ guaranteeing this equality is given by the following well-known criterion.

\begin{lemma}\label{small_lam} Let $F\in \Lambda_K$ and fix $n\geq 0$. If $\lambda(F)<p^n$ then the Iwasawa invariants of $F$ and $\pi_n(F)$ agree. 
\end{lemma}
\begin{proof} See \cite[Lemma 2.2]{GCMB} for a proof over $\Z_p[[T]]$. Since the Weierstrass preparation theorem also holds over $\Lambda_K$, the same argument applies. 
%This follows from the Weierstrass preparation theorem in $\Lambda_K$ using the same argument as given in CMB. %See also CITE for different proof. 
\end{proof}

%%%%%%
\begin{corollary} \label{nunits} If $U\in \Lambda^\times$ then the Iwasawa invariants of $\pi_n(UF)$ and $\pi_n(F)$ agree.
\end{corollary}
\begin{proof} Recall that a polynomial $D\in \Oo[T]$ is called distinguished if $D\equiv T^{\deg D}\Mod \varpi$. Using the division algorithm for distinguished polynomials \cite[Proposition 7.2]{washington82}, we can write $F=\omega_nQ_n+F_n$ for some $Q_n\in \Lambda_K$ and $F_n\in K[T]$ with $\deg F_n<p^n$. Now
$$
\lambda(\pi_n(F))=\lambda(F_n)=\lambda(UF_n)=\lambda(\pi_n(UF_n))=\lambda(\pi_n(UF)),
$$ 
where the third equality follows from Lemma \ref{small_lam} (because $\lambda(UF_n)=\lambda(F_n)<p^n$) and the final equality is because $UF\equiv UF_n\Mod \omega_n$. The same string of equalities holds for $\mu$-invariants.
\end{proof}

In the next two lemmas, we fix $F\in \Lambda_K$ with Iwasawa invariants $\lambda$ and $\mu$, and write $F=p^\mu(T^\lambda+\varpi F_0)U,$ where $F_0\in \Oo[T]$ has degree $<\lambda$.

%We also note that Iwasawa invariants can be computed by evaluating power series $F$ at $x\in\Oo_{\C_p}=\{x\in \C_p\mid \ord_p(x)\geq 0\}$ with sufficiently small valuation. 

%The following simple corollary of the Weierstrass Preparation Theorem allows us to recover the Iwasawa invariants of a function by computing valuations. In particular, it says that for any fixed $F\in \Lambda_K$ we can always find some $x\in \Oo_{\C_p}=\{x\in \C_p\mid \ord_px\geq 0\}$ so that the valuation of $F(x)$ can be written in terms of the Iwasawa invariants of $F$. 

%%%%%%%
\begin{lemma}\label{mulamfromval}  If $x\in \overline{K}$ satisfies $0\leq \ord_px<\frac{1}{e\lambda}$ then 
$$
\ord_p F(x)=\mu+\lambda\ord_px.
$$
\end{lemma}
\begin{proof} %Let $\lambda=\lambda(F)$,  $\mu=\mu(F)$, and write $F=p^\mu(T^\lambda+\varpi F_0)U,$ where $F_0\in \Oo[T]$.
%Set $\lambda=\lambda(F)$ and $\mu=\mu(F)$ and let $F=p^\mu DU$ be the Weierstrass decomposition of $F$. Write $D=T^\lambda +\varpi D_0$ for some polynomial $D_0\in \Oo[T]$ of degree $<\lambda$.
Observe that 
$$
\ord_p (x^\lambda+\varpi F_0(x))=\min(\lambda\ord_p x,1/e+\ord_p F_0(x))=\lambda\ord_p x
$$
since $\lambda\ord_p x< 1/e+\ord_pF_0(x)$. The statement now follows from the Weierstrass decomposition of $F$. %as $\ord_p F(x)=\mu+\ord_p (x)$. 
\end{proof}

Recall $\gamma$ is a fixed topological generator of the subgroup $1+2p\Z_p$ of $\Z_p^\times$. Define the map
$$
\eta:\Lambda_K\rightarrow\Lambda_K,\quad F(T)\mapsto  F\big(\gamma(1+T)-1\big),
$$
and let $\eta^j$ denote the $j$-fold composition $\eta\circ \cdots \circ \eta$. Note that  $\eta$ is a (ring) automorphism of $\Lambda_K$ and 
\begin{equation}\label{eta}
\eta^j(F)=F(\gamma^j(1+T)-1)
\end{equation}
for all $j\in \Z$.  The following lemma will be useful in later sections. 

%%%%%%%%%%%%%%
\begin{lemma}\label{etainv} The Iwasawa invariants of $\eta^j(F)$ and $F$ agree for all $j\in \Z$. 
\end{lemma}
\begin{proof}% Let $\lambda=\lambda(F)$,  $\mu=\mu(F)$, and write $F=p^\mu(T^\lambda+\varpi F_0)U,$ where $F_0\in \Oo[T]$.
%Let $\lambda$ and $\mu$ be the invariants of $F$.
%and consider the Weierstrass decomposition $F=p^{\mu}(T^{\lambda}+\varpi F_0)U$. 
%We need only show that 
 It suffices to show that $\eta^j(T^\lambda+\varpi F_0)$ is a distinguished polynomial and that $\eta^j(U)$ is a unit, in which case the Weierstrass decomposition of $\eta^j(F)$ is simply $p^{\mu}\eta^j(T^{\lambda}+\varpi F_0)\eta^j(U)$. Since $\gamma^j\equiv1 \Mod \varpi$ we have
$$
\eta^j(T^{\lambda}+\varpi F_0)\equiv \big(\gamma^j(1+T)-1)^{\lambda}\equiv T^{\lambda}\Mod \varpi.
$$
Hence $\eta^j$ preserves distinguished polynomials. Furthermore, writing $U=\sum_{n\geq 0}u_nT^n$ we see that
$$
\ord_p\eta^j(U)(0)=\ord_pU(\gamma^j-1)=\min_{n}\ord_p\big(u_n(\gamma^j-1)^n\big)=\ord_p u_0=0
$$
because $\ord_p(\gamma^j-1)\geq 1$.
\end{proof}

 %%%%%%%%%%%%%%%%%%%%%%%%%%%%%%%%%%%%%%%%%%%%%%%%%%%%%%%%%%%%%%%%%%%%%%%%%%%%%%%%%%%%%%%%%%%%%%%%%%%%%%%%%%%%
\subsection{Mazur-Tate Elements}\label{MTLpsection} 

%%%%%%%%%%%%%%%%%%%%%%%%%%%%%%%
\subsubsection{Modular symbols} \label{section_modsymb}
Let $V_g(R)=\Sym^g(R^2)$ be the algebra of degree $g$ homogeneous polynomials in the variables $X$ and $Y$ over a commutative ring $R$. There is a right action of $\GL_2(R)$ on $V_{g}(R)$ given by 
$$
P(X,Y)\mid\gamma=P((X,Y)\gamma^*)=P(dX-cY,-bX+aY)
$$
where $\gamma= \begin{psmallmatrix} a &b\\ c&d\end{psmallmatrix}$ and $\gamma^*=\begin{psmallmatrix} d &-b\\ -c&a\end{psmallmatrix}$. Writing $\Delta_0=\Div^0(\mathbf{P}^1(\Q))$, define an action of the semigroup $M_2(\Z)$ of $2\times2$ integer matrices on  $\xi\in \Hom_\Z(\Delta_0, V_{g}(R))$ by 
$$
(\xi\mid\gamma)(D)=\xi(\gamma D)\mid\gamma,\qquad D\in \Delta_0
$$
where the action of $\gamma\in M_2(\Z)$ on $\Delta_0$ is by fractional linear transformations. The space of \emph{modular symbols of weight $g$} for a congruence subgroup $\Gamma\subseteq \SL_2(\Z)$ is defined by
$$
\Symb_\Gamma(V_g(R))=\Hom_\Gamma(\Delta_0, V_{g}(R))=\{\xi:\Delta_0\rightarrow V_{g}(R)\mid \xi|\gamma=\xi \,\,\text{for all $\gamma\in \Gamma$}\}. 
$$

\noindent Fix a modular symbol $\xi\in \Symb_\Gamma(V_g(R))$ and an integer $j\in \{0,\dots, g\}$. 
\begin{definition}\nf
For coprime integers $a$ and $m$ with $1\leq a<m$, or for the pair $m=1$ and $a=0$, define $\big[\frac{a}{m}\big]_{\xi, j}\in R$ by the equation
\begin{equation}\label{modularele}
\big(\xi \mid \begin{psmallmatrix} 1 & a\\ 0&m\end{psmallmatrix}\big)(\{\infty\}-\{0\})=\sum_{j=0}^{g}\bigg[\frac{a}{m}\bigg]_{\xi,j} X^jY^{g-j}.
\end{equation}
Define the \emph{Mazur-Tate element} for the tuple $(\xi ,m,j)$ by
 $$
 \vartheta_{m,j}(\xi)=\sum_{a\in (\Z/m\Z)^\times} \bigg[\frac{a}{m}\bigg]_{\xi,j} \cdot \sigma_a\in  R[\Gg_m],
 $$
  where $\Gg_m=\Gal(\Q(\zeta_m)/\Q)$ and $\sigma_a: \zeta_m\mapsto \zeta_m^a$. (Here $\zeta_m$ denotes a primitive $m$th root of unity.)
 \end{definition}

\begin{remark} \nf The Mazur-Tate elements of \cite{PW} are obtained by taking $j=0$ and  $m$ to be a prime power in the above definition. 
\end{remark}

\subsubsection{Cohomological symbols} \label{section_coho} Suppose $R$ is a discrete valuation ring with associated (nonarchimedean) valued field $(K=\Frac(R),|\cdot |_K)$. In this case, the space $\Symb_{\Gamma}(V_g(K))$ is equipped with a norm
$$
||\xi|| = \max_{D\in \Delta_0}||\xi(D)||,
$$
where the value of $||\cdot ||$ on a homogeneous polynomial $F=\sum_{j=0}^g a_jX^jY^{g-j}\in V_g(K)$ is given by 
$
||F|| = \max_{0\leq j\leq g}|a_j|_K.
$
Following \cite{PW}, we say that $\xi\in \Symb_\Gamma(V_g(K))$ is \emph{cohomological} if $||\xi||=1$. 
By scaling by an appropriate power of a uniformizer, one can associate to any nonzero $K$-valued symbol a unique (up to $R$-units) cohomological symbol.

%%%%%
\begin{lemma}\label{coholem} If $\xi\in \Symb_\Gamma(V_g(K))$ is cohomological then $[a/m]_{\xi,j}\in R$.
\end{lemma} 
\begin{proof} Since $R$ is a discrete valuation ring, one can check that $||F\, |\, \gamma||\leq ||F||$ for any $F\in V_g(K)$ and $\gamma\in M_2(\Z)$. Hence if $\gamma=\begin{psmallmatrix} 1 & a\\ 0&m\end{psmallmatrix}$ then
$$
||\big(\xi \mid \gamma \big)(\{\infty\}-\{0\})||=||\xi(\{\infty\}-\{a/m\})\mid \gamma|| \leq ||\xi(\{\infty\}-\{a/m\})||\leq 1.
$$
\end{proof}

%%%%%%%%%%%%%%%%%%%%%%
\subsubsection{Mazur-Tate elements of modular forms}\label{section_MT} Let $f=\sum a_nq^n\in S_k(N,\e)$ denote a cuspidal newform of weight $k\geq 2$, level $\Gamma_1(N)$, and nebentype $\e$. Associated to $f$
 %Associated to a cuspform $f$ of weight $k$ and level $\Gamma$
 is the map $\xi_f:\Delta_0\rightarrow V_{k-2}(\C)$ defined by
\begin{align*}
\xi_f(\{r\}-\{s\})=2\pi i \int_s^r f(z)(zX+Y)^{k-2}dz,
\end{align*}
where the integral is over the geodesic in the complex upper half-plane connecting $r$ and $s$. One can check  that for any $\gamma\in \GL_2^+(\Q)$ we have 
$
\xi_f| \gamma=\xi_{f|_k \gamma},
$
where $(f|_k \gamma)(z)=\det\gamma^{k-1}(cz+d)^{-k}f(\gamma z)$ is the usual slash operator on modular forms. Thus $\xi_f$ inherits $\Gamma_1(N)$-invariance from $f$ and defines a modular symbol on $V_{k-2}(\C)$.

Let  $K_f/\Q$ denote the number field generated by the Fourier coefficients of $f$. Since $\Gamma_1(N)$ is normalized by the matrix $\iota=\begin{psmallmatrix} -1 &0\\ 0&1\end{psmallmatrix}$, we can write $\xi_f=\xi_f^++\xi_f^-$, where  $\xi^\pm=\frac{\xi\pm\xi\mid\iota}{2}$ lie in the $\pm$-eigenspace for $\iota$. It is well-known (see \cite[Proposition 5.11]{PasolPopa}, for example) that there exist complex periods $\Omega^\pm_f\in \C^{\times}$ such that the modular symbols
%\begin{equation}\label{varpm}
$
\varphi_f^\pm:=\xi_f^\pm/\Omega_f^\pm
$
%\end{equation}
are defined over $K_f$, and we therefore define
$$
\vartheta_{m,j}^\pm(f):=\vartheta_{m,j}(\varphi_f^\pm)\in K_f[\Gg_m].
$$
From the definitions, note that the coefficients of $\vartheta_{m,j}^\pm(f)$, which we denote $[a/m]_{f,j}^\pm:=[a/m]_{\varphi_f^\pm,j}$, can be expressed explicitly in terms of integrals:
\begin{equation}\label{integrals}
\bigg[\frac{a}{m}\bigg]_{f,j}^\pm =\frac{\pi i}{\Omega_f^\pm}\binom{k-2}{j}\bigg( \int_{a/m}^{i\infty}f(z)(mz-a)^jdz\pm (-1)^{k-j} \int_{-a/m}^{i\infty}f(z)(mz+a)^jdz\bigg).
\end{equation}

An interesting property of Mazur-Tate elements is that they interpolate the algebraic parts of the critical $L$-values for $f$ -- see \cite[Proposition 2.3]{PW} for an interpolation formula at $j=0$ and characters of $p$-power conductor. More generally, we have the following: 

%is stated for $j=0$ and characters of $p$-power conductor in \cite[Proposition 2.3]{PW} and more generally we have the following. 
%%%%%
\begin{proposition}\label{Lvaluesprop} For a primitive character $\chi:(\Z/m\Z)^\times\rightarrow \C^\times$ we have
$$
\chi\big(\vartheta^\pm_{m,j}(f)\big)=\begin{cases}\binom{k-2}{j}\frac{j!m^j\tau(\chi)}{\Omega^{\pm}_f(-2\pi i)^j}\cdot  L(f_{ \chi^{-1}},j+1)\in K_f(\chi) &\quad \text{if $\pm=\sgn \chi(-1)(-1)^{k-j},$}\\
0 & \quad\text{otherwise,}\\
\end{cases}
$$
where $K_f(\chi)=K_f(\im\chi)$, $\tau(\chi)=\sum_{a\in (\Z/m\Z)^\times}\chi(a)\zeta_m^a$ is a classical Gauss sum, and $f_{\chi}=\sum \chi(n)a_nq^n$ denotes the twist of $f$ by $\chi$. 
\end{proposition}
\begin{proof} %This follows from \eqref{integrals} and \cite[8.6]{MTT}.
By \cite[Proposition 3]{MTT}, the values $\xi_{f,j}(a,m):=2\pi i \binom{k-2}{j}\int_{a/m}^{i\infty}f(z)(mz-a)^jdz,$ depend only on $a\Mod m$, so
$$
\sum_{a\in (\Z/m\Z)^\times}\chi(a)\xi_{f,j}(-a,m)=\chi(-1)\sum_{a\in (\Z/m\Z)^\times}\chi(a)\xi_{f,j}(a,m).
$$
Thus
\begin{align*}
\chi\big(\vartheta^\pm_{m,j}(f)\big)&=\bigg(\frac{1\pm\chi(-1)(-1)^{k-j}}{2\Omega_f^\pm}\bigg)\sum_{a\in (\Z/m\Z)^\times} \chi(a)\xi_{f,j}(a,m)\\
&=\begin{cases} \frac{1}{\Omega_f^\pm}\sum_{a\in (\Z/m\Z)^\times} \chi(a)\xi_{f,j}(a,m) &\quad \text{if $\pm=\sgn \chi(-1)(-1)^{k-j},$}\\
0 & \quad\text{otherwise.}
\end{cases}
\end{align*}
The result now follows from \cite[8.6]{MTT}, which states that
$$
\sum_{a\in (\Z/m\Z)^\times} \chi(a)\xi_{f,j}(a,m)=-\binom{k-2}{j}\frac{j!m^j\tau(\chi)}{(-2\pi i)^j}\cdot  L(f_{ \chi^{-1}},j+1).
$$
\end{proof}
We now define the $p$-adic Mazur-Tate elements attached to $f$.  Let $K/\Q_p$ be the completion of $K_f$ at the prime over $p$ determined by our fixed embedding $\overline{\Q}\hookrightarrow \overline{\Q_p}$. 
With respect to this embedding, we may view the $K_f$-valued modular symbols $\varphi_f^\pm$ as being $K$-valued.  The symbols $\varphi_f^\pm$ depend on our choice of periods $\Omega_f^\pm$ and, while $\lambda$-invariants are stable under scaling, the choice of periods can affect the $\mu$-invariants of the Mazur-Tate elements attached to $\varphi_f^\pm$. In order to normalize this choice, we follow \cite{PW} and assume  from now on that $\varphi_f^\pm$ are cohomological symbols (see \S\ref{section_coho}). In other words, we choose periods $\Omega_f^\pm$ so that $||\xi_f^\pm/\Omega_f^\pm||=1$ with respect to $\overline{\Q}\hookrightarrow \overline{\Q_p}$. In particular, by Lemma \ref{coholem}, this choice means that the Mazur-Tate elements $\vartheta_{m,j}^\pm(f)$ have coefficients in the valuation ring $\Oo$ of $K$.

For the remainder of this article, we adopt the following notation: for $n\geq 0$, write 
$$
N = N(p,n)= \begin{cases} n+1 &\text{if $p>2$,}\\
n+2 & \text{if $p=2$}.
\end{cases}
$$ 
(The distinction between $N$ as the level of a modular form and its meaning above will always be clear from the context.)
There is a canonical decomposition 
$$
\Gg_{p^N}=\Delta\times\Gamma_n, \qquad 
$$
 where $\Delta$ is cyclic of order $p-1$ (resp., order 2) for $p$ odd (resp., $p=2$) and $\Gamma_n$ is cyclic of order $p^n$. For any character $\psi\in \hat \Delta=\Hom(\Delta,\Z_p^\times)$, we therefore have an induced map $\psi:\Oo[\Gg_{p^N}]\rightarrow \Oo[\Gamma_n]$ obtained by writing $\sigma_a\in \Gg_{p^N}$ as $\sigma_a=\delta(a)\gamma_n(a)$, where $\delta(a)\in \Delta$ and $\gamma_n(a)\in \Gamma_n$, and defining $\psi( \sigma_a)=\psi(\delta(a))\gamma_n(a)$. 
 
 \begin{definition} Let $\psi\in \hat \Delta$ and define
 $$
\theta_{n,j}^\psi(f)=\psi(\vartheta_{p^N,j}^{*}(f))=\sum_{a\in (\Z/p^N\Z)^\times} \psi\big(\delta(a)\big)\bigg[\frac{a}{p^N}\bigg]^{*}_{f,j}\cdot\gamma_n(a)\in \Oo[G_n],
$$
where $*=\sgn\psi(-1)(-1)^{k-j}$. When $\psi$ is trivial and $j=0$ we simply write $ \theta_n(f)$. 
 \end{definition}
 
\subsection{$p$-adic $L$-functions}\label{section_Lpdef}
Here we recall relevant results on Pollack's signed $p$-adic $L$-functions that will be used in the next section. The reader is referred to \cite[\S\S2-5]{pollack03} for details. Let $\alpha$ be an admissible (i.e., $\ord_p(\alpha)<k-1$) root of $X^2-a_pX+\e(p)p^{k-1}$ and let
\begin{equation}\label{def_Lp}
L_p(f,\alpha,\chi) = \int_{\Z_p^\times}\chi d\mu_{f,\alpha}^{\pm}
\end{equation}
denote the $p$-adic $L$-function of $f$ at $\alpha$, defined as in \cite{pollack03} with respect to our fixed choice of cohomological periods $\Omega_f^\pm$ and subject to the following remark.

\begin{remark}\label{compatibility} \nf Using the notation of \cite{pollack03}, we make two small modifications in the definition of the $p$-adic $L$-function: %(see \S2.2 of \emph{op. cit.}):
\begin{itemize}
\item (compare with \cite[\S 2.2]{pollack03}) We include the factor $(-1)^{k-j}$ in the modular symbols 
$$
\lambda^\pm(f,z^j;a,m) =\frac{\pi i}{\Omega_f^\pm}\bigg(\int_{i\infty}^{a/m}f(z)(mz-a)^j\, dz\pm (-1)^{k-j} \int_{i\infty}^{-a/m}f(z)(mz+a)^j\, dz\bigg)
%\frac{\eta(f,P;a,m)\pm (-1)^{k-\deg P}\eta(f,P;a,-m)}{2\Omega_f^\pm}.
$$ 
\item  (compare with \cite[Definition 2.9]{pollack03}) We include the weight of $f$ in the sign of the $p$-adic distribution $\mu_{f,\alpha}^\pm$ above and define  $\pm:=\sgn \chi(-1)(-1)^k$ in equation \eqref{def_Lp}.
\end{itemize}
\end{remark}

These modifications are needed in order to guarantee the following compatibility relation between the modular symbols $\lambda^\pm(f,z^j;a,m)$ of \cite{pollack03,MTT} and the symbols of \cite{PW}. In particular, up to a choice of periods, we have
\begin{equation}\label{kjl}
\binom{k-2}{j}\lambda^\pm(f,z^j;a,p^n)=\bigg[\frac{a}{m}\bigg]_{f,j}^\pm.
\end{equation}

%Let $\omega: \Delta\rightarrow \Z_p^\times$ denote the Teichmuller character and let $\langle \cdot \rangle:\Z_p^\times\rightarrow 1+2p\Z_p$ denote the projection $x\mapsto x\omega^{-1}(x)$.
 Any continuous character  $\chi: \Z_p^\times \rightarrow \C_p^\times$ has a unique decomposition 
$
\chi=\psi\chi_u,
$ for some $\psi\in\hat \Delta$ and $u\in \C_p$ with $|u-1|_p<1$, where $\chi_u$ is the character defined by
$$
\Z_p^\times\xrightarrow{x\mapsto \langle x\rangle } 1+2p\Z_p\xrightarrow{\gamma\mapsto u} \C_p^\times.
$$
(Note $\chi_u(x)=u^{\log_\gamma(x)}$ for all $x\in \Z_p^\times$.) The characters $\psi$ and $\chi_u$ are called the tame and wild parts of $\chi$, respectively. Define
$$
L_p(f,\alpha,\psi,T):= L_p(f,\alpha,\psi\chi_u)\in K(\alpha)[[T]], \qquad T=u-1.
$$
When $f$ is ordinary at $p$, there
is a unique admissible root $\alpha\in \Oo^\times$ and we define
$$
L_p(f,\psi,T):=L_p(f,\alpha,\psi,T)\in \Lambda_K.
$$

\subsubsection{Signed $p$-adic $L$-functions}\label{section_signedLp} We now review Pollack's plus/minus $p$-adic $L$-functions. Define
\begin{align*}
G_\psi^+(f,T) &:=\frac{L_p(f,\alpha_1,\psi,T)+L_p(f,\alpha_2,\psi,T)}{2}\in K[[T]],\\
G_\psi^-(f,T) &:=\frac{L_p(f,\alpha_1,\psi,T)-L_p(f,\alpha_2,\psi,T)}{2\alpha_1}\in K[[T]],
\end{align*}
where  $\alpha_1,\alpha_2$ are the roots of $X^2-a_pX+\e(p)p^{k-1}$.
Recalling our convention that $\star=+$ for odd $p$ and $\star=-$ otherwise, define
\begin{equation}\label{def_plus/minus}
L_{p,\psi}^\pm(f,T) =\frac{G_\psi^\pm(f,T)}{\log_{k}^{\star\pm}(T)},
\end{equation}
where 
\begin{align*}
\log_{k}^\pm(T)&=p^{-(k-1)}\prod_{j=0}^{k-2}\prod_{\substack{n=1\\ \e_n=\pm}}^\infty\frac{\Phi_{p^n}\big(\gamma^{-j}(1+T)\big)}{p}\in \Q_p[[T]].
\end{align*}
When $\psi=1$ we simply write $L_{p}^\pm(f,T)$. 

%%%%%%%%%%%%
\begin{theorem} \label{theorem_pollackdecomp}
If $a_p(f)=0$ then $L_{p,\psi}^\pm(f)\in \Lambda_K$ and 
$$
L_p(f,\alpha,\psi,T)=\log_{k}^{\star}(T)L_{p,\psi}^+(f,T)+\alpha \log_{k}^{-\star} L_{p,\psi}^-(f,T). 
$$
%\item  $G_\psi^\pm(f,\gamma^j\zeta_{p^m}-1)=0$ for all $0\leq j\leq k-2$ and $m\geq 1$ with $\e_m=\star\pm$ (see \cite[pg. 541]{pollack03}) 
\end{theorem}
\begin{proof} This is \cite[Theorem 5.1]{pollack03}.
\end{proof}

%%%%%%%%%%%%%%%%%%%%%%%%%%%%%%%%%%%%%%%%%%%%%%%%%%%%%%%%%%%%%%%%%%%%%%%%%%%%%%%%%%%%%%%%%%%%%%%%%%%%%%%%%%%%
\section{Lifts of Mazur-Tate elements}\label{liftssection}

\noindent \it  Notation. \nf For the remainder of this article, we write $\Phi_m(T)=\sum_{i=0}^{p-1} T^{ip^{m-1}}$ for the $p^m$th cyclotomic polynomial and $\zeta_m$ for a primitive $p^m$th root of unity. 

\smallskip

The goal of this section is to prove Theorem  \ref{lifts1} and an effective version of Theorem \ref{main1} at small weights.
Let $f$ be as in \S\ref{section_MT}, let $j\in \{0,\dots, k-2\}$, and fix $\psi\in \hat\Delta$.
Throughout \S\ref{liftssection}, we identify $\theta_{n,j}^\psi(f)$ with its image $\theta_{n,j}^\psi(f,T)$ under the isomorphism $\Oo[\Gamma_n] \rightarrow \Lambda_n$, $\gamma_n \mapsto 1+T$, where $\gamma_n$ are compatible generators of $\Gamma_n$. 
Explicitly, 
 \begin{align*}
\theta_{n,j}^\psi(f,T)=\sum_{a \in (\Z/p^N\Z)^\times}\psi(a) \bigg[\frac{a}{p^N}\bigg]_{f,j}^{\sgn\psi(-1)(-1)^{k-j}} \cdot (1+T)^{\log_{\gamma}(a)},
 \end{align*}
 where we recall $\gamma$ is a fixed topological generator of $1+2p\Z_p$, $\log_\gamma=\frac{\log_p}{\log_p(\gamma)}$, and  $\log_\gamma(a):=\log_\gamma(\alpha)\Mod p^{n}$ for any $\alpha\in \Z_p^\times$ with $\alpha\equiv a\Mod p^N$ %lifting $a \in (\Z/p^N\Z)^\times$. 

%%%%%%%%%%%%%%%%%%%%%%%%%%%%
\subsection{Polynomial approximation}\label{pola}
Fix $n\geq 0$ and let $\alpha_1$, $\alpha_2\in \overline{\Q_p}$ be roots of the polynomial $X^2-a_pX+\e(p)p^{k-1}$ (the results of \S\ref{pola} hold for $a_p\neq 0$). Define the polynomial
$$
L_{n,j}(f,\alpha_i,\psi,T)=\binom{k-2}{j}\sum_{a\in (\Z/p^{N}\Z)^\times} \psi\omega^{-j}(a)\mu^{*}_{f,\alpha_i}(x^j;a+p^{N}\Z_p)(1+T)^{\log_\gamma(a)}\in K(\alpha_i)[T],
$$
where  $*=\sgn\psi(-1)(-1)^k$  %$\omega:(\Z/2p\Z)^\times \rightarrow \Z_p^\times$ is the mod $p$ cyclotomic character,
 and $\mu^{\pm}_{f,\alpha_i}$ is the $p$-adic distribution defined in \cite{pollack03} (subject to Remark \ref{compatibility}), and  $\omega:(\Z/2p\Z)^\times\rightarrow \Z_p^\times$ is the mod $p$ cyclotomic character.

%%%%%%%%
\begin{proposition}\label{LpToLn} If $0\leq m\leq n$ then 
$$
\textstyle \binom{k-2}{j}L_p(f,\alpha_i,\psi,\gamma^j\zeta_m-1)=L_{n,j}(f,\alpha_i,\psi,\zeta_m-1).
$$
\end{proposition}
\begin{proof}  
We assume $p>2$; the proof for $p=2$ is similar. By definition, the wild character $\chi_\gamma$ satisfies
\begin{equation}\label{decompw}
\chi_\gamma(x)=\gamma^{\log_{\gamma}(x)}=x\omega^{-1}(x),
\end{equation}
hence $\chi_{\gamma^j\zeta_m}=\omega^{-j}\chi_{\zeta_m}x^j$. Noting that $\psi\omega^{-j}\chi_{\zeta_m}$ is constant on $1+p^{m+1}\Z_p=\gamma^{p^m\Z_p}$, we have
\begin{align*}
L_p(f,\alpha,\psi,\gamma^j\zeta_m-1) &= \int_{\Z_p^\times}\psi\omega^{-j}\chi_{\zeta_m}x^j d\mu_{f,\alpha}^{*}\\
%L_p(f,\alpha,\psi\omega^{-j}\chi_{\zeta_m}x^j) \\
&= \sum_{a\in (\Z/p^{m+1}\Z)^\times}\psi\omega^{-j}\chi_{\zeta_m}(a)\mu_{f,\alpha}^{*}(x^j;a+p^{m+1}\Z_p)\\
&=\sum_{a\in (\Z/p^{m+1}\Z)^\times}\psi\omega^{-j}(a)\mu_{f,\alpha}^{*}(x^j;a+p^{m+1}\Z_p)\zeta_m^{\log_\gamma(a)}\\
&= \textstyle \binom{k-2}{j}^{-1}L_{m,j}(f,\alpha,\psi,\zeta_m-1),
\end{align*}
where  $*=\sgn\psi(-1)(-1)^k$ and $\chi_{\zeta_m}(a)$ denotes the value of $\chi_{\zeta_m}$ on the coset in $\Z_p^\times/(1+p^{m+1}\Z_p)\cong(\Z/p^{m+1}\Z)^\times$ corresponding to $a$. It now follows from the additivity of $\mu_{f,\alpha}^\pm$  \cite[Prop. 2.8]{pollack03} and Lemma \ref{additive} that $L_{m,j}(f,\alpha,\psi,\zeta_m-1)=L_{n,j}(f,\alpha, \psi,\zeta_m-1)$ for all $0\leq m\leq n$.
\end{proof}

%%%%%%%%%%%%%%%%%%%%%%%%%%%%
\subsection{Relation to Mazur-Tate elements}
For the remainder of \S\ref{liftssection} we assume that $a_p=0$. Thus $\alpha := \alpha_1=\sqrt{-\e(p)p^{k-1}}$ and $\alpha_2=-\alpha$. 
%This symmetry allows us to compute $G_{n,j}^\pm(f,\psi)$ explicitly. 
Define
\begin{align*}
G_{n,j}^+(f,\psi,T)&=\frac{L_{n,j}(f,\alpha_1,\psi,T)+L_{n,j}(f,\alpha_2,\psi,T)}{2}\in K[T],\\
G_{n,j}^-(f,\psi,T)&=\frac{L_{n,j}(f,\alpha_1,\psi,T)-L_{n,j}(f,\alpha_2,\psi,T)}{2\alpha_1}\in K[T]. 
\end{align*}
These functions can be thought of as polynomial analogues of the power series $G_\psi^\pm(f,T)$ from \S\ref{section_signedLp}. Note the decomposition
$$
L_{n,j}(f,\alpha_i,\psi,T)=G_{n,j}^+(f,\psi,T)+\alpha_i G_{n,j}^-(f,\psi,T).
$$
Recall our convention that $N=n+1$ for odd $p$ and $N=n+2$ if $p=2$. 

%%%%%%%%
\begin{proposition} Let $*=\sgn\psi(-1)(-1)^k$. Then 
$$
G_{n,j}^+(f,\psi,T)=\begin{cases} \displaystyle\frac{1}{\alpha^{N}} \sum_{a\in (\Z/p^{N}\Z)^\times}\psi\omega^{-j}(a)\bigg[\frac{a}{p^{N}}\bigg]_{f,j}^*(1+T)^{\log_\gamma(a)} &\text{$N$ even,} \\
 \displaystyle\frac{-\e(p)p^{k-2}}{\alpha^{N+1}} \sum_{a\in (\Z/p^{N}\Z)^\times}\psi\omega^{-j}(a)\bigg[\frac{a}{p^{N-1}}\bigg]_{f,j}^*(1+T)^{\log_\gamma(a)} & \text{$N$ odd,}
\end{cases}
$$
$$
G_{n,j}^-(f,\psi,T)=\begin{cases} \displaystyle\frac{-\e(p)p^{k-2}}{\alpha^{N+2}} \sum_{a\in (\Z/p^{N}\Z)^\times}\psi\omega^{-j}(a)\bigg[\frac{a}{p^{N-1}}\bigg]_{f,j}^*(1+T)^{\log_\gamma(a)} & \text{$N$ even,}\\
\displaystyle\frac{1}{\alpha^{N+1}} \sum_{a\in (\Z/p^{N}\Z)^\times}\psi\omega^{-j}(a)\bigg[\frac{a}{p^{N}}\bigg]_{f,j}^*(1+T)^{\log_\gamma(a)} &\text{$N$ odd.}
\end{cases}
$$
%In particular, $G_{n,j}^{\pm}(f,\psi,T)\in K[T]$. 
\end{proposition}

\begin{proof} We argue for odd $p$ (the case $p=2$ is similar). From \eqref{kjl} and the definition of $\mu_{f,\alpha_i}^\pm$ (see \cite[\S2.2]{pollack03}), we have
$$
 \binom{k-2}{j}\mu_{f,\alpha}^\pm(x^j, a+p^{n+1}\Z_p)=\frac{1}{\alpha^{n+1}}\bigg(\bigg[\frac{a}{p^{n+1}}\bigg]_{f,j}^\pm-\frac{\e(p)p^{k-2}}{\alpha}\bigg[\frac{a}{p^{n}}\bigg]_{f,j}^\pm\bigg).
$$
Since $\alpha_1=-\alpha_2$ this implies
\begin{align*}
\textstyle \binom{k-2}{j}\displaystyle\big(\mu_{f,\alpha_1}^\pm(x^j;a+p^{n+1}\Z_p)+\mu^{\pm}_{f,\alpha_2}(x^j;a+p^{n+1}\Z_p)\big)&=
\begin{cases}  \frac{2}{\alpha^{n+1}} \big[\frac{a}{p^{n+1}}\big]_{f,j}^\pm& \text{$n$ odd,}\\
 \frac{-2\e(p)p^{k-2}}{\alpha^{n+2}}  \big[\frac{a}{p^{n}}\big]_{f,j}^\pm&\text{$n$ even,}
\end{cases}\\
\textstyle \binom{k-2}{j}\displaystyle\big(\mu_{f,\alpha_1}^\pm(x^j;a+p^{n+1}\Z_p)-\mu^{\pm}_{f,\alpha_2}(x^j;a+p^{n+1}\Z_p)\big)&=
\begin{cases}  \frac{-2\e(p)p^{k-2}}{\alpha^{n+2}}  \big[\frac{a}{p^{n}}\big]_{f,j}^\pm& \text{$n$ odd,}\\
\frac{2}{\alpha^{n+1}} \big[\frac{a}{p^{n+1}}\big]_{f,j}^\pm    &\text{$n$ even.}
\end{cases}
\end{align*}
The desired equalities now follow from the definition of $G_{n,j}^\pm(f,\psi,T)$. %The final claim is because $[a/p^n]_{f,j}^\pm\in \Oo$ and $\alpha=\sqrt{-\e(p)p^{k-1}}$ (thus $\alpha^{n}\in K$ for $n$ even and $\alpha^{n+1}\in K$ for $n$ odd). 
\end{proof}

%%%%%%%%
\begin{corollary}\label{MT} $\theta_{n,j}^\psi(f,T)=\begin{cases} \alpha^{N}G_{n,j}^+(f,\psi\omega^j,T)&\text{$N$ even,} \\
\alpha^{N+1}G_{n,j}^-(f,\psi\omega^j,T)&\text{$N$ odd.}
\end{cases}$
\end{corollary}

\begin{proof} This follows from the previous proposition and the definitions.
\end{proof}

We now determine explicit zeros of the polynomials $G_{n,j}^\pm(f,\psi,T)$ by relating them to the power series $G_\psi^\pm(f,T)$ of \S\ref{section_Lpdef}.

%%%%%%%%
\begin{proposition}\label{Gnroots} 
For all $0\leq m \leq n$ we have
$$
\textstyle \binom{k-2}{j}\displaystyle G_\psi^\pm(f,\gamma^j\zeta_m-1)=G_{n,j}^\pm(f,\psi,\zeta_m-1).
$$
 In particular, $G_{n,j}^\pm(f,\psi,0)=\textstyle \binom{k-2}{j}G_\psi^\pm(f,\gamma^j-1)$ and $G_{n,j}^{\pm}(f,\psi,\zeta_m-1)=0$ for all $1\leq m\leq n$ with $\e_m=\star\pm$.
\end{proposition}
\begin{proof} Using the definitions and Proposition \ref{LpToLn} we have 
\begin{align*}
\textstyle \binom{k-2}{j}\displaystyle G_\psi^+(f,\gamma^j\zeta_m-1)&=\binom{k-2}{j}\bigg( \frac{L_p(f,\alpha_1,\psi,\gamma^j\zeta_m-1)+L_p(f,\alpha_2,\psi,\gamma^j\zeta_m-1)}{2}\bigg)\\
&=\frac{L_{n,j}(f,\alpha_1,\psi,\zeta_m-1)+L_{n,j}(f,\alpha_2,\psi,\zeta_m-1)}{2}\\
&=G_{n,j}^+(f,\psi,\zeta_m-1),
\end{align*}
with an analogous computation for $G_\psi^-$. 
For the final statement, note that Theorem \ref{theorem_pollackdecomp} and \eqref{def_plus/minus} imply $G_\psi^\pm(f,T)\in \log_{k}^{\star \pm}(T)\Lambda_K$. Since $\log_{k}^{\star \pm}(T)$ has zeroes at $T=\gamma^j\zeta_{m}-1$ for all $0\leq j\leq k-2$ and $m\geq 1$ with $\e_m=\star\pm$, the result now follows. 
\end{proof}

%%%%%%%%%%%%%%
\subsection{Construction of lifts}\label{section_lifts}
Define the polynomials
$$
\omega_{n,j}^\pm =\displaystyle\prod_{\substack{i=1\\ \e_i=\pm}}^n\Phi_i(\gamma^{-j}(1+T)),\qquad \log_{k,n}^\pm = \prod_{j=0}^{k-2}\omega_{n,j}^{\pm}\in \Z_p[T]. 
$$
Notice that we can recover the half-logarithms $\log_{k}^\pm$ of \S\ref{section_signedLp} as limits of $\log_{k,n}^\pm$: as functions converging on the open unit disc we have 
$$
\lim_{n\rightarrow \infty} p^{-(k-1)(1+\delta_n^\pm)}\log_{k,n}^\pm=\log_k^\pm,
$$
where $\delta_n^+=\lfloor n/2\rfloor$ and $\delta_n^-=\lfloor (n+1)/2\rfloor$.  We call the functions $\log_{k,n}^\pm$ \emph{half-logarithms at layer $n$}. To analyze the tail-end of $\log_k^\pm$, we also define
$$
u_{n,j}^\pm(T)= \prod_{\substack{i=n+1\\ \e_i=\pm}}^\infty\frac{\Phi_{i}\big(\gamma^{-j}(1+T)\big)}{p},\qquad v_{k,n}^\pm(T) =\prod_{j=0}^{k-2}u_{n,j}^\pm(T)\in \Q_p[[T]].
$$
and note that
\begin{equation}\label{logpm}
 \log_{k}^\pm =p^{-(k-1)(1+\delta_n^\pm)} \log_{k,n}^\pm  v_{k,n}^\pm.
 \end{equation}
In particular, $\log_{k,n}^\pm$ divides $\log_{k}^\pm$ in $\Q_p[[T]]$ and since the latter converge on the open $\C_p$-unit disc (see \cite[Corollary 4.2]{pollack03}), the same is true of $v_{k,n}^\pm$. %must also converge on this disc. 

%%%%%%%%
\begin{lemma}\label{units} For any $1\leq m\leq n$ and $j\in \Z$ we have 
$$
v^\pm_{k,n}(\gamma^j\zeta_m-1)=v^\pm_{k,n}(\gamma^{j}-1)\in \Z_p^\times.
$$
\end{lemma}
\begin{proof} The equality follows directly from the fact that $\Phi_i(x\zeta_m)=\Phi_i(x)$ for all $i\geq m+1$. To show these values are units, note that  for any $x\in 1+2p\Z_p$ and $i\geq 1$ we have 
$\ord_p \Phi_i(x)=1$ by Lemma \ref{cyclo-j}. Thus $u_{n,j'}^\pm(\gamma^j-1)$ are units for any $j,j'\in \Z$, and hence $v_{k,n}^\pm(\gamma^j-1)$ must also be units.
\end{proof}

We are now ready to prove Theorem \ref{lifts1} (which is obtained by taking $p>2$ and $j=0$ in the following theorem). 
 Recall the map
$\eta^j:\Lambda_K\rightarrow\Lambda_K$ of \S\ref{section_iwinv}, defined by $F(T)\mapsto F\big(\gamma^j(1+T)-1\big).$

%%%%%%%%%
\begin{theorem} \label{lifts} If $a_p=0$ then there are explicit units $u_{n,j}\in \Z_p[\e(p)]^\times$ such that for $n\geq 0$, 
$$
\theta_{n,j}^\psi \equiv u_{n,j}C^{\e_n}\textstyle\binom{k-2}{j}\displaystyle \eta^{j}\big(L_{p,{\psi\omega^{j}}}^{\star \e_{n+1}}\log_{k,n}^{\e_{n+1}}\big) \Mod \omega_n
$$
where
\begin{equation}\label{Cconst}
C^\pm=\begin{cases}2^{k-1}&\text{if $p=2$ and $\pm=-$,}\\
1&\text{otherwise.} 
\end{cases}
\end{equation}
\end{theorem}

%In order to see explicitly where the factor of $2^{k-1}$ appears at $p=2$, we sacrifice some notational simplicity in the proof below in order to include $p=2$ simultaneously.

\begin{proof} 
Define the units $u_{n,j} \in \Z_p[\e(p)]^\times$ by
$$
u_{n,j}=\begin{cases}(-\e(p))^{\frac{N}{2}}v^{\star}_{k,n}(\gamma^{j}-1)&\quad\text{$N$ even,} \\
(-\e(p))^{\frac{N+1}{2}}v^{-\star}_{k,n}(\gamma^{j}-1)&\quad\text{$N$ odd.}
\end{cases}
$$
(We know these are units by Lemma \ref{units}.) The roots of $\omega_n^\pm :=\omega_{n,j=0}^\pm$ are $\zeta_m-1$ for $1\leq m\leq n$ with $\e_m=\pm$, hence Proposition \ref{Gnroots} implies that we can find $H_{n,j}^\pm(\psi,T)\in K[T]$ such that for $*\in \{+,-\}$, 
$$
G_{n,j}^{*}(\psi,T)=\omega_n^{\star*}(T) H_{n,j}^{*}(\psi,T).
$$
Thus,
\begin{equation}\label{MTH}
\theta_{n,j}^{\psi}(f,T)=\begin{cases} \alpha^{N}\omega_n^{\star}(T)H_{n,j}^{+}(\psi\omega^j,T)&\text{$N$ even,} \\
\alpha^{N+1}\omega_n^{-\star}(T)H_{n,j}^{-}(\psi\omega^j,T)&\text{$N$ odd,}
\end{cases}
\end{equation}
by Corollary \ref{MT}. Now note that $\eta^j(\omega_{n,j}^\pm)=\omega_{n}^\pm$,  so
\begin{equation}\label{eta1}
\eta^j(\log_{k,n}^\pm)=\omega_n^\pm\cdot \eta^j\bigg(\prod_{\substack{j'=0\\ j'\neq j}}^{k-2}\omega_{n,j'}^\pm\bigg).
\end{equation}
 Since $\omega_n=(1+T)^{p^n}-1=T\omega_n^+\omega_n^-$, the result will follow from equations \eqref{MTH} and \eqref{eta1} if we can show the following for any tame character $\psi$:
 \begin{align}
\alpha^{N}H_{n,j}^{+}(\psi,T) &\equiv  u_{n,j}\textstyle \binom{k-2}{j}\displaystyle \eta^j\bigg(L_{p,\psi}^{+}\prod_{\substack{j'=0\\ j'\neq j}}^{k-2}\omega_{n,j'}^{\star}\bigg)\Mod T\omega_n^{-\star}, &\text{(for $N$ even)} \label{H+} \\ 
\alpha^{N+1}H_{n,j}^{-}(\psi,T)&\equiv u_{n,j}C^{-} \textstyle \binom{k-2}{j}\displaystyle\eta^j\bigg(L_{p,\psi}^{-}\prod_{\substack{j'=0\\ j'\neq j}}^{k-2}\omega_{n,j'}^{-\star}\bigg)\Mod T\omega_n^{\star}, &\text{(for $N$ odd)}\label{H-}.
\end{align}

To prove these congruences, we need only check that the left and right sides agree at all roots of the modulus $T\omega_n^\pm$, i.e., at the values $T=0$ and $T=\zeta_m-1$ for all $1\leq m\leq n$ with $\e_m=\pm$. We prove \eqref{H-}; the proof of \eqref{H+} is similar. Suppose $N$ is odd (thus $n$ is even if $p>2$, or $n$ is odd if $p=2$). Let $m=0$ or take $1\leq m\leq n$ with $\e_m=\star$. Then
\begin{align*}
\alpha^{N+1}\textstyle \binom{k-2}{j}\displaystyle^{-1}H_{n,j}^-(\psi,\zeta_m-1)&= \alpha^{N+1}\textstyle \binom{k-2}{j}\displaystyle^{-1}\omega_{n}^{-\star}(\zeta_m-1)^{-1}G_{n,j}^-(\psi, \zeta_m-1)  \\
&= \alpha^{N+1}\omega_{n}^{-\star}(\zeta_m-1)^{-1}G_{\psi}^-(\gamma^j\zeta_m-1)\\
&= \alpha^{N+1}\omega_{n}^{-\star}(\zeta_m-1)^{-1}L_{p,\psi}^-(\gamma^j\zeta_m-1) \log_{k}^{-\star}(\gamma^j\zeta_m-1)\\
&= \alpha^{N+1}\omega_{n}^{-\star}(\zeta_m-1)^{-1}\eta^{j}\big(L_{p,\psi}^- \log_{k}^{-\star}\big)(\zeta_m-1) \\
&= \alpha^{N+1}p^{-(k-1)(1+\delta_n^{-\star})}\omega_{n}^{-\star}(\zeta_m-1)^{-1}\eta^{j}\bigg(L_{p,\psi}^-\log_{k,n}^{-\star}v_{k,n}^{-\star}\bigg)(\zeta_m-1)  \\
&= (-\e(p))^{\frac{N+1}{2}}C^{-}v_{k,n}^{-\star}(\gamma^j\zeta_m-1)\eta^{j}\bigg(L_{p,\psi}^-\textstyle\prod_{j'=0,j'\neq j}^{k-2}\omega_{n,j'}^{-\star}\bigg)(\zeta_m-1)\\
&=u_{n,j}C^{-}\eta^{j}\bigg(L_{p,\psi}^-\textstyle\prod_{j'=0,j'\neq j}^{k-2}\omega_{n,j'}^{-\star}\bigg)(\zeta_m-1)
\end{align*}
where the first and third equalities are by definition, 
the second is from Proposition \ref{Gnroots},
the fourth is from \eqref{eta},
the fifth is from \eqref{logpm},
the sixth is because $\alpha=\sqrt{-\e(p)p^{k-1}}$, 
and the final equation follows from Lemma \ref{units} and the definition of $u_{n,j}$. 
\end{proof}

\subsection{Proof of Theorem \ref{main1} at small weights}
Recall the sequence $q_n$ defined by 
$$
q_n=\begin{cases}p^{n-1}-p^{n-2}+\cdots +p-1 & \quad\text{if $n\geq 2$ is even,}\\
p^{n-1}-p^{n-2}+\cdots +p^2-p & \quad\text{if $n\geq 3$ is odd}
\end{cases}
$$
and $q_0=q_1=0$.

%%%%%%%
\begin{lemma}\label{omeganinv} For all $n\geq 1$ we have $\mu(\log_{k,n}^{\pm})=0$ and $\lambda(\log_{k,n}^{\e_{n+1}})=(k-1)q_n$. 
\end{lemma} 
\begin{proof} Since $\eta^i\big(\Phi_n(\gamma^{-i}(1+T))\big)=\Phi_n(1+T)$, Lemma \ref{etainv} implies that the Iwasawa invariants of $\Phi_n(\gamma^{-i}(1+T))$ agree with those of $\Phi_n(1+T)$. But $\Phi_{n}(1+T)$ has $\lambda$-invariant $p^n-p^{n-1}$ for $n\geq 1$ and trivial $\mu$-invariant, so $\lambda(\omega_{n,i}^{\e_{n+1}})=q_n$ and $\mu(\omega_{n,i}^{\e_{n+1}})=0$. The result follows. 
\end{proof}

 Let $\lambda^\pm(f,\psi)$ and $\mu^\pm(f,\psi)$ denote the Iwasawa invariants of $L_{p,\psi}^\pm(f)$. 
We can now prove Theorem \ref{main1} when $\leq p+1$ and when $k=p+2$ under additional assumptions on $\lambda^\pm(f,\psi)$. We refer the reader to Appendix \ref{tablessection} for examples illustrating this theorem.

\begin{theorem}\label{theorem_smallweights} 
Let $f\in S_k(\Gamma_1(N))$ be a  newform with $a_p(f)=0$. Let $\psi\in \hat\Delta$, $j\in \{0,\dots, k-2\}$, and let $C^\pm$ be as in Theorem \ref{lifts}. 
\begin{enumerate}
\item Suppose that $2\leq k\leq p+1$. Let $n_0^-$ (resp., $n_0^+$) be the smallest positive even (resp., odd) integer for which $\lambda^{-\star}(f,\psi\omega^j)<p^n-(k-1)q_n$ (resp., $\lambda^{+\star}(f,\psi\omega^j)<p^n-(k-1)q_n$). Then for all even $n\geq n_0^-$ (resp., all odd $n\geq n_0^{+}$) we have 
\begin{align*}
\mu\big(\theta_{n,j}^\psi(f)\big)&=\mu^{\star*}(f,\psi\omega^j)+\ord_p C^{*},\\
\lambda\big(\theta_{n,j}^\psi(f)\big) &=\lambda^{\star*}(f,\psi\omega^j)+(k-1)q_n,
\end{align*}
where $*=-$ if $n$ is even and $*=+$ if $n$ is odd. 
\item Suppose that $k=p+2$, $\lambda^{-\star}(f,\psi\omega^j)=0$, and $\lambda^{+\star}(f,\psi\omega^j)<p$. Then for all $n\geq 1 $ we have
\begin{align*}
\mu\big(\theta_{n,j}^\psi(f)\big)&=\mu^{\star*}(f,\psi\omega^j)+\ord_p\bigg(C^{*}\binom{p}{j}\bigg)\\
\lambda\big(\theta_{n,j}^\psi(f)\big)  &=\lambda^{\star*}(f,\psi\omega^j)+(k-1)q_n,
\end{align*}
where $*=-$ if $n$ is even and $*=+$ if $n$ is odd. 
\end{enumerate}
\end{theorem} 
\begin{proof} By Lemmas \ref{lem_basicIw}, \ref{etainv}, and \ref{omeganinv} we know that $C^{\e_n}\binom{k-2}{j}\eta^j(L_{p,\psi\omega^j}^{\star\e_{n+1}}{\log_{k,n}^{\e_{n+1}}})$ has $\lambda$-invariant equal to $(k-1)q_n+\lambda^{\star\e_{n+1}}(f,\psi\omega^j)$ and $\mu$-invariant equal to $\mu^{\star\e_{n+1}}(f,\psi\omega^j)$. For part (1), we have $(k-1)q_n+\lambda^{\star\e_{n+1}}(f,\psi\omega^j)<p^n$ for $n\geq n_0^{\e_{n+1}}$ by assumption, hence the statement follows from combining Lemma \ref{small_lam} and Theorem \ref{lifts}. For part (2), note that for all $n\geq 1$ we have $(k-1)q_n=(p+1)q_n=p^n-1$ or $p^n-p$, depending on whether $n$ is even or odd respectively, and the statement now follows using the same argument as part (1). 
\end{proof}

When $k>p+2$ (or when $k=p+2$ and $\lambda^\pm(f,\psi\omega^j)$ do not satisfy the bounds in part (2) of the theorem), the $\lambda$-invariants of the lifts in Theorem \ref{lifts} are no longer bounded above by $p^n$, and therefore we cannot use Lemma \ref{small_lam} to relate the invariants of the lifts to those of the Mazur-Tate elements as in the proof above. 
We consider this problem in the following section.

%%%%%%%%%%%%%%%%%%%%%%%%%%%%%%%%%%%%%%%%%%%%%%%%%%%%%%%%%%%%%%%%%%%%%%%%%%%%%%%%%%%%%%%%%%%%%%%%%%%%%%%%%%%%
\section{Finite-layer Iwasawa invariants}\label{section_finite-layer}

 The goal of this section is to prove Theorem \ref{mainpl1}  and provide a condition on the Newton polygon that will allow us to check whether a function satisfies the `$p$-large' assumption of this theorem (see Definition \ref{pldef}). The Newton polygon condition will ultimately be used to verify that the functions $L_p^\pm\log_{k,n}^\pm$ are $p$-large at all sufficiently large $n$, a fact which appears difficult to check directly from the definitions. Because some of the proofs in this section are already quite long, we have relegated several lemmas to Appendix \ref{appendix_lemmas}.

In this section we let $K$ denote an arbitrary finite extension of $\Q_p$ with valuation ring $\Oo$, uniformizer $\varpi$, and ramification index $e$.

%%%%%%%%%%%%%%
\subsection{$p$-large functions}

%\subsubsection{$p$-large functions.}\label{plsection1}
Fix $n\geq 1$. For an integer $N\geq 0$, let
 \begin{align*}
 R(N)_n&=N\Mod(p^n-p^{n-1}),\quad \text{and}\\
Q(N)_n&=\bigg\lfloor {\frac{N}{p^n-p^{n-1}}}\bigg\rfloor,
 \end{align*}
denote the remainder and quotient of $N$ upon division by $p^n-p^{n-1}$. Thus, we can uniquely write 
$$
N=R(N)_n+Q(N)_n(p^n-p^{n-1}).
$$

\begin{definition}\label{pldef}\nf  For $F=\sum a_iT^i\in \Lambda_K$ with Iwasawa invariants $\lambda$ and $\mu$, define 
\begin{align*}
R(F)_n& = R(\lambda-p^n)_n,\\
Q(F)_n&= Q(\lambda-p^n)_n,\\
v_i(F,n) &= \mu+ \begin{cases} Q(F)_n+1+\frac{1}{e}& \text{if $0\leq i\leq R(F)_n+p^{n-1},$}\\ 
  Q(F)_n+1& \text{if $R(F)_n+p^{n-1}+1\leq i\leq  p^n-1,$}\\ 
Q(F)_n-Q(i-p^n)_n+\frac{1}{e} & \text{if $p^n \leq i\leq \lambda-1$ and $ R(i-p^n)_n\leq R(F)_n,$}\\
Q(F)_n-Q(i-p^n)_n & \text{if $p^n \leq i\leq \lambda-1$ and $ R(i-p^n)_n>R(F)_n.$}
\end{cases}
\end{align*}
\noindent We say that $F$ is \emph{$p$-large at $n$} if $\lambda \geq p^n$ and $\ord_pa_i\geq v_i(F,n)$ for all $0\leq i\leq \lambda-1$.
\end{definition}

\noindent Since $n$ is fixed, we will usually omit the subscript $n$ and simply write $R(N), Q(N), R(F), Q(F),$ and $v_i(F)$, for the obvious counterparts at $n$. 

%%%%%
\begin{example}\nf  $F=T^{p^n}$ is $p$-large at $n$, while $F=T^{p^n}+p$ is not. 
\end{example}

\begin{definition}\label{def_cN}\nf For each integer $i\leq p^n-1$, define the sequence $(c_{N,n}^{(i)})_{N\geq 0}$ recursively by
%Define the sequence of \emph{$\omega_n$-division terms} $c_{N,n}^{(i)}$ for $N\geq 0$ and $ i\leq p^n-1$ recursively by
\begin{align}
c_{0,n}^{(i)}&=\begin{cases}-\displaystyle\binom{p^n}{i}& \text{if $1\leq i\leq p^n-1,$}\\ 
0 & \text{if $i\leq 0$,}
\end{cases}\\ \label{recursion}
c_{N,n}^{(i)}&=c_{N-1,n}^{(i-1)}+c_{N-1,n}^{(p^n-1)}c_{0,n}^{(i)},\quad \text{for $N\geq1$.}
\end{align}
%(Analogous sequences can  be defined for any distinguished polynomial.) 
\end{definition}
%(The name `division terms' comes from Lemma \ref{explicitproj}.) 
\noindent Since $n$ is fixed, we usually write $c_{N}^{(i)}=c_{N,n}^{(i)}$. Note that the recursion relation implies $c_N^{(i)} = 0$ for any $N$ and $i\leq 0$.

%%%%%%
\begin{proposition}\label{lbounds} For all $N\geq 0$ we  have
\begin{align}\label{lb1}
\ord_p c_N^{(i)}&\begin{cases}\geq Q(N)+2& \text{if $0\leq i\leq R(N)+ p^{n-1}-1$}\\ 
= Q(N)+1 & \text{if $i=R(N)+p^{n-1}$}\\
\geq Q(N)+1& \text{if $R(N)+p^{n-1}+1\leq i\leq p^n-1$. }
\end{cases}
\end{align}

\end{proposition}

\begin{proof} 
 If $N=0$ then the first and middle (in)equality of \eqref{lb1} follow from Lemma \ref{basiclem}.(2), and the last inequality follows from Lemma \ref{basiclem}.(1). Now define for $M\geq 0$ the set
 $$
S_M=\{N\geq 0 \mid Q(N)=M\}.
$$
By induction on $M$, it suffices to establish the following claims.

\smallskip

\noindent \bf Claim 1. \it Fix $M\geq 0$. If \eqref{lb1} holds at $N=M(p^n-p^{n-1})$ then it holds for all $N\in S_M$. 
%also holds for all $N$ with $Q(N)=M$.

\noindent \bf Claim 2. \it Fix $M\geq 0$. If \eqref{lb1} holds for all $N\in S_M$ then it also holds for $N=(M+1)(p^n-p^{n-1})$. \nf

\smallskip

\noindent \it Proof of Claim 1. \nf We induct on $N$. By assumption, \eqref{lb1} is true at $N=M(p^n-p^{n-1})$ so assume it holds for some $N\in S_M$ with $N+1$ also in $S_M$. Since $N,N+1\in S_M$, we have $Q(N)=Q(N+1)=M$  and $R(N+1)=R(N)+1$. Therefore, we must show that
\begin{equation}\label{ms1}
\ord_p c_{N+1}^{(i)}\begin{cases}\geq M+2& \text{if $0\leq i\leq R(N)+ p^{n-1}$}\\ 
= M+1 & \text{if $i=R(N)+p^{n-1}+1$}\\
\geq M+1& \text{if $R(N)+p^{n-1}+2\leq i\leq p^n-1$. }
\end{cases}
\end{equation}
Our inductive hypothesis tells us that 
\begin{equation}\label{one}
\ord_p c_N^{(i-1)}\begin{cases}\geq M+2& \text{if $1\leq i\leq R(N)+ p^{n-1}$}\\ 
= M+1 & \text{if $i=R(N)+p^{n-1}+1$}\\
\geq M+1& \text{if $R(N)+p^{n-1}+2\leq i\leq p^n$,}
\end{cases}
\end{equation}
and combining this with Lemma \ref{basiclem}.(1) yields
\begin{align}\label{two}
\ord_p(c_N^{(p^n-1)}c_0^{(i)})&\geq M+2.
\end{align}
for all $0\leq i\leq p^n-1$. One can now deduce \eqref{ms1} by taking valuations in \eqref{recursion} and using \eqref{one} and \eqref{two}. 

\smallskip

\noindent \it Proof of Claim 2. \nf For each $0\leq j \leq p^n-p^{n-1}-1$, let $N_j\in S_M$ denote the unique integer having $R(N_j)=p^n-p^{n-1}-1-j$. 
By assumption, \eqref{lb1}  holds for all $N_j$. We must show that 
%We claim that it also holds for $N=(M+1)(p^n-p^{n-1})$, i.e., we will show that 
\begin{equation}\label{M+1}
\ord_p c_{(M+1)(p^n-p^{n-1})}^{(i)}\begin{cases}\geq M+3& \text{if $0\leq i\leq p^{n-1}-1,$}\\ 
= M+2 & \text{if $i=p^{n-1},$}\\
\geq M+2& \text{if $p^{n-1}+1\leq i\leq p^n-1$. }
\end{cases}
\end{equation}
Since \eqref{lb1} holds for $N_0$,
\begin{equation}\label{M2}
\ord_p c_{N_0}^{(i-1)} \begin{cases}\geq M+2&\text{if $i\leq p^n-1,$}\\
= M+1&\text{if $i= p^n$.}
\end{cases}
\end{equation}
But we can say more: if $i \leq p^{n-1}$ then $N_0\geq i-1$ so we can apply Lemma \ref{basiclem}.(3) to obtain
\begin{equation}\label{2nd2}
c_{N_0}^{(i-1)}=\sum_{j=0}^{i-2}\underbrace{{c_{N_{j+1}}}^{(p^n-1)}}_\text{$\ord_p\geq M+1$}\underbrace{c_0^{(i-j-1)}}_\text{$\ord_p\geq 2$},
\end{equation}
where the leftmost valuation is because \eqref{lb1} holds for all $N_j$ and the rightmost valuation follows from Lemma \ref{basiclem}.(2). Thus
$
\ord_p c_{N_0}^{(i-1)}\geq M+3
$
for all $i \leq p^{n-1}$. On the other hand, since \eqref{lb1} holds for $N=0$, the equality in \eqref{M2} implies
\begin{equation}\label{2nd}
\ord_p(c_{N_0}^{(p^n-1)}c_0^{(i)})= M+1+\ord_p c_0^{(i)}\begin{cases} \geq M+3&\text{if $0\leq i\leq p^{n-1}-1$}\\
=M+2&\text{if $i= p^{n-1},$}\\
\geq M+2&\text{if $p^{n-1}+1\leq i\leq p^n-1$}.\\
\end{cases}
\end{equation}
By \eqref{recursion}, we have $c_{(M+1)(p^n-p^{n-1})}^{(i)}=c_{N_0}^{(i-1)}+c_{N_0}^{(p^n-1)}c_0^{(i)}$, and  \eqref{M+1} now follows by taking valuations. 
\end{proof}

For the rest of this section, we fix a power series $F=\sum a_i T^i\in \Lambda_K$.

%%%%%%
\begin{lemma}\label{explicitproj} $
\displaystyle F\equiv  \sum_{i=0}^{p^n-1}\bigg(a_i+\sum_{N=0}^\infty a_{p^n+N}c_{N,n}^{(i)}\bigg)T^i\Mod \omega_n. 
$
\end{lemma}
\begin{proof} By induction, we have 
\begin{equation}\label{wndivlem}
T^{p^n+N}\equiv \sum_{i=0}^{p^n-1}c_N^{(i)}T^i\Mod \omega_n.
\end{equation}
for all $N\geq 0$.
By Proposition \ref{lbounds}, the terms $c_N^{(i)}$ have unbounded valuations as $N\rightarrow \infty$, hence the sum $\sum_{N=0}^\infty a_{p^n+N}c_N^{(i)}$ converges. The result now follows by writing
$$
F= \sum_{i=0}^{p^{n}-1}a_iT^i+ \sum_{N=0}^{\infty}a_{p^n+N}T^{p^n+N}, 
$$
reducing modulo $\omega_n$, and applying \eqref{wndivlem} to each term $T^{p^n+N}$. 
\end{proof}

We are now ready to prove Theorem \ref{mainpl1}. 
%%%%%%
\begin{theorem}\label{mainpl} If $F$ is $p$-large at $n\geq 1$ then
\begin{align*}
\mu(\pi_n(F))&=\mu(F)+I(F)_n\\
\lambda(\pi_n(F))&=\lambda(F)-I(F)_n(p^n-p^{n-1})=R(F)_n+p^{n-1}, 
\end{align*}
where $I(F)_n:=Q(F)_n+1$.
In particular, 
 \begin{equation}\label{nrel}
 \lambda(F)+\mu(F)(p^n-p^{n-1})=\lambda(\pi_n(F))+\mu(\pi_n(F))(p^n-p^{n-1}).
\end{equation} 
\end{theorem}

\begin{proof} 
Since $\lambda(F)=p^n+R(F)+Q(F)(p^n-p^{n-1})$, it suffices to prove %We must show that
\begin{align*}
\mu(\pi_n(F)) &=\mu(F)+Q(F)+1, \\
\lambda(\pi_n(F)) &= R(F)+p^{n-1}.
\end{align*}
We may assume that $\mu(F)=0$. From Lemma \ref{explicitproj}, the Iwasawa invariants of $\pi_n(F)$ are equal to those of the polynomial $\sum_{i=0}^{p^n-1}(a_i+s_i)T^i$, where $s_i=\sum_{N\geq 0}a_{p^n+N}c_{N}^{(i)}$. Letting $b_i=a_i+s_i$, Lemma \ref{lemma1} implies
\begin{align*}
\ord_p(b_i)&\geq \min( \ord_p a_i, \ord_p s_i) \begin{cases}>Q(F)+1& \text{if $0\leq i\leq R(F)+ p^{n-1}-1,$}\\ 
\geq Q(F)+1& \text{if $R(F)+p^{n-1}+1\leq i\leq p^n-1$, }
\end{cases}
\end{align*}
and
\begin{align*} 
\ord_p(b_{R(F)+p^{n-1}})&=\min( \ord_p a_{R(F)+p^{n-1}}, Q(F)+1)=Q(F)+1.
\end{align*}
The desired Iwasawa invariants follow. Equation \eqref{nrel} can now be obtained by writing $\lambda(F)=p^n+R(F)+Q(F)(p^n-p^{n-1})$ and using the formulae for the Iwasawa invariants of $\pi_n(F)$ just proven. 
\end{proof}

\begin{example}\nf %We can recompute the Iwasawa invariants of Example  \ref{tpn}: 
Since $T^{p^n}$ is $p$-large at $n$, Theorem \ref{mainpl} yields $\mu(\pi_n(T^{p^n}))=1$ and $\lambda(\pi_n(T^{p^n}))=p^{n-1}$.
(These invariants were also computed directly in Example \ref{tpn}.)
\end{example}

%%%%%%%
\begin{remark}\label{examplefail}\nf Theorem \ref{mainpl} can fail for functions that are not $p$-large at $n$. One extreme example is $F=\omega_n$, where clearly $\lambda(\omega_n)=p^n$ and $\mu(\omega_n)=0$, but since $\pi_n(\omega_n)=0$, both of the $n$th-layer Iwasawa invariants are zero.  
\end{remark}

%%%%%%%%%%%%%%%%%%%%%%%%%%%%%%%%%%%%%%%%%%%%%%%%%%%%%%%%%%%%%%%%%%
\subsection{$p$-largeness via the Newton polygon} \label{plargeNPsection} In this section, we fix an integer $n\geq 1$ and a power series $F=\sum a_iT^i\in\Lambda_K$ with Iwasawa invariants $\lambda$ and $\mu$. Recall the \emph{$p$-adic Newton polygon} of $F,$ denoted $\NP_p(F)$, is the unbounded region in $ \R^2$ determined by the lower convex hull of the points $\{(i,\ord_p a_i) \mid i\geq 0\}$.
 Let 
$$
 \T_{n}(F)=\sum_{i=0}^{n}a_iT^i
 $$
 denote the degree $n$ polynomial truncation of $F$ and, to ease notation,  in this section we use the notation
$$
t_n:=p^n-p^{n-1}.
$$

%%%%%%
\begin{definition}\label{bur} \nf % Let $F\in \Lambda_K$ have Iwasawa invariants $\lambda\geq p^n$ and $\mu$. 
Let $\mathcal{B}_n(F)\subseteq \R^2$ be the region lying on or above the graph of the piecewise linear function $f_F: [0,\lambda-1]\rightarrow \R$ defined as follows. If $Q(F)=0$, define
\begin{equation}\label{fQz}
f_F(x)=\mu+\begin{cases}2-\frac{x}{R(F)+p^{n-1}+1},& \text{$0\leq x\leq R(F) +p^{n-1}+1,$}\\
1,& \text{$R(F)+p^{n-1}+1\leq x\leq \lambda-1$.}
\end{cases}
\end{equation}
If $Q(F)>0$, define
\begin{equation}\label{fQ}
f_F(x)=\mu+\begin{cases}2+Q(F),& \text{$0\leq x\leq R(F) +p^{n-1},$}\\ 
2+Q(F)-\frac{Q(F)+\frac{1}{e}}{Q(F)t_n+1}\big(x-R(F)-p^{n-1}\big),& \text{$R(F) +p^{n-1}\leq x\leq \lambda-t_n+1,$}\\
1,& \text{$\lambda-t_n+1\leq x\leq \lambda-1$.}
\end{cases}
\end{equation}
See the shaded regions in Figure \ref{ffig1} for diagrams of $\cB_n(F)$ when $e=1$ and $\mu=0$.
\end{definition}

%The idea behind the following proposition is illustrated in Figure \ref{ffig1}. 

%%%%%%
\begin{proposition}\label{newton}  If $\lambda\geq p^n$ and $\NP_p(\T_{\lambda-1}(F))\subseteq\mathcal{B}_n(F)$ then $F$ is $p$-large at $n$.
\end{proposition}
\begin{proof} We may assume that $\mu=0$. 
 %The fact that multiplying $F$ by $x\in K$ shifts both the Newton polygon and the $\mathcal{B}$-region vertically by $\ord_px$ allows us to proceed assuming $\mu(F)=0$. (Note that $p$-largeness is preserved by scaling.)
 Since $\mathcal{B}_n(F)$ contains the Newton polygon of $ \T_{\lambda-1}(F)=\sum_{i=0}^{\lambda-1}a_iT^i$, each vertex of the Newton polygon must lie on or above the graph of the function $f_F$. By definition this means that
\begin{equation}\label{SPL}
\ord_pa_i \geq f_F(i)
\end{equation}
for all $i\in \{0,\dots, \lambda-1\}$ (see Figure \ref{ffig1}). 
If $F$ is not $p$-large at $n$ then there must be some $i_0$ for which $\ord_p(a_{i_0})\leq v_{i_0}(F)-\frac{1}{e}$. 
%(here we use that $K$ is unramified),
But now \eqref{SPL} implies 
$
f_F(i_0)\leq v_{i_0}(F)-\frac{1}{e},
$
which contradicts Lemma \ref{lemv}. 
\end{proof}

\begin{figure*}[h!]
\begin{center}
%%%%%%%%%%%%%%%%%%%%%%%%%%%%%%%%%%%%%%%%%%%%%%%%%%%%%%%%%%%%%%%%%% q=0 & q>0 %%%%%%%%%%%%%%%%%%%%%%%%%%%%%%%%%%%%%%%%%%%%%%%%%%%%
\hspace{-5cm}
\begin{minipage}{.1\textwidth}
\newcommand*{\TopPath}{(0,5.5) -- (0,2) -- (6,1) -- (11,1) -- (11,5.5)}%
\begin{tikzpicture}[xscale=.5,yscale=.7]
\shade [top color= white, bottom color=gray] \TopPath -- cycle;
\draw [thick] \TopPath;
\draw (4,6) node[right,black,xshift=.3cm]{$Q(F)=0$};

\foreach \plm[count=\cnt] in {*}
\draw[] plot[only marks, mark size=+.7pt, mark=\plm] coordinates {(1,0) (2,0) (4,0) (5,0)  (7,0)
(8,0)(10,0) (1,1) (2,1) }; 
%draws lattice points with 'x's

\foreach \plm[count=\cnt] in {*}
\draw[] plot[only marks, mark size=+2.5pt, mark=\plm] coordinates {
                        (6,1)(7,1)(8,1)(10,1)(11,1)
(0,2) (1,2) (2,2)(4,2)(5,2)
}; 
%draws 'good' lattice points

\foreach \plm[count=\cnt] in {*}
\draw[] plot[only marks, mark size=+.5pt, mark=\plm] coordinates {
               (4,1)(5,1)(6,1)(7,1)
(0,2)(1,2)(2,2)(4,2)(5,2)(6,2)(7,2)(8,2)(10,2)(11,2)
(0,3)(1,3)(2,3)(4,3)(5,3)(6,3)(7,3)(8,3)(10,3)(11,3)
(0,4)(1,4)(2,4)(4,4)(5,4)(6,4)(7,4)(8,4)(10,4)(11,4)
(0,5)(1,5)(2,5)(4,5)(5,5)(6,5)(7,5)(8,5)(10,5)(11,5)
}; 
%draw other lattice points

\draw[{black}] plot coordinates {(0,5.3) (0,0) (11.5,0) }; %draws axes

\foreach \x in {1}
\draw (\x cm, 0pt) -- (\x cm,-0pt) %node[anchor=north,font=\footnotesize] 
node[anchor=north,xshift=1cm,font=\footnotesize]
{$ \dots$};

\foreach \x in {9}
\draw (\x cm, 0pt) -- (\x cm,-0pt) node[anchor=north,font=\footnotesize] {$ \dots$};

\foreach \x in {6}
\draw (\x cm, 3pt) -- (\x cm,-3pt) node[anchor=north,font=\footnotesize] {$ \scriptscriptstyle R(F)+p^{n-1}+1$};

\foreach \x in {11}
\draw (\x cm, 3pt) -- (\x cm,-3pt) node[anchor=north,font=\footnotesize] {$\scriptscriptstyle\lambda-1$};

\foreach \y in {2}
\draw (3pt, \y cm) -- (-3pt,\y cm) node[anchor=east,font=\footnotesize] {$ \scriptstyle 2$};

\foreach \y in {1}
\draw (3pt,\y cm) -- (-3pt,\y cm) node[anchor=east,font=\footnotesize] {$ \scriptstyle 1$};

\let\foo\undefined{\TopPath}
\end{tikzpicture}
\end{minipage}
\hspace{6cm}
%%%%%%%%%%%%%%%%%%%%%%%%%%%%%%%%%%%%%%%%%%%%%%%%%%%%%%
%%%%%% q>0 %%%%%%%%%%%%%%%%%%%%%%%%%%%%%%%%%%%%%%%%%%%%%%%%%%%%%%%%%%%%%%%%%%
\begin{minipage}{.0\textwidth}
\newcommand*{\TopPath}{(0,8.3) -- (0,6) -- (4,6) -- (17,1) -- (18,1)--(18,8.3)}%
\newcommand*{\DottedPathOne}{(-.5,2.8) -- (-.5,3.4)}
\newcommand*{\DottedPathTwo}{(1.5,-.4) -- (2.4,-.4)}
\newcommand*{\DottedPathThree}{(11.5,-.4) -- (12.4,-.4)}

\begin{tikzpicture}[xscale=.3,yscale=.46]
\hspace{-13mm}
\shade [top color= white, bottom color=gray] \TopPath -- cycle;
\draw [thick] \TopPath;
%\draw [draw=none] \TopPath;
%\draw [thick] \TopPathTwo;
%\draw [dotted, thick] \DottedPath;
\draw [dotted, thick] \DottedPathOne;
\draw [dotted, thick] \DottedPathTwo;
\draw [dotted, thick] \DottedPathThree;
\draw (7,9) node[right,black,xshift=.2cm]{$ Q(F)>0$};

\foreach \plm[count=\cnt] in {*}
\draw[{black}] plot coordinates {(0,8.3) (0,0) (18.5,0)}; %draws axes

\foreach \x in {4}
\draw (\x cm, 5pt) -- (\x cm,-5pt) 
%node[anchor=north,xshift=-.2cm,font=\footnotesize]
node[anchor=north,xshift=.2cm,font=\footnotesize] 
{$\scriptscriptstyle R(F)+p^{n-1}$};

\foreach \x in {17}
\draw (\x cm, 5pt) -- (\x cm,-5pt) node[anchor=north,xshift=-2.6mm,font=\footnotesize] {$\scriptscriptstyle \lambda-t_n+1$};

\foreach \x in {18}
\draw (\x cm, 5pt) -- (\x cm,-5pt) node[anchor=north,xshift=1.9mm,font=\footnotesize] {$\scriptscriptstyle \lambda-1$};

% y-axis labels

\foreach \y in {1}
\draw (5pt,\y cm) -- (-5pt,\y cm) node[anchor=east,font=\footnotesize] {$\scriptscriptstyle 1$};

\foreach \y in {2}
\draw (5pt,\y cm) -- (-5pt,\y cm) node[anchor=east,font=\footnotesize] {$\scriptscriptstyle 2$};

\foreach \y in {4}
\draw (5pt,\y cm) -- (-5pt,\y cm) node[anchor=east,font=\footnotesize] {$\scriptscriptstyle Q(F)$};

\foreach \y in {5}
\draw (5pt,\y cm) -- (-5pt,\y cm) node[anchor=east,font=\footnotesize] {$\scriptscriptstyle Q(F)+1$};

\foreach \y in {6}
\draw (5pt,\y cm) -- (-5pt,\y cm) node[anchor=east,font=\footnotesize] {$\scriptscriptstyle Q(F)+2$};

\foreach \plm[count=\cnt] in {*}
\draw[] plot[only marks, mark size=+.7pt, mark=\plm] coordinates {
(0,0)(1,0)(3,0)(4,0)(5,0)(6,0)(7,0)(8,0)(9,0)(10,0)(11,0)(13,0)(14,0)(15,0)(16,0)(17,0)
(0,1)(1,1)(3,1)(4,1)(5,1)(6,1)(7,1)(8,1)(9,1)(10,1)(14,1)   (14,1)(15,1)(16,1)(17,1)(18,1)  
(0,2)(1,2)(3,2)(4,2)(5,2)(6,2)(7,2)(8,2)(9,2)(10,2)(14,2)     
(0,4)(1,4)(3,4)(4,4)(5,4)(6,4)(7,4)(8,4)(9,4)(10,4)(14,4)(15,4)(16,4)(17,4)(18,4)
(0,5)(1,5)(3,5)(4,5)(5,5)(6,5)(7,5)(8,5)(9,5)(10,5)(14,5)(15,5)(16,5)(17,5)(18,5)
(0,6)(1,6)(3,6)(4,6)(5,6)(6,6)(7,6)(8,6)(9,6)(10,6)(14,6)(15,6)(16,6)(17,6)(18,6)
(0,7)(1,7)(3,7)(4,7)(5,7)(14,7)(15,7)(16,7)(17,7)(18,7)
(18,0)(18,1)(18,2)(14,8)(15,8)(16,8)(17,8)(18,8)
(11,1)(11,2)(11,4)(11,5)(11,6)(11,7)(11,8)
(13,1)(13,2)(13,4)(13,5)(13,6)(13,7)(13,8)
(17,2)
}; 
%draws lattice points with 'x's

\foreach \plm[count=\cnt] in {*}
\draw[] plot[only marks, mark size=+3.3pt, mark=\plm] coordinates {
(0,6)(1,6)(3,6)(4,6)
(5,5) (6,5) (7,5)
(8,4) (9,4) (10,4)
(14,2) (15,2) (16,2)
(17,1)(18,1)
}; 
%draws 'good' lattice points

\foreach \plm[count=\cnt] in {*}
\draw[] plot[only marks, mark size=+.5pt, mark=\plm] coordinates {   
(0,8)(1,8)(3,8)(4,8)(5,8)(6,8)(7,8)(8,8)(9,8)(10,8)(14,8)      
(6,7)(7,7)(8,7)(9,7)(10,7)(14,7)      
(9,6)(10,6)(14,6)      
(14,5)      
(18,4)(18,5)(18,6)(18,7)(18,8)
}; 
%draw other lattice points
\let\foo\undefined{\TopPath}
\end{tikzpicture}
\end{minipage}
\caption{ In gray are the regions $\cB_n(F)$ in the cases $Q(F)=0$ and $Q(F)>0$, under the assumption $\mu=0$. In bold are the points $(i,v_i(F))$.  The diagram illustrates Proposition \ref{newton}: if the Newton polygon of $\T_{\lambda-1}(F)$ is contained entirely in $\mathcal{B}_n(F)$  then each vertex of the Newton polygon is a lattice point within the gray region, implying that $\ord_pa_i\geq v_i(F)$ for all $i\in \{0,\dots, \lambda-1\}$ and therefore that $F$ is $p$-large at $n$.}
 %represent the lowest possible valuations permissible for $F$ to be $p$-large at $n$. If the Newton polygon of $F$ is contained entirely in $\mathcal{B}_n(F)$  then all vertices of the Newton polygon lie on or above the bold points, indicating that $F$ must be $p$-large at $n$.}
\label{ffig1}
\end{center}
\end{figure*}

%%%%%%%%%%%%%%%%%%%%%%%%%%%%%%%%%%%%%%%%%%%%%%%%%%%%%%%%%%%%%%%%%%%%%%%%%%%%%%%%%%%%%%%%%%%%%%%%%%%%%%%%%%%%
\section{Iwasawa invariants of Mazur-Tate elements when $k\geq p+2$}\label{mainsection}

In this section we prove Theorem \ref{main1} for weights $k\geq p+2$. In light of Theorem \ref{lifts}, this amounts to computing the Iwasawa invariants of $\pi_n(\eta^j(L_p^\pm(f)\log_{k,n}^{\pm}))$ for $n\gg0$, a task to which we now turn. 

%%%%%%%%%%%%%%%%%%%%%%
\subsection{Iwasawa invariants of half-logarithms at layer $n$}\label{plargeiw} %We assume for \S\ref{plargeiw} that $K/\Q_p$ is unramified. 
As a placeholder for the signed $p$-adic $L$-functions, in the current section we let $\mathcal{L}\in \Lambda_K$ denote a fixed power series with Iwasawa invariants $\lambda$ and $\mu$. Our goal is to compute the Iwasawa invariants of $\pi_n(\eta^j(\mathcal L\log_{k,n}^{\e_{n+1}}))$  when $k\geq p+2$.
Note that since 
\begin{align*}
\mu(\eta^j(\mathcal L\log_{k,n}^{\e_{n+1}}))&=\mu,\\
\lambda(\eta^j(\mathcal L\log_{k,n}^{\e_{n+1}}))&=\lambda+(k-1)q_n,
\end{align*}
when $k\leq p+1$ we have $\lambda+(k-1)q_n<p^n$ for $n\gg0$ and the $n$th-layer invariants are therefore unchanged by Lemma \ref{small_lam}.

We begin with some preliminary remarks.
First, recall from the proof of Lemma \ref{omeganinv} that $\omega_{n,i}^{\e_{n+1}}$ has $\lambda$-invariant $q_n$ and trivial $\mu$-invariant. By Lemma \ref{etainv}, the same is true of the polynomials $\eta^j(\omega_{n,i}^{\e_{n+1}})$. Thus, by the Weierstrass preparation theorem, we may write 
$
\eta^j(\omega_{n,i}^{\e_{n+1}})=(T^{q_n}+pW_{i})U_{i},
$
for some $W_i\in \Z_p[T]$ of degree $<q_n$ and units $U_i\in \Lambda^\times$, both dependent on $n$ and $j$. Therefore, 
$$
\eta^j(\log_{k,n}^{\e_{n+1}})\equiv \prod_{i=0}^{k-2}(T^{q_n}+pW_i)\Mod \Lambda^\times.
$$
Similarly, we have
$
\eta^j(\mathcal L)\equiv p^\mu (T^\lambda+ p L_0)\Mod \Lambda^\times
$
for some $L_0\in \Oo[T]$ of degree $<\lambda$. It follows that
\begin{equation}\label{logdecomp1}
\eta^j(\mathcal L \log_{k,n}^{\e_{n+1}})=LF_{k,n}U_n
\end{equation}
where $U_n\in \Lambda^\times$ and 
\begin{align*}
L &:=p^\mu(T^\lambda+p L_0),\\
F_{k,n}&:=\prod_{i=0}^{k-2}(T^{q_n}+pW_i).
\end{align*}
Thus,
by Corollary \ref{nunits},
\begin{align}\label{logdecomp2}\begin{split}
\mu(\pi_n(\eta^j(\mathcal L\log_{k,n}^{\e_{n+1}})))&=\mu(\pi_n(LF_{k,n})),\\
\lambda(\pi_n(\eta^j(\mathcal L\log_{k,n}^{\e_{n+1}})))&=\lambda(\pi_n(LF_{k,n})).
\end{split}
\end{align}
It therefore suffices to compute the Iwasawa invariants of $\pi_n(LF_{k,n})$. To  accomplish this, we need only show that $LF_{k,n}$ is $p$-large at all sufficiently large values of $n$, in which case we can appeal to Theorem \ref{mainpl}.

Before proceeding, we set some notation. Let 
$$
\lambda_{n}=\lambda(LF_{k,n})=\lambda+(k-1)q_n.
$$
Define
\begin{align}
\nu&=\bigg\lfloor\frac{k-p-2}{p^2-1}\bigg\rfloor,\qquad  \nu^- =\nu(p-1)+1,\qquad \nu^+ = p\nu^{-}\label{nudef1}
\end{align}
Let $l, a,$ and $b$ be the unique integers satisfying 
\begin{align}
k-p-2&=l+\nu(p^2-1)&(0\leq l\leq p^2-2)\label{lkdef}\\
l&=a+bp& (0\leq a \leq p-1),\label{abkdef}
\end{align}

%%%%%
\begin{lemma}\label{RQ} For $n\gg0$ we have
\begin{align*}
R(LF_{k,n})_n&=\begin{cases}
\lambda-\nu^{\e_{n+1}}+aq_n+bq_n'&\text{if $k\not\equiv p+2 \Mod (p^2-1)$, }\\
\lambda-\nu^{\e_{n+1}} & \text{if $k\equiv p+2 \Mod (p^2-1)$ and $\lambda \geq \nu^{\e_{n+1}}$,}\\
\lambda-\nu^{\e_{n+1}}+p^n-p^{n-1}&\text{if $k\equiv p+2 \Mod (p^2-1)$ and $\lambda <\nu^{\e_{n+1}}$,}\\\end{cases}\\
Q(LF_{k,n})_n&=\begin{cases}
\frac{(k-1-a)p-b}{p^2-1}-\frac{p}{p-1}&\text{if $k\not\equiv p+2 \Mod (p^2-1)$, }\\
\frac{(k-1)p}{p^2-1}-\frac{p}{p-1}& \text{if $k\equiv p+2 \Mod (p^2-1)$ and $\lambda \geq \nu^{\e_{n+1}}$,}\\
\frac{(k-1)p}{p^2-1}-\frac{p}{p-1}-1&\text{if $k\equiv p+2 \Mod (p^2-1)$ and $\lambda <\nu^{\e_{n+1}}$.}
\\\end{cases}
\end{align*}
\end{lemma}
\begin{proof}  Since $k-p-2=l+\nu(p^2-1)=a+bp+\nu(p^2-1)$, we have
\begin{align*}
R(F_{k,n})_n&\equiv (k-1)q_n-p^n\Mod(p^n-p^{n-1})\\
%&\equiv aq_n+bpq_n+q_n(p+1)(1+\nu(p-1))-p^n\Mod(p^n-p^{n-1})\\
%&\equiv aq_n+bq_n'-\left\{\begin{aligned}
%1+\nu(p-1),& \quad\text{$n$ even}\\[1ex]
%p+\nu(p^2-p), & \quad\text{$n$ odd}\\[1ex] 
%\end{aligned}\right\}\Mod(p^n-p^{n-1})\\
&\equiv aq_n+bq_n'-\nu^{\e_{n+1}}\Mod (p^n-p^{n-1}).
\end{align*}
The desired equations for $R(LF_{k,n})_n$ now follow from $R(LF_{k,n})_n\equiv \lambda+R(F_{k,n})_n\Mod (p^n-p^{n-1})$ and the definitions of $a$ and $b$. 
Now write $Q(LF_{k,n})_n=(\lambda_n-p^n-R(LF_{k,n})_n)/t_n$, substitute in the equations for $R(LF_{k,n})_n$ just computed, and take the limit as $n\rightarrow \infty$. (Use the fact that $q_n/t_n\rightarrow p/(p^2-1)$ and $q_n'/t_n\rightarrow 1/(p^2-1)$.)
\end{proof}

%%%%%%%%%
\begin{proposition}\label{lpl} Suppose that $p>2$, or that $p=2$ and $e\leq 2$. If
\begin{itemize}
\item[(i)] $k\geq p+3$, 
\item[(ii)] $k=p+2$ and $\lambda\geq 1$, or
\item[(iii)] $k=p+2$ and $\lambda\geq p$,
\end{itemize}
then $LF_{k,n}$ is $p$-large at $n$ for all $n\gg0$ (if (i)), all even $n\gg0$ (if (ii)), or all odd $n\gg 0$ (if (iii)).% If $p=2$ then the same statements hold under the additional assumption that $e\leq 2$.
\end{proposition}

\begin{proof} First, note that our assumptions on $k$ and $\lambda$ force $\lambda_n\geq p^n$ for all $n\gg0$ (if (i)), or for all even (resp., odd) $n$ if (ii) (resp., (iii)). Letting 
$$
G_{k,n} = p^\mu(p^{k-1}+pT^{(k-2)q_n}+pT^{\lambda_n-1}+T^{\lambda_n}),
$$
 it suffices to show
\begin{equation}\label{sufnp}
\NP_p\big(\T_{\lambda_{n}-1}(G_{k,n}) \big)\subseteq \mathcal{B}_n(G_{k,n}),\qquad n\gg0,
\end{equation} 
since then (using Lemma \ref{boundNP} and the fact that $LF_{k,n}$ and $G_{k,n}$ have the same Iwasawa invariants) we have
$$
\NP_p\big(\T_{\lambda_{n}-1}(LF_{k,n})) \subseteq\NP_p\big(\T_{\lambda_{n}-1}(G_{k,n}))
\subseteq \mathcal{B}_n(G_{k,n})
=\mathcal{B}_n(LF_{k,n}), 
$$
in which case the result follows from Proposition \ref{newton}. 

We may assume $\mu=0$. Set $R_n=R(LF_{k,n})_n$ and $Q_n=Q(LF_{k,n})_n$, so 
$$
\lambda_{n}=\lambda+(k-1)q_n=p^n+R_n+Q_nt_n.
$$
Let $\NP_n := \NP_p\big(\T_{\lambda_{n}-1}(G_{k,n}) \big)$ and  $\cB_n:=\cB_n( G_{k,n})$, both of which are pictured in Figure \ref{fig3}.  Let
\begin{align*}
l(\NP_n)&:\quad  y=k-1-\frac{x}{q_n},\\
l(\cB_n)&: \quad y=
\begin{cases}
Q_n+2-\frac{x}{R_n+p^{n-1}+1}&\quad \text{if $Q_n=0$}\\
Q_n+2-\frac{Q_n+\frac{1}{e}}{Q_nt_n+1}\big(x-R_n-p^{n-1}\big)&\quad \text{if $Q_n>0$}
\end{cases}
\end{align*}
denote the unique line segments of $\NP_n$ and $\cB_n$ with nonzero and non-infinite slope. The containment \eqref{sufnp} will follow if we can establish the following properties (see Figure \ref{fig3}) for $n$ large enough:
\begin{enumerate}
\item The middle vertex of $\NP_n$ lies along the rightmost (horizontal) edge of $\mathcal{B}_n$. 
\item The slope of $l(\NP_n)$ is less than or equal to the slope of $l(\cB_n)$. 
%The unique nonzero slope of $\NP_n$ is less than or equal to the unique nonzero slope of $\mathcal{B}_n$.
\item If $Q_n>0$, then the rightmost vertex of $l(\cB_n)$ is below the point on $l(\NP_n)$ with $x$-coordinate $\lambda_n-t_n+1$. 
%then the line with nonzero slope on $\NP_n$ does not intersect the line with nonzero slope on $\mathcal{B}_n$. 
\end{enumerate}

%%%%%%%%%%%%%%%%%%%%%%%%%%%%%%%%%%%%%%%%%%%%%%%%%%%%%%%%%%%%%%%%%% Lemma 4.4 M=0 and M>0 %%%%%%%%%%%%%%%%%%%%%%%%%%%%%%%%%%%%%%%%%%%%%%%%%%%

\begin{figure*}[htp!]
\begin{center}
\hspace{-9cm}
\begin{minipage}{.2\textwidth}
\newcommand*{\SolidPath}{(0,6.5) -- (0,2) --(4.5,1) -- (8,1) -- (8,6.5) }
\newcommand*{\DottedPath}{(0,6.5) -- (0,5) -- (6.5,1) -- (8,1) -- (8,6.5)}

\begin{tikzpicture}[xscale=.6,yscale=.6]
%\shade [top color= white, bottom color=gray] \TopPath -- cycle;
%\draw [draw=none] \TopPath;
\draw [thick] \SolidPath;
\draw [dotted, ultra thick] \DottedPath;
\draw (3,6.5) node[right,black,xshift=.3cm]{$Q_n=0$};

\foreach \plm[count=\cnt] in {*}
\draw[{black}] plot coordinates {(0,6.5) (0,0) (8.5,0)}; %draws axes

% x-axis labels
\foreach \x in {8}
\draw (\x cm, 3pt) -- (\x cm,-3pt) node[anchor=north,font=\footnotesize,xshift=.1cm] {$\scriptstyle \lambda_n-1$};

\foreach \x in {6.5}
\draw (\x cm, 3pt) -- (\x cm,-3pt) node[anchor=north,font=\footnotesize,xshift=.1cm] {$\scriptstyle (k-2)q_n$};

\foreach \x in {4.5}
\draw (\x cm, 3pt) -- (\x cm,-3pt) node[anchor=north,font=\footnotesize,xshift=-.1cm] {$\scriptstyle R_n+p^{n-1}+1$};

% y-axis labels
\foreach \y in {1}
\draw (2pt,\y cm) -- (-2pt,\y cm) node[anchor=east,font=\footnotesize] {$\scriptstyle1$};

\foreach \y in {2}
\draw (2pt,\y cm) -- (-2pt,\y cm) node[anchor=east,font=\footnotesize] {$\scriptstyle2$};

\foreach \y in {5}
\draw (2pt,\y cm) -- (-2pt,\y cm) node[anchor=east,font=\footnotesize] {$\scriptstyle k-1$};

\foreach \plm[count=\cnt] in {*}
\draw[] plot[only marks, mark size=+3pt, mark=\plm] coordinates {
(0,5) (0,2) (4.5,1) (6.5,1) (8,1)
}; 
%draws vertices

\let\foo\undefined{\SolidPath}
\let\foo\undefined{\DottedPath}
\end{tikzpicture}
\end{minipage}
\hspace{4cm}
\begin{minipage}{.0\textwidth}

\newcommand*{\SolidPath}{(0,6.5) -- (0,2.5) --(2.5,2.5) -- (8,1) -- (8,0.5) -- (11,0.5) -- (11,6.5) }
\newcommand*{\DottedPath}{(0,6.5) -- (0,6) -- (9.5,0.5) -- (11,0.5) -- (11,6.5)}

\begin{tikzpicture}[xscale=.6,yscale=.6]
%\shade [top color= white, bottom color=gray] \TopPath -- cycle;
%\draw [draw=none] \TopPath;
\draw [thick] \SolidPath;
\draw [dotted, ultra thick] \DottedPath;
\draw (4,6.5) node[right,black,xshift=.3cm]{$Q_n>0$};
%\draw (4,3.8) node[right,black,xshift=.3cm]{$-\frac{1}{q_n}$};

\foreach \plm[count=\cnt] in {*}
\draw[{black}] plot coordinates {(0,6.5) (0,0) (11.5,0)}; %draws axes

% x-axis labels
\foreach \x in {11}
\draw (\x cm, 3pt) -- (\x cm,-3pt) node[anchor=north,font=\footnotesize,xshift=.2cm] {$\scriptstyle \lambda_n-1$};

\foreach \x in {9.5}
\draw (\x cm, 3pt) -- (\x cm,-3pt) node[anchor=north,font=\footnotesize,xshift=0cm] {$\scriptstyle (k-2)q_n$};

\foreach \x in {8}
\draw (\x cm, 3pt) -- (\x cm,-3pt) node[anchor=north,font=\footnotesize,xshift=-.35cm] {$\scriptstyle \lambda_n-t_n+1$};

\foreach \x in {2.5}
\draw (\x cm, 3pt) -- (\x cm,-3pt) node[anchor=north,font=\footnotesize] {$\scriptstyle R_n+p^{n-1}$};

% y-axis labels
\foreach \y in {1}
\draw (2pt,\y cm) -- (-2pt,\y cm) node[anchor=east,font=\footnotesize] {$\scriptstyle2-\frac{1}{e}$};

\foreach \y in {0.5}
\draw (2pt,\y cm) -- (-2pt,\y cm) node[anchor=east,font=\footnotesize] {$\scriptstyle1$};

\foreach \y in {2.5}
\draw (2pt,\y cm) -- (-2pt,\y cm) node[anchor=east,font=\footnotesize] {$\scriptstyle Q_n+2$};

\foreach \y in {6}
\draw (2pt,\y cm) -- (-2pt,\y cm) node[anchor=east,font=\footnotesize] {$\scriptstyle k-1$};

\foreach \plm[count=\cnt] in {*}
\draw[] plot[only marks, mark size=+3pt, mark=\plm] coordinates {
(0,6) (0,2.5) (2.5,2.5) (8,1) (8,0.5) (9.5,0.5) (11,0.5)
}; 
%draws vertices
\let\foo\undefined{\SolidPath}
\let\foo\undefined{\DottedPath}
\end{tikzpicture}
\end{minipage}
\caption{The boundary of $\NP_n$ (dotted) and $\cB_n$ (bold) when $Q_n=0$ and $Q_n>0$, illustrating the containment \eqref{sufnp}.}
\label{fig3}
\end{center}
\end{figure*}

\begin{enumerate}
\item[Case $Q_n=0$:] \nf %The region $\mathcal{B}_n$ has vertices $(0,2)$, $(R_n+p^{n-1}+1,1)$, and $(\lambda_n-1,1)$. 
To establish (1), we must show that 
 \begin{equation}\label{ieq}
 R_n+p^{n-1}+1\leq (k-2)q_n,\qquad n\gg0.
 \end{equation}
If $k\geq  p+4$ then $(R_n+p^{n-1}+1)/q_n\leq p^n/q_n\rightarrow p+1$ as $n\rightarrow \infty$, hence $(R_n+p^{n-1}+1)/q_n\leq p+2\leq k-2$ for $n\gg0$. If $k=p+3$ then from Lemma \ref{RQ}, $R_n=\lambda-1+q_n$ or $\lambda-p+q_n$ depending on whether $n$ is even or odd, respectively. Now \eqref{ieq} follows from the fact that one can take $n\gg0$ so that $\lambda\leq t_n-q_n.$ Finally, suppose $k=p+2$ and $\lambda\geq \nu_{p+2}^{-*} (=\text{1 or $p$})$ for some $*\in \{+,-\}$. Then Lemma \ref{RQ} gives $R_n=\lambda-\nu_{p+2}^{-*}$ for all $n\gg0$ of parity $*$, and \eqref{ieq} follows from the fact that $\lambda\leq pq_n-p^{n-1}$ for large enough $n$. To establish (2), we must show that 
 $-\frac{1}{q_n}\leq -\frac{1}{R_n+p^{n-1}+1}$ for $n\gg0$, and this follows  from the fact that $q_n<p^{n-1}\leq p^{n-1}+R_n+1$. 

\item[Case $Q_n>0$:] \nf  
  Since $q_n\leq t_n-1$, we know that $\lambda_n-t_n+1\leq (k-2)q_n$ for $n\geq 1$, thus (1) holds.  For (2), we need to show that  $-\frac{1}{q_n}\leq -\frac{Q_n+\frac{1}{e}}{Q_nt_n+1}$,  which is equivalent to checking that $\frac{1}{e}(q_n-e)/(t_n-q_n)\leq Q_n$, and this is true since $\frac{1}{e}(q_n-e)/(t_n-q_n)<1$ for $n\gg 0$. To establish (3), it suffices to check that 
  $$
  2-\frac{1}{e}\leq k-1-\frac{1}{q_n}(\lambda_n-t_n+1),\quad n\gg0.
  $$
Taking the limit as $n\rightarrow \infty $, the above inequality becomes $2-1/e\leq (p^2-1)/p$, which holds for all $p\geq 3$, and at $p=2$ under the assumption $e\leq 2$. 
  \end{enumerate}
  \end{proof}

For $*\in \{+,-\}$, define
\begin{equation}\label{iotaL}
\cI_k^*(\cL):=
 \begin{cases}
\big\lfloor\frac{p(k-2)-1}{p^2-1}\big\rfloor&\text{if  $k\not\equiv p+2 \Mod (p^2-1)$}\\
\frac{p(k-2)-1}{p^2-1}& \text{if $k\equiv p+2 \Mod (p^2-1)$ and $\lambda \geq \nu_k^{*}$}\\
\frac{p(k-2-p)}{p^2-1}&\text{if $k\equiv p+2 \Mod (p^2-1)$ and $\lambda < \nu_k^{*}$},
\end{cases}
\end{equation}
where we recall $\lambda=\lambda(\cL)$. 

\begin{corollary} \label{coro_Llog} Suppose that $p>2$, or that $p=2$ and $e\leq 2$. If $k\geq p+2$ 
then for $n\gg0$ we have 
%Under the assumptions of the previous lemma, we have for $n\gg0$ that
\begin{align*}
\mu(\pi_n(\eta^j(\cL\log_{k,n}^{\e_{n+1}})))&=\mu+\cI_k^{*}(\cL),\\
\lambda(\pi_n(\eta^j(\mathcal L\log_{k,n}^{\e_{n+1}})))&=\lambda+(k-1)q_n-\cI_k^{*}(\cL)(p^n-p^{n-1})\\
&= \lambda-\nu_k^{*}+\begin{cases}
%aq_n+bq_n'-p^{n-1}(p-2)&\text{if $k=2$}\\
aq_n+bq_n'+p^{n-1}&\text{if $k\not\equiv p+2 \Mod (p^2-1)$, }\\
p^{n-1} &\text{if $k\equiv p+2 \Mod (p^2-1)$ and $\lambda\geq \nu^{*}$,}\\
p^{n} &\text{if $k\equiv p+2 \Mod (p^2-1)$ and $\lambda< \nu^{*}$,}\\
\end{cases}
\end{align*}
where $*=-$ if $n$ is even and $*=+$ if $n$ is odd. 
\end{corollary}
\begin{proof} %By Corollary \ref{RQ}, one can check using the definitions of $a$ and $b$ that for $k\geq p+2$, 
%$1+\displaystyle\lim_{\substack{n\rightarrow\infty \\ \e_n=-*}}Q(LF_{k,n})_n$
By \eqref{logdecomp2}, Proposition \ref{lpl}, Theorem \ref{mainpl}, it suffices to check that 
$$
\cI_k^*(\cL)=1+\displaystyle\lim_{\substack{n\rightarrow\infty \\ \e_n=-*}}Q(LF_{k,n})_n
$$
for $*\in \{+,-\}$. %Using Lemma \ref{RQ}, this limit is clear  when $k\equiv p+2\Mod (p^2-1)$, and if $k\not \equiv p+2\Mod (p^2-1)$ then 
This follows from Lemma \ref{RQ} and the definitions of $a$ and $b$. 
\end{proof}

\subsection{Iwasawa Invariants of Mazur-Tate elements at large weights}\label{section_largeweights}

%In light of Theorem \ref{theorem_smallweights}, the goal of this section is to prove Theorem \ref{main1} when $k\geq p+2$.

 % or when $k=p+2$ and $\lambda^\pm\geq BOUNDS$. 
 %$\psi\in \hat\Delta$ and $j\in \{0,\dots, k-2\}$. 
 
 We are now ready to prove the remaining parts of Theorem \ref{main1}. 
Let $f$ be as in \S\ref{section_MT} and assume $a_p(f)=0$. Fix $\psi\in \hat\Delta$, $j\in \{0,\dots, k-1\}$, and let $\theta_{n,j}^\psi\in \Lambda_{K,n}$ and $L_{p,\psi}^\pm\in \Lambda_K$ denote the Mazur-Tate elements and signed $p$-adic $L$-functions associated to $f$. Define
$$
\iota^*(f,\psi):=\cI_{k}^{*}\big(L_{p,\psi}^{\star*}\big),
$$
where $*\in \{+,-\}$ and $\cI_k^\pm$ is as defined in \eqref{iotaL}.  
To ease notation, write $\mu^\pm(\psi) = \mu\big(L_{p,\psi}^{\pm}\big)$, $\lambda^\pm(\psi)= \lambda\big(L_{p,\psi}^{\pm}\big)$, and $\iota^\pm(\psi)=\iota(f,\psi)$. 
Combined with Theorem \ref{theorem_smallweights}, Theorem \ref{main1} is obtained from the following theorem by restricting to $p>2$, $\psi=1$, and $j=0$. Recall $e$ is the ramification index of $K$ and $\omega$ is the mod $p$ cyclotomic character. %and $C$ is the constant appearing in Theorem \ref{lifts}. 
%the goal of this section is to prove Theorem \ref{main1} when $k\geq p+2$.

 %%%%%%%%
\begin{theorem}\label{main} Let $f\in S_k(\Gamma_1(N))$ be a newform with $a_p(f)=0$ and $k\geq p+2$. Suppose that $p>2$, or that $p=2$ and $e\leq 2$. Then for all $n\gg0$ we have
\begin{align}
\mu\big(\theta_{n,j}^{\psi}\big)&=\mu^{\star*}(\psi\omega^j) +\iota^*(\psi\omega^j)+\ord_p\textstyle \binom{k-2}{j}+\begin{cases}k-1 &\quad \text{if $p=2$ and $*=-$}\\
0&\quad \text{otherwise},\label{muthm}\\ 
\end{cases}
\\
\lambda(\theta_{n,j}^{\psi}) &=\lambda^{\star*}(\psi\omega^j)+(k-1)q_n-\iota^*(\psi\omega^j)(p^n-p^{n-1}) \label{lam1thm}\\
&=\lambda^{\star*}(\psi\omega^j)-\nu_k^{*}\\
&\quad  +\begin{cases}
aq_n+bq_n'+p^{n-1}&\text{if $k\not\equiv p+2\Mod (p^2-1)$}
\\
p^{n}& \text{if $k\equiv p+2\Mod (p^2-1)$ and $\lambda^{\star*}(\psi\omega^j)<\nu_k^*$}\\
p^{n-1}& \text{if $k\equiv p+2\Mod (p^2-1)$ and $\lambda^{\star*}(\psi\omega^j)\geq \nu_k^*$} 
\end{cases}  \label{lam2thm}
\end{align}
where $*=-$ if $n$ is even and $*=+$ if $n$ is odd, and the integers $a$ and $b$ are as defined in \eqref{abkdef}.
\end{theorem}

\begin{proof} By Theorem \ref{lifts} and equations \eqref{nliw1} and \eqref{nliw2} we have 
\begin{align*}
\mu(\theta_{n,j}^\psi)&=\mu\big(\pi_n(\eta^j(L_{p,\psi\omega^j}^{\star\e_{n+1}}\log_{k,n}^{\e_{n+1}}))\big)+\ord_p\textstyle \binom{k-2}{j}+\begin{cases}k-1 &\quad \text{if $p=2$ and $n$ is odd}\\
0&\quad \text{otherwise},\label{muthm}\\ 
\end{cases}
\\
\lambda(\theta_{n,j}^\psi)&=\lambda\big(\pi_n(\eta^j(L_{p,\psi\omega^j}^{\star\e_{n+1}}\log_{k,n}^{\e_{n+1}}))\big).
\end{align*}
The statement now follows from Corollary \ref{coro_Llog}.
\end{proof}

We refer the reader to Appendix \ref{tablessection} for examples illustrating the above theorem. 

\begin{remark}\label{remark_effective} \nf Let $n_0\geq 0$ be the smallest integer $n$ such that the $\lambda$-invariants of $\theta_{n,j}^\psi$ follow the pattern of Theorem \ref{main1}. At small weights $k\leq p+1$, we see from Theorem \ref{theorem_smallweights} that $n_0$ depends the signed invariants $\lambda^\pm(\psi\omega^j)$. Computations suggest this may also be true at higher weights. For example, in the weight $k=10$ block of Table \ref{table_p3}, we have $n_0=1$ for all modular forms except $\texttt{G0N256k10A}$, which has $n_0=5$. Note that the plus/minus $\lambda$-invariants of $\texttt{G0N256k10A}$ (7 and 9, respectively) are larger than the corresponding invariants (6 and 2) of the other modular forms in this block. One possible explanation for this is that the proof of Proposition \ref{lpl} requires showing certain inequalities (e.g., \eqref{ieq}) involving $\lambda^\pm$ hold for $n\gg0$. In particular, the larger the invariants $\lambda^\pm$, the larger one must take $n$ for these inequalities to hold.
\end{remark}

%%%%%%%%%%%%%
\section{Corollaries and conjectures}

%%%%%%%%%%%%%%%%%%%%%%%
\subsection{Signed invariants attached to congruent forms of weights 2 and $\pmb{p+1}$ }\label{pminvsection}

In this section, we prove Corollary \ref{coro_C}.
First, we recall the main properties of Sprung's $\sharp/\flat$ $p$-adic $L$-functions.
If $g$ is a $p$-non-ordinary eigen-cuspform of weight two and level not divisible by $p$, Sprung \cite{sprung17} constructed functions $L_{p,\psi}^{\sharp/\flat}(g)$ with $p$-adically bounded coefficients (hence well-defined Iwasawa invariants) which, like Pollack's signed $p$-adic $L$-functions, are  conjecturally related to the structure of the Pontryagin dual of certain signed Selmer groups attached to $g$ along the cyclotomic $\Z_p$-extension of $\Q$ (see \cite{sprung09} and \cite[\S 6.2]{LLZ0}).  If  $a_p(g)=0$, the $\sharp/\flat$-$p$-adic $L$-functions of Sprung agree with the $\pm$-$p$-adic $L$-functions of Pollack via $\sharp\leftrightarrow +$ and $\flat\leftrightarrow -$ (see \cite[Remark 6.16]{sprung09}). 
In what follows, we let $\overline{\rho_f}:G_\Q\rightarrow \GL(\overline{\F_p})$ denote the residual representation attached to an eigenform $f\in S_k(N,\e)$ by Deligne \cite{deligne69}. 

%%%%%%%
\begin{corollary}\label{coro_p1} Let $p>2$ and let $f\in S_{p+1}(\Gamma_0(N))$ be a newform with $a_p(f)=0$. Then there exists an eigenform $g\in S_2(\Gamma_0(N))$ with $\overline{\rho_f}\cong \overline{\rho_g}$ such that if $\mu\big(L_{p,\psi}^{\sharp}(g)\big) =\mu\big(L_{p,\psi}^{\flat}(g)\big)$ then: 
\begin{enumerate}
\item $\mu\big(L_{p,\psi}^+(f)\big)=\mu\big(L_{p,\psi}^+(f)\big)=0$ if and only if $\mu\big(L_{p,\psi}^{\sharp}(g)\big) =\mu\big(L_{p,\psi}^{\flat}(g)\big) =0$, and 
\item if either of the conditions in (1) hold then 
\begin{align*}
\lambda\big(L_{p,\psi}^+(f)\big)&= \lambda\big(L_{p,\psi}^\flat(g)\big)+p-1,\\
\lambda\big(L_{p,\psi}^-(f)\big)&=\lambda\big(L_{p,\psi}^\sharp(g)\big).
\end{align*}
\end{enumerate}
\end{corollary}
\begin{proof} Let $G_{\Q_p}=\Gal(\overline{\Q}_p/\Q_p)$, which we identify with the decomposition group corresponding to our fixed embedding $\overline{\Q}\hookrightarrow \overline{\Q_p}$. By \cite[Theorem 2.6]{FontaineEdixhoven92} (see also \cite[Remark 5.2.5]{PW}), since $f$ has weight $p+1$ we know that the representation $\overline{\rho_f}|_{G_{\Q_p}}$ is irreducible with Serre weight two (in the sense of \cite[Definition VII.1.7]{cornell2013modular}). 
It follows from \cite[Corollary 5.3]{PW} that there exists an eigenform $g\in S_2(\Gamma_0(N))$ with $\overline{\rho_f}\cong \overline{\rho_g}$ such that
\begin{equation}\label{mu_p1}
\text{$\mu(\theta_n^\psi(f))=0$ for $n\gg0$ if and only if $\mu^{+}(g,\psi)=\mu^{-}(g,\psi)=0$,}
\end{equation}
 and if either of these conditions hold then
\begin{equation}\label{lambda_p1}
\lambda(\theta_n^\psi(f))=p^n-p^{n-1}+q_{n-1}+\lambda^{\e_n}(g,\psi), \qquad n\gg0.
\end{equation}
Here $\mu^{+}(g,\psi):=\lim_{n\rightarrow \infty }\mu(\theta_{2n+1}^\psi(g))$ and $\lambda^{+}(g,\psi):=\lim_{n\rightarrow \infty}\lambda(\theta_{2n+1}^\psi(g))-q_{2n+1}$, and the `minus' counterparts are defined similarly as limits over even integers. By Theorem \ref{theorem_smallweights}, we have $\mu(\theta_n^\psi(f))=\mu(L_{p,\psi}^{-\e_n}(f))$ for $n\gg0$. By \cite[Corollary 8.9]{sprung17}, if the $\mu$-invariants of $L_{p,\psi}^{\sharp}(g)$ and $L_{p,\psi}^{\flat}(g)$ agree then $\mu^{\pm}(g,\psi)=\mu(L_{p,\psi}^{\sharp/\flat}(g))$ and $\lambda^{\pm}(g,\psi)=\lambda(L_{p,\psi}^{\sharp/\flat}(g))$. Combined with \eqref{mu_p1}, this proves part (1). 

For part (2), observe that equation \eqref{lambda_p1} can also be written in the form 
$$
\lambda(\theta_n^\psi(f))=pq_n+ \begin{cases} \lambda(L_{p,\psi}^\sharp(g)) \quad& \text{if $n\gg0$ is even,}\\
p-1+\lambda(L_{p,\psi}^\flat(g)) \quad& \text{if $n\gg0$ is odd.}
\end{cases}
$$
On the other hand, Theorem \ref{theorem_smallweights} implies that
$$
\lambda(\theta_n^\psi(f))=pq_n+\begin{cases}\lambda(L_{p,\psi}^{-}(f))&\quad \text{if $n\gg0$ is even}\\
\lambda(L_{p,\psi}^{+}(f))&\quad \text{if $n\gg0$ is odd.}
\end{cases}
$$ 
Combining these two expressions completes the proof.
\end{proof}

\begin{remark} \phantom{}\nf 
\begin{enumerate}
\item  The sign of the invariants $\lambda^{\pm}(g,\psi)$ and $\mu^{\pm}(g,\psi)$ defined in the above proof are opposite the signs of the invariants defined in \cite{sprung17}. We believe the choice of parity in \cite[Definition 8.8]{sprung17} is likely a typo: the proof of  \cite[Corollary 8.9]{sprung17} cites an unpublished preprint of Iovita, Pollack, and Greenberg, wherein the parity of $\lambda^{\pm}(g,\psi)$ and $\mu^{\pm}(g,\psi)$ are defined as in the above proof. Furthermore, using \cite[Definition 8.8]{sprung17} as stated, if $g$ corresponds to an elliptic curve $E$ with $a_p(E)=0$ then \cite[6.18]{pollack03} (or Theorem \ref{theorem_smallweights}) and \cite[Corollary 8.9]{sprung17} imply $\lambda(L_p^+(E))=\lambda(L_p^-(E))$, which is a contradiction as this certainly need not hold in general. 

\item In \cite{EPW}, Emerton, Pollack, and Weston show that as one varies across pairs of $p$-ordinary and $p$-stabilized newforms $f$ and $g$ for which $\overline{\rho_f}\cong\overline{\rho_g}$ is irreducible, the difference $\lambda(f)-\lambda(g)$ depends only on certain local terms coming from Euler factors at primes dividing the levels of $f$ and $g$. 
See \cite[Theorem 2]{EPW}. In particular, under their assumptions, this difference is independent of the weights of the modular forms. The above corollary shows that this need not be the case for signed $\lambda$-invariants in the $p$-non-ordinary case.
\end{enumerate}
\end{remark}

Table \ref{tablep1} below illustrates the behavior of this corollary at $p=3$. Each block of the table contains a pair of rational newforms $f\in S_4(\Gamma_0(N))$ and $g\in S_2(\Gamma_0(N))$ for which $a_3(f)=a_3(g)=0$ and $\overline{\rho_f}\cong \overline{\rho_g}$. By computing the Mazur-Tate elements attached to both modular forms and using Theorem \ref{theorem_smallweights} to estimate their signed Iwasawa invariants, we note that $\lambda(L_p^+(f))=\lambda(L_p^-(g))+2$ and $\lambda(L_p^-(f))=\lambda(L_p^+(g))$.

\begin{table}[H]
%\caption{$\lambda$-invariants at $p=3$.}
\begin{center}
\scalebox{1}{
\begin{tabular}{|c|c||c|c||c|c|c|c|c|c||c|}
\hline
$k$ 	& $N$ &$\lambda(L_p^+)$& $\lambda(L_p^-)$& 1 & 2 &3&4&5&6& Magma label\\ 
\hline
\hline
4&32&2&0&2&6&20&60&182&546&\texttt{G0N32k4A}\\
2&32&0&0&0&2&6&20&60&182&\texttt{G0N32k2A}\\
\hline
4&154&4&2&2&8&22&62&184&548&\texttt{G0N154k4D}\\ 
2&154&2&2&2&4&8&22&62&184&\texttt{G0N154k2C}\\
\hline
4&256&9&1&1&7&9&61&189&547&\texttt{G0N256k4B}\\
2&256&1&7&1&0&7&27&61&189&\texttt{G0N256k2A}\\
\hline
%%%%%%%%%%%%%%%%%%%%
\end{tabular}
}
\caption{}
%Invariants $\lambda(\theta_{n}(f))$ for $1\leq n\leq 6$ attached to rational newforms $f$ of level $\Gamma_0(N)$ with $a_3=0$. Forms of the same level are congruent mod $3$.}
\label{tablep1}
\end{center}
\end{table}

%%%%%%%%%%%%%%%%%%%%%
\subsection{$p$-adic valuation of critical $L$-values} The goal of this section is to prove Corollary \ref{p1two}. 

 % We resume the notation of \S\ref{mainsection1}. Here we compute the valuation of the critical values of the $L$-function of $f$. This generalizes  \cite[Proposition 6.9]{pollack03} to weights $k>2$. (See also \cite[Theorem 8.5]{sprung17} for an analogous computation in the $p$-non-ordinary case at weight $k=2$ with $a_p\neq 0$.)

%%%%%%%%%
\begin{lemma}\label{ordlog} Let $n\geq 1$. If $n$ is odd then 
$$
\ord_p\big(\log_{k}^+(\gamma^j\zeta_n-1)\big)=(k-1)\bigg(\frac{q_{n}}{p^{n}-p^{n-1}} -\frac{n+1}{2}\bigg),
$$
and if $n$ is even then 
$$
\ord_p\big(\log_{k}^-(\gamma^j\zeta_n-1)\big)=
(k-1)\bigg(\frac{q_{n}}{p^{n}-p^{n-1}} -\frac{n+2}{2}\bigg).
$$
\end{lemma}
\begin{proof}
Equation \eqref{logpm} and Lemma \ref{units}  give 
$$
\ord_p\log_k^{\e_{n+1}}(\gamma^j\zeta_n-1)=-(k-1)(1+\delta_n^{\e_{n+1}})+\ord_p \log_{k,n}^{\e_{n+1}}(\gamma^j\zeta_n-1).
$$
The result now follows from the definition of $\log_{k,n}^{\pm}$ and Lemma \ref{cyclo-j}.%, which 
\end{proof}

%%%%%%%%%
\begin{corollary}\label{ordLval} Let $f\in S_k(\Gamma_1(N))$ be a  newform with $a_p(f)=0$ and let $j\in \{0,\dots, k-2\}$. If $\chi$ is a character on $\Z_p^\times$ of order $p^n$ then for $n\gg0$ we have
\begin{equation}\label{pvalcomp}
\ord_p\bigg(\frac{j!p^{j(n+1)}\tau(\chi)}{\Omega_f^\pm(-2\pi i)^j}L(f_{\chi^{-1}},j+1)\bigg)=\mu(L_{p,\psi\omega^j}^{\star\e_{n+1}}(f))+\frac{(k-1)q_n+\lambda(L_{p,\psi\omega^j}^{\star\e_{n+1}}(f))}{p^n-p^{n-1}}
\end{equation}
where $\pm=\sgn\chi(-1)(-1)^{k-j}$, $\tau(\chi)$ is a Gauss sum, and $\psi$ is the tame part of $\chi$. In particular, if $p>2$ (or $p=2$ and $e\leq 2$) then for all $n\gg0$ we have 
\begin{equation}\label{pvalcomp1}
\ord_p\bigg(C^{\e_n}\binom{k-2}{j}\frac{j!p^{j(n+1)}\tau(\chi)}{\Omega_f^\pm(-2\pi i)^j}L(f_{\chi^{-1}},j+1)\bigg)
=\mu\big(\theta_{n,j}^\psi(f)\big)+\frac{\lambda\big(\theta_{n,j}^\psi(f)\big)}{p^n-p^{n-1}},
\end{equation}
where $C^\pm$ is as in Theorem \ref{lifts}. 
\end{corollary}

\begin{proof} %Throughout this proof, we use the notation of  \cite[\S2]{pollack03}. 
By the interpolation property of the $p$-adic $L$-function (see \cite[Proposition 14]{MTT} or \cite[Proposition 2.11]{pollack03}) it suffices to compute the valuation of $\alpha^{N}L_p(f,\alpha,x^j\chi)$. (Here we are using our convention that $N=n+1$ if $p>2$ and $N=n+2$ if $p=2$.)
Writing $\chi=\psi\chi_{\zeta_n}$, where $\zeta_n=\chi(\gamma)$ is a primitive $p^n$th root of unity and $\psi$ is the tame part of $\chi$, we can decompose $x^j\chi=\psi\omega^j\chi_{\gamma^j\zeta_n}$ into its tame and wild parts (see \eqref{decompw}). By Theorem \ref{theorem_pollackdecomp} and \cite[Corollary 4.2]{pollack03}, we have
\begin{align}\label{dc1}\begin{split}
\alpha^{N}L_p(f,\alpha,x^j\chi)&=\alpha^{N}L_p(f,\alpha,\psi\omega^j,\gamma^j\zeta_n-1)\\
&=\begin{cases}\alpha^{N+1} L_{p,\psi\omega^j}^-(f,\gamma^j\zeta_n-1)\log_{k}^{-\star}(\gamma^j\zeta_n-1) \quad&\text{$N$ odd,}\\
\alpha^{N}L_{p,\psi\omega^j}^+(f,\gamma^j\zeta_n-1)\log_{k}^{\star}(\gamma^j\zeta_n-1) \quad&\text{$N$ even.}
\end{cases}
\end{split}
\end{align}

Let $\lambda^\pm$ and $\mu^\pm$ be the Iwasawa invariants of $L_{p,\psi\omega^j}^\pm(f,T)$. Taking $n\gg0$ so that $\lambda^\pm<(p^n-p^{n-1})/e$, Lemma \ref{mulamfromval} implies 
$$
\ord_pL_{p,\psi\omega^j}^\pm(f,\gamma^j\zeta_n-1)=\mu^\pm+\frac{\lambda^\pm}{p^n-p^{n-1}}.
$$
Equation \eqref{pvalcomp} now follows from \eqref{dc1}, Lemma \ref{ordlog}, and the fact that $\ord_p(\alpha)=(k-1)/2$. 
Equation \eqref{pvalcomp1} is then a direct consequence of \eqref{pvalcomp} and Theorems \ref{theorem_smallweights} and \ref{main}.%, which gives the relation 
\end{proof}

\begin{remark} \phantom{}\nf
\begin{enumerate}
\item Corollary \ref{ordLval} generalizes \cite[Proposition 6.9]{pollack03} to weights $k>2$. See also \cite[Theorem 8.5]{sprung17} for an analogous computation in the $p$-non-ordinary case at weight $k=2$ with $a_p\neq 0$.
\item The proof of Corollary \ref{ordLval} shows that equation \eqref{pvalcomp} holds for all even $n\geq n_0^-$ and all odd $n\geq n_0^+$, where $n_0^-$ (resp., $n_0^+$) is the smallest positive even (resp., odd) integer for which $\lambda^{-\star}<(p^n-p^{n-1})/e$ (resp., $\lambda^{+\star}<(p^n-p^{n-1})/e$). Equation \eqref{pvalcomp1} can also be made effective at weights $k\leq p+1$, where the lower bounds for even/odd $n$ are obtained by taking the maximum of the bounds of the previous sentence and those appearing Theorem \ref{theorem_smallweights}. 
%$k\leq p+1$, then equation \eqref{pvalcomp1} can also be made effective b
\end{enumerate}
 \end{remark}

\subsection{Two conjectures}
\subsubsection{The size of signed $\lambda$-invariants when $k\equiv p+2$ mod $(p^2-1)$.}

We resume the notation of \S\ref{section_largeweights} and assume $p>2$, $\psi=1$, and $j=0$ for simplicity. At weights $k\equiv p+2\Mod (p^2-1)$, Theorem \ref{main} asserts that the $\lambda$-invariants of Mazur-Tate elements can exhibit either $O(p^n)$ or $O(p^{n-1})$ growth depending on the size of $\lambda^\pm:=\lambda(L_p^\pm(f))$. In particular, if it were true that $\lambda^+<\nu^+$ and $\lambda^-\geq \nu^-$ then Theorem \ref{main} would imply
$$
\lambda(\theta_n(f))=\begin{cases}
O(p^n)&\text{for all odd $n\gg0,$}\\
O(p^{n-1})&\text{for all even $n\gg0$}.
\end{cases}
$$
Curiously, we have found no examples of this type of behavior. See for instance the data in Tables \ref{table_p3} and \ref{table_p5} at weights $k=p+2$, where we see precisely one of $O(p^n)$ or $O(p^{n-1})$ growth for \emph{all} $n\gg0$ (irrespective of parity).  This suggests that if one signed $\lambda$-invariant  is small (or large) relative to $\nu^\pm$, then so is the other, motivating the following:

\begin{conjecture} \it Let $f$ be as in Theorem \ref{main} and suppose $k \equiv p+2\Mod (p^2-1)$. Then $\lambda^+<\nu^{+}$ if and only if $\lambda^{-}<\nu^{-}$. 
\end{conjecture}

\subsubsection{$\lambda$-invariants of $p$-congruent modular forms} 

Let $f$ and $g$ be cuspidal newforms of level $\Gamma_1(N)$, $p\nmid N$, and weights $k_f$ and $k_g$, respectively. Assume that the residual representations $\overline{\rho_f}$ and $\overline{\rho_g}$ are isomorphic. When $k_g=2$, Pollack and Weston \cite{PW} use a version of mod $p$ multiplicity one to relate (under some assumptions) the modular symbols attached to $f$ and $g$ and ultimately  deduce a congruence relation between the corresponding Mazur-Tate elements. This congruence has the form 
$$
 \varpi^{-a^\pm}\theta_n(f)\equiv \cor_{n-1}^n \theta_{n-1}(g) \Mod \varpi \Oo[\Gamma_n]
$$
for some nonnegative integers $a^+$ and $a^-$ (depending on $f$ and the parity of $n$), from which one can deduce 
\begin{equation}\label{lambdaPW}
\lambda(\theta_n(f))=p^n-p^{n-1}+\lambda(\theta_{n-1}(g)), \qquad n\gg0.
\end{equation}
See \cite[Theorems 5.1 and 6.1]{PW} for detailed statements.

%A strong enough version of mod $p$ multiplicity one might allow one to relate the modular symbols attached to $f$ and $g$ (as in \cite{PW}), from which one could deduce a relation between the corresponding Mazur-Tate elements.

In the case $k_g>2$, if we assume $a_p(f)=a_p(g)=0$ and $k_f\equiv k_g\Mod (p^2-1)$ (in addition to the assumption $\overline{\rho_f}\cong \overline{\rho_g}$), computations suggest (see Table \ref{table_cong2}) that 
\begin{equation}\label{lambdaconj}
\lambda(\theta_n(f))=\lambda(\theta_n(g)), \qquad n\gg0. 
\end{equation}
Indeed, computer calculations indicate that under the stated assumptions one can choose periods $\Omega_f^\pm$ and $\Omega_g^\pm$ such that
$$
\varpi^{-\mu^\pm(f)} \theta_n(f)\equiv \varpi^{-\mu^\pm(g)}\theta_n(g) \Mod \varpi \Oo[\Gamma_n],
$$
from which the above equality between the $\lambda$-invariants of Mazur-Tate elements would follow. In particular, if \eqref{lambdaconj} holds then Theorem \ref{main1}
would force a relation between the signed $\lambda$-invariants for $f$ and $g$, summarized in the following conjecture. See Table \ref{table_cong2} for examples.

\begin{conjecture}\label{conj_congruence} Let $f\in S_{k_f}(\Gamma_1(N))$ and $g\in S_{k_g}(\Gamma_1(N))$ be newforms with $a_p(f)=a_p(g)=0$. Suppose that 
\begin{enumerate}
\item[(i)] $k_g>2$,
\item[(ii)]  $k_f\equiv k_g\Mod (p^2-1)$, and
\item[(iii)] $\overline{\rho_f}\cong \overline{\rho_g}$.
\end{enumerate}
 If $k_g\not\equiv p+2\Mod (p^2-1)$, or if $k_g\equiv p+2\Mod (p^2-1)$ and $\lambda(L_p^\pm(\dagger))<\nu^\pm_{k_{\dagger}}$ or $\lambda(L_p^\pm(\dagger))\geq\nu^\pm_{k_{\dagger}}$ for both $\dagger\in \{f,g\}$, then 
 \begin{align*}
   \lambda(L_p^+(f))- \lambda(L_p^+(g))&=\frac{p(k_f-k_g)}{p+1},\\
 \lambda(L_p^-(f))- \lambda(L_p^-(g))&=\frac{k_f-k_g}{p+1}.
 \end{align*}
 \end{conjecture}

\begin{remark} \nf Relations between signed Iwasawa invariants for $p$-congruent modular forms are known in the equal-weight case -- 
see \cite{HL19} and \cite{CL}. For example, when $2<k_f=k_g<p+1$ and  the extension of $\Q_p$ generated by the Fourier coefficients of $f$ and $g$ is unramified, Corpuz and Lei \cite{CL} show that $\mu(L_p^{\sharp/\flat}(f))=0$ if and only if $\mu(L_p^{\sharp/\flat}(g))=0$, and if either condition holds then $\lambda(L_p^*(f))=\lambda(L_p^*(g))$ for $*\in \{\sharp,\flat\}$, assuming that $f$ and $g$ have the same level. This is consistent with the above conjecture. %(When the levels of $f$ and $g$ are not the same, results of \cite{HL19} and \cite{CL} show that the invariants $\lambda(L_p^*(f))$ and $\lambda(L_p^*(g))$ may differ by local factors coming from primes dividing the levels, mirroring results of \cite{Greenberg-Vatsal} and \cite{EPW} in the ordinary case.)
\end{remark}

\begin{table*}[htp]
%\caption{$\lambda$-invariants at $p=3$.}
\begin{center}
\scalebox{1}{
\begin{tabular}{|c|c||c|c||c|c|c|c|c|c|c|}
\hline
$k$ 	& $N$ &$\lambda(L_p^+)$& $\lambda(L_p^-)$ & 1 & 2 &3&4&5&6& label\\ \hline\hline
%%%%%%%%%%%%%%%%%%%%
2&32&0&0&0&2&6&20&60&182&\texttt{G0N32k2A}\\
10&32&6&2&2&8&24&74&222&668&\texttt{G0N32k10A}\\
18 &32&12&4&2&8&24&74&222&668  &\texttt{G0N32k18A}\\ 
\hline
3&7&1&0&1&4&13&40&121&364&\texttt{G1N7k3A}\\
11&7&7&2&1&4&13&40&121&364&\texttt{G1N7k11B}\\
19&7&13 &  4 &1&4&13&40&121&364&\texttt{G1N7k19B} \\
\hline
4&32&2&0&2&6&20&60&182&546&\texttt{G0N32k4A}\\
12&32&8&2&2&6&20&60&182&546&\texttt{G0N32k12A}\\ 
  20&32&14&4&2&6&20&60&182&546&\texttt{G0N32k20A}\\
  \hline
  5&4&2&0&2&8&26&80&242&728&\texttt{G1N4k5A}\\
13&4&8&2&2&8&26&80&242&728&\texttt{G1N4k13A}\\
21&4&14&4&2&8&26&80&242&728&\texttt{G1N4k21A}\\
\hline
\end{tabular}
}
\caption{Invariants $\lambda(\theta_n)$, $1\leq n\leq 6$, attached to rational newforms of level $\Gamma_0(N)$ with $a_3=0$. The signed $\lambda$-invariants are those predicted from the Mazur-Tate elements by Theorems \ref{theorem_smallweights} and \ref{main}. Modular forms in the same block are congruent modulo $p=3$ and, at weights $k>2$, have signed invariants satisfying Conjecture \ref{conj_congruence}.}
\label{table_cong2}
\end{center}
\end{table*}

\appendix
%%%%%%%%%%%%%%%%%%%%%%%%%%%%%%%%%%%%%%%%%%%%%%%%%%%%%%%%%%%%%%%%%%%%%%%%%%%%%%%%%%%%%%%%%%%%%%%%%%%%%%%%%%%%
\section{Lemmas}\label{appendix_lemmas}
%%%%%%%%%%%%%%%%%%%%%%%%%%%%%%%

%%%%%%
\begin{lemma}\label{additive} Let $\mu$ be any finitely additive function on subsets of $\Z_p^\times$ and let $L_n(T)=\sum_{a\in (\Z/p^{N}\Z)^\times} \mu(a+p^{N}\Z_p)(1+T)^{\log_\gamma(a)}$. For all $0\leq m\leq n$ we have 
$L_n(\zeta_m-1)=L_m(\zeta_m-1)$.
\end{lemma}
\begin{proof} We prove the case where $p$ is odd, the case $p=2$ is similar. Since $a\equiv b\Mod p^{m+1}$ implies $\log_{\gamma}(a)\equiv\log_\gamma(b)\Mod p^m$ and 
$$
\bigcup_{\substack{a\in (\Z/p^{n+1}\Z)^\times\\ a\equiv b\Mod p^{m+1}}}a+p^{n+1}\Z_p=b+p^{m+1}\Z_p,
$$
where the union is disjoint, we have
\begin{align*}
L_{n}(\zeta_m-1) &=\sum_{a\in (\Z/p^{n+1}\Z)^\times} \mu(a+p^{n+1}\Z_p)\zeta_m^{\log_\gamma(a)}\\
&=\sum_{b\in (\Z/p^{m+1}\Z)^\times} \sum_{\substack{a\in (\Z/p^{n+1}\Z)^\times\\ a\equiv b\Mod p^{m+1}}}\mu(a+p^{n+1}\Z_p)\zeta_m^{\log_\gamma(b)} \\
&=\sum_{b\in (\Z/p^{m+1}\Z)^\times} \mu(b+p^{m+1}\Z_p)\zeta_m^{\log_{\gamma}(b)}\\
&= L_{m}(\zeta_m-1).
\end{align*}
\end{proof}

%%%%%%
\begin{lemma} \label{cyclo-j} For any $u\in 1+2p\Z_p$ and $n\geq 0$ we have 
$$
\ord_p\big(\Phi_{m}(u\zeta_n)\big)=
\begin{cases}
\frac{p^m-p^{m-1}}{p^{n}-p^{n-1}} &\quad\text{if $m<n$,}\\
1 &\quad\text{if $m>n$,}
\end{cases}
$$
\end{lemma}
\begin{proof} First, assume $1\leq m<n$. Observe that 
$$
\Phi_{m}(u\zeta_n)=\frac{(u\zeta_n)^{p^m}-1}{(u\zeta_n)^{p^{m-1}}-1}=u^{(p-1)p^{m-1}}\bigg(\frac{\zeta_{n-m}-u^{-p^m}}{\zeta_{n-m+1}-u^{-p^{m-1}}}\bigg).
$$
Note that $u^{-p^m}\in 1+2p^{m+1}\Z_p$, so $u^{-p^m}=1+2p^{m+1}u'$ for some $u'\in \Z_p$. Hence,
\begin{align*}
\ord_p(\zeta_{n-m}-u^{-p^m})&=\ord_p(\zeta_{n-m}-1-2p^{m+1}u')\\
&=\min\big(\ord_p(\zeta_{n-m}-1),\ord_p(2p^{m+1}u')\big)\\
&= \frac{1}{p^{n-m}-p^{n-m-1}},
\end{align*}
and the claimed valuation of $\Phi_{m}(u\zeta_n)$ follows.
Now suppose $m>n$. Then% $\zeta_n^{p^{m-1}}=1$ and we have 
$$
\Phi_m(u\zeta_n)=1+\sum_{t=1}^{p-1}u^{tp^{m-1}}
\equiv 1+(p-1)\Mod 2p^m
\equiv p \Mod 2p^m,
$$
because $u^{tp^{m-1}}\in 1+2p^{m}\Z_p$. % for $1\leq t\leq p-1$. 
Thus $\Phi_m(u\zeta_n)\in p+2p^m\Z_p$ and it follows that $\ord_p\big(\Phi_m(u\zeta_n)\big)=1$ for $m\geq  2$. If $m=1$ (hence $n=0$) one can check directly that $\Phi_1(u)=p\Mod p^2$. 
\end{proof}

%%%%%%
\begin{lemma}\label{basiclem} Let $N\geq 0$. 

(1) $\ord_p c_N^{(i)} \geq 1$ for all $i\leq p^n-1$. 

(2) If $N+1\leq i\leq p^n-1$ then $\ord_p c_N^{(i)} =n-\ord_p(i-N)$. 

(3) If $i\leq N$ then $c_N^{(i)} =\sum_{j=0}^{i-1} c_{N-1-j}^{(p^n-1)}c_0^{(i-j)}$.

\end{lemma}
\begin{proof}\phantom{}
\begin{enumerate}
\item  If $i\leq 0$ then  $c_N^{(i)} = 0$ and we are done. If $1\leq i\leq p^n-1$ and $N=0$, the result follows from the identity 
\begin{equation}\label{vbinom}
 \ord_p\binom{p^n}{i}=n-\ord_p(i).
\end{equation}
%The rest follows from induction (use the recursion relation). 
Now suppose (1) holds for some $N\geq 0$ and all $1\leq i\leq p^n-1$. Then \eqref{recursion} yields
$$
\ord_p c_{N+1}^{(i)}\geq \min(\ord_p   c_N^{(i-1)},\ord_p   c_N^{(p^n-1)}+\ord_pc_0^{(i)})\geq 1.
$$
\item  If $N=0$ then the result follows from \eqref{vbinom}. Suppose it holds for some $N \geq 0$. Let $i\geq N+2$. Then $\ord_p c_{N}^{(i-1)}=n-\ord_p(i-1-N)$ and $\ord_p(c_{N}^{(p^n-1)}c_0^{(i)})=2n-\ord_p(p^n-1-N)-\ord_pi$. Since $i\leq p^n-1$ one can check that $n-\ord_p(i-1-N)<2n-\ord_p(p^n-1-N)-\ord_pi$. Therefore, \eqref{recursion} yields
$$
\ord_p c_{N+1}^{(i)}=\min\big(\ord_p   c_N^{(i-1)},\ord_p   (c_N^{(p^n-1)}c_0^{(i)})\big)=n-\ord_p(i-(N+1)). 
$$
\item Apply \eqref{recursion} $i$ times.% to $c_N^{(i)}$.
\end{enumerate}
\end{proof}

%%%%%%
\begin{lemma}\label{lemma1} Let $F=\sum a_iT^i\in \Lambda_K$ and suppose $\mu(F)=0$. Let $s_i=\sum_{N\geq 0}a_{p^n+N}c_{N}^{(i)}$.  If 
\begin{enumerate}[label=(\roman*)]
\item $\lambda(F)=p^n$, or
\item $\lambda(F)\geq p^n+1$ and $\ord_p a_i\geq v_i(F)$ for all $p^n \leq i\leq \lambda-1$,
\end{enumerate}
then 
$$
\ord_ps_i \begin{cases}> Q(F)+1& \text{if $0\leq i\leq R(F)+ p^{n-1}-1,$}\\ 
= Q(F)+1 & \text{if $i=R(F)+p^{n-1},$}\\
\geq Q(F)+1& \text{if $R(F)+p^{n-1}+1\leq i\leq p^n-1$.}
\end{cases}
$$
\end{lemma}

\begin{proof}  Let $\lambda=\lambda(F)$, $\mu=\mu(F)$, and $N_F=\lambda-p^n$. Consider the following claim. 

\noindent \bf Claim. \it If $N\neq N_F$ then 
 \begin{enumerate}[label=(\theenumi)]
\item $\ord_p a_{p^n+N}c_N^{(i)}> Q(F)+1$ for all $0\leq i\leq R(F)+p^{n-1},$
\item $\ord_p a_{p^n+N}c_N^{(i)}\geq Q(F)+1$ for all $R(F)+p^{n-1}+1\leq i\leq p^n-1.$ 
\end{enumerate}\nf

\noindent Assuming the claim, the lemma is proved as follows. 
 By Proposition \ref{lbounds} and the fact that $a_{\lambda}$ is a unit (we are assuming $\mu=0$) we know 
\begin{equation}\label{alamb}
\ord_p a_{\lambda}c_{N_F}^{(i)}\begin{cases}\geq Q(F)+2& \text{if $0\leq i\leq R(F)+ p^{n-1}-1$}\\ 
= Q(F)+1 & \text{if $i=R(F)+p^{n-1}$}\\
\geq Q(F)+1& \text{if $R(F)+p^{n-1}+1\leq i\leq p^n-1$.}
\end{cases}
\end{equation}
Thus, 
$$
\ord_p s_{R(F)+p^{n-1}}= \min \bigg(\ord_p a_{\lambda}c_{N_F}^{(R(F)+p^{n-1})}, \ord_p\sum_{N\neq N_F}a_{p^n+N}c_N^{(R(F)+p^{n-1})} \bigg)=Q(F)+1,
$$
and for $i\neq R(F)+p^{n-1}$, 
$$
\ord_p s_i\geq \min_{N\geq 0}\big(\ord_p a_{p^n+N}c_N^{(i)}\big) \begin{cases} >Q(F)+1& \text{if $0\leq i\leq R(F)+ p^{n-1}-1,$}\\ 
\geq Q(F)+1& \text{if $R(F)+p^{n-1}+1\leq i\leq p^n-1$,}
\end{cases}
$$
%So it remains to prove the claim. 

\vspace{1mm}

\noindent \it Proof of the Claim. \nf If $N\geq N_F+1$ then $Q(N)\geq Q(F)$ and the claim follows from Proposition \ref{lbounds}.
Now suppose $N\leq N_F-1$ (hence $\lambda> p^n$). Then our hypothesis on the valuations of $a_i$ together with Proposition \ref{lbounds} yields the following for all $0\leq i\leq p^n-1$:
\begin{align*}
\ord_pa_{p^n+N}c_N^{(i)} &=\ord_pa_{p^n+N}+\ord_p c_N^{(i)}\\
&\geq \bigg(Q(F)-Q(N)+
\left\{\!\begin{aligned}
&1/e &&\text{if $R(N)\leq R(F)$ } \\[1ex]
&0 &&\text{ if $R(N)>R(F)$}\\[1ex] 
\end{aligned}\right\} 
\bigg)+\bigg(Q(N)+1\bigg)\\
&=Q(F)+1+\begin{cases}
1/e& \text{if $R(N)\leq R(F)$} \\
0& \text{if $R(N)>R(F)$}
\end{cases}
\end{align*}
It remains to show part (1) of the claim when $R(N)>R(F)$, but this follows directly from (the first inequality in) Proposition \ref{lbounds}. 
\end{proof}

%%%%%%%%
\begin{lemma}\label{lemv} Let $F\in \Lambda_K$ and let $f_F$ be as in Definition \ref{bur}. If  $\lambda(F)\geq p^n$ then $f_F(i)>v_i(F)-\frac{1}{e}$ for all $0\leq i\leq \lambda(F)-1$.
\end{lemma}
\begin{proof} We may assume $\mu=0$. Let $Q=Q(F)$, $R=R(F)$, $R_i=R(i-p^n)$, $Q_i=Q(i-p^n)$, and $v_i=v_i(F)$ (see Definition \ref{pldef}). The case where $Q=0$ follows directly from the definitions, so assume $Q>0$. There are three cases:
\begin{enumerate}
\item[(1)] $i\in [0,R+p^{n-1}],$
\item[(2)] $i\in [R+p^{n-1}+1, p^n-1],$ 
\item[(3)]  $i\in [p^n, \lambda-1]$.
\end{enumerate}
For convenience, define
$$
X(i)=\frac{(Q+\frac{1}{e} )(i-R-p^{n-1})}{Qt_n+1}.
$$

\begin{enumerate}

\item This is clear as $v_i-\frac{1}{e}=Q+1<Q+2=f_F(i)$.

\item   In this case, we have $v_i=Q+1$ and $f_F(i)=Q+2-X(i)$, so it suffices to show $X(i)<1+\frac{1}{e}$. 
Using the bounds on $i$ and the definition of $R$, we have
\begin{align*}
X(i)&<\frac{(Q+\frac{1}{e} )(i-R-p^{n-1})}{Qt_n}\\
&\leq \frac{(Q+\frac{1}{e} )((p^n-1)-0-p^{n-1})}{Qt_n}\\
&< \frac{(Q+\frac{1}{e})t_n}{Qt_n}\\
&\leq 1+\frac{1}{e}.
\end{align*}

\item  In this case, we have $v_i=Q-Q_i+\begin{cases}\frac{1}{e}& \text{if $R_i\leq R$}\\ 
  0& \text{if $R_i> R$}.
  \end{cases}$. 
  
\noindent If $i\leq \lambda-t_n$ then $f_F(i)=Q+2-X(i)$ and the statement follows from 
\begin{align*}
X(i)-Q_i%&=\frac{(Q(F)+1)(i-R(F)-p^{n-1})}{Q(F)t_n+1}-\frac{i-p^n-R(i-p^n)}{t_n}\\
&<\frac{(Q+\frac{1}{e})(i-R-p^{n-1})-Q(i-p^n-R_i)}{Qt_n}\\
&=1+\frac{\frac{1}{e}(i-p^{n-1}-R)+Q(R_i-R)}{Qt_n}\\
&\leq 1+\frac{\frac{1}{e}(\lambda-t_n-p^{n-1}-R)-Q(R-R_i)}{Qt_n}\\
&=1+\frac{1}{e}-\frac{R-R_i}{t_n}\\
&\leq
\begin{cases}2& \text{if $R_i\leq R$}\\ 
  2+\frac{1}{e}& \text{if $R_i> R$}.
  \end{cases}
\end{align*}
\noindent If $i\geq \lambda-t_n+1$ then $f_F(i)=1$ and the statement follows from the fact that $ Q_F-Q_i\leq1$.
\end{enumerate}
\end{proof}

%%%%%%%
\begin{lemma}\label{prodlem} For integers $m\geq 1$ and $q\geq0$ we have 
$$
\prod_{i=1}^{m}(T^q+pY_i)=\sum_{i=0}^{m}p^{i}T^{(m-i)q}e_{i}(Y_1,\dots, Y_{m}),
$$
where $e_{i}(Y_1,\dots, Y_{m})$ denotes the $i$\textsuperscript{th} elementary symmetric polynomial in the $m$ variables $Y_1,\dots, Y_{m}$. 
\end{lemma}
\begin{proof} This follows from induction and the recursion relation 
$$
e_{i}(Y_1,\dots, Y_{m+1})=e_{i}(Y_1,\dots, Y_{m})+Y_{m+1}e_{i-1}(Y_1,\dots, Y_{m}).
$$
\end{proof}

%%%%%%%
\begin{lemma}\label{boundNP} Using the notation of \S\ref{plargeiw}, let $G_{k,n} = p^\mu(p^{k-1}+pT^{(k-2)q_n}+pT^{\lambda_n-1}+T^{\lambda_n})$. Then 
$
 \NP_p\big(\T_{\lambda_n-1} (LF_{k,n})\big)\subseteq \NP_p\big(\T_{\lambda_n-1}(G_{k,n} )\big)
 $
 \end{lemma}
\begin{proof} We may assume $\mu=0$. Write $LF_{k,n}=\sum_{i=0}^{\lambda_n} a_iT^i$. It suffices to prove that 
\begin{enumerate}
\item $\ord_p a_i\geq k-1-\frac{i}{q_n}$ for all $0\leq i\leq (k-2)q_n$, and 
\item $\ord_p a_i\geq 1$ for all $(k-2)q_n\leq i\leq \lambda_n-1$. 
\end{enumerate}
 Using Lemma \ref{prodlem}, we can write
\begin{equation}\label{Fkndef}
F_{k,n}=\sum_{i=0}^{k-1}p^{k-1-i}T^{iq_n}e_{k-1-i}(W_0,\cdots W_{k-2}).
\end{equation}
 Letting $e_{i}=e_{i}(W_0,\cdots W_{k-2})$, it follows that 
\begin{align*}
LF_{k,n} &\equiv T^{q_n}\big(L\sum_{i=1}^{k-1}p^{k-1-i}T^{(i-1)q_n}e_{k-1-i}\big)\Mod p^{k-1}\\
&\equiv T^{2q_n}\big(L\sum_{i=2}^{k-1}p^{k-1-i}T^{(i-2)q_n}e_{k-1-i}\big)\Mod p^{k-2}\\
& \vdots \\
&\equiv T^{(k-2)q_n}\big(L\sum_{i=k-2}^{k-1}p^{k-1-i}T^{(i-k+2)q_n}e_{k-1-i}\big)\Mod p^{2}\\
&\equiv T^{(k-1)q_n+\lambda}\Mod p.
\end{align*}
Thus, the above congruences yield the inequalities:
\begin{equation}\label{aik}
\ord_pa_i\geq  \begin{cases}
k-1& \text{for $0\leq i\leq  q_n-1,$}\\ 
k-2& \text{for $q_n\leq i\leq 2q_n-1,$}\\ 
\vdots & \vdots\\
2& \text{for $(k-3)q_n\leq i\leq (k-2)q_n-1,$}\\ 
1 & \text{for $(k-2)q_n \leq i\leq \lambda_n-1$.}
\end{cases}
\end{equation}
The result follows. 
\end{proof}

%%%%%%%%%%%%%%%%%%%%%%%%%%%%%%%%%%%%%%%%%%%%%%%%%%%%%%%%%%%%%%%%%%%%%%%%%%%%%%%%%%%%%%%%%%%%%%%%%%%
\section{Tables of $\lambda$-invariants} \label{tablessection}

The following tables contain the invariants $\lambda(\theta_n(f))=\lambda(\theta_{n,j=0}^{\psi=1}(f))$ at primes $p\in\{2,3,5\}$ and the (conjectural) invariants $\lambda^\pm=\lambda(L_{p}^\pm(f))$ implied (using the same idea as in Example \ref{compiwex}) by Theorems \ref{theorem_smallweights} and \ref{main}. We include data on all newforms $f\in S_k(N,\e)$, $p\nmid N$, with $a_p(f)=0$, degree $d=[K_f:\Q]$, and weight $2\leq k\leq 10$ under the following constraints:

\begin{enumerate}[label=(\alph*)]

\item If $k$ is even, $d=1$, and $\e=1$, then we include all newforms of level $N\leq 300$.

\item If either (i) $k$ is even and $2\leq d \leq 6$ or (ii) $k$ is odd and $1\leq d\leq 6$, then we include all newforms of level $N\leq 40$. 

\end{enumerate}

\noindent Omitted from these tables are data on newforms  of weight $k=2$ with $d=1$, since these forms correspond to isogeny classes of rational elliptic curves whose Iwasawa invariants can be found on LMFDB, for example. 

The column titled `label' contains the Magma  label for the modular form, which can be recovered in Magma using the command \texttt{Newform("label")}. In all of our data, when $d>1$ the prime $p$ splits into at most two distinct prime factors in $K_f$ (both of which have ramification index $\leq 2$ when $p=2$, so our results apply). To distinguish between these primes, we simply write ``$d^*$'' to indicate that our computations are with respect to the second prime $\Pp$ lying over $p$, ordered according to the Magma function \texttt{Decomposition(O,p)}. The code used to carry out these computations is available upon request.

\newpage

\begin{table}
    \caption{$\lambda$-invariants at $p=2$.}
        \label{table_p2}
    \begin{minipage}[t]{.54\linewidth}
 \scalebox{.68}{
\begin{tabular}{|ccl||cc||c|c|c|c|c|c|c|c||c|}
\hline
$k$ 	& $N$ &$d$&$\lambda^+$& $\lambda^-$&  0&1 & 2 &3&4&5&6&7 &label\\ \hline 
2&21&2&0&3&0&0&1&5&5&13&21&45&\texttt{G1N21k2B}\\
2&33&2&0&1&0&1&1&3&5&11&21&43&\texttt{G1N33k2E}\\
2&39&4&0&1&0&1&1&3&5&11&21&43&\texttt{G1N39k2C}\\
2&39&2&0&1&0&1&1&3&5&11&21&43&\texttt{G1N39k2E}\\
\hline
3&11&1&1&1&0&1&3&5&11&21&43&85&\texttt{G1N11k3B}\\
3&19&1&3&1&0&1&2&5&13&21&45&85&\texttt{G1N19k3C}\\
3&21&2&3&1&0&1&2&5&13&21&45&85&\texttt{G1N21k3E}\\
3&27&1&5&1&0&1&2&5&15&21&47&85&\texttt{G1N27k3C}\\
3&35&1&5&1&0&1&2&5&15&21&47&85&\texttt{G1N35k3G}\\
3&35&1&5&1&0&1&2&5&15&21&47&85&\texttt{G1N35k3H}\\
3&39&2&1&1&0&1&3&5&11&21&43&85&\texttt{G1N39k3D}\\
3&39&2&1&1&0&1&3&5&11&21&43&85&\texttt{G1N39k3G}\\
\hline
4&9&1&1&2&0&0&2&4&8&16&32&64&\texttt{G0N9k4A}\\
4&21&2&2&3&0&1&3&5&9&17&33&65&\texttt{G1N21k4G}\\
4&33&2&1&2&0&1&2&4&8&16&32&64&\texttt{G1N33k4I}\\
4&39&1&1&2&0&0&2&4&8&16&32&64&\texttt{G0N39k4A}\\
4&39&2&1&2&0&1&2&4&8&16&32&64&\texttt{G1N39k4E}\\
4&39&4& 1&2&0&1&2&4&8&16&32&64&\texttt{G1N39k4H}\\
4&53&1&2&3&0&1&3&5&9&17&33&65&\texttt{G0N53k4A}\\
4&95&1&1&2&0&0&2&4&8&16&32&64&\texttt{G0N95k4A}\\
4&105&1&1&2&0&0&2&4&8&16&32&64&\texttt{G0N105k4A}\\
4&117&1&2&7&0&1&3&7&9&21&33&69&\texttt{G0N117k4A}\\
4&121&1&1&2&0&0&2&4&8&16&32&64&\texttt{G0N121k4A}\\
4&165&1&2&3&0&0&3&5&9&17&33&65&\texttt{G0N165k4A}\\
4&189&1&2&3&0&1&3&5&9&17&33&65&\texttt{G0N189k4A}\\
4&189&1&2&3&0&0&3&5&9&17&33&65&\texttt{G0N189k4B}\\
4&225&1&3&4&0&0&1&6&10&18&34&66&\texttt{G0N225k4A}\\
4&243&1&2&9&0&1&3&3&9&23&33&71&\texttt{G0N243k4A}\\
4&243&1&8&3&0&1&3&5&15&17&39&65&\texttt{G0N243k4B}\\
\hline
5&11&1&1&3&0&1&3&7&13&27&53&107&\texttt{G1N11k5B}\\
5&19&1&1&5&0&1&3&5&13&29&53&109&\texttt{G1N19k5C}\\
5&21&2&1&5&0&1&3&5&13&29&53&109&\texttt{G1N21k5E}\\
5&27&1&1&7&0&1&3&6&13&31&53&111&\texttt{G1N27k5C}\\
5&35&1&1&7&0&1&3&6&13&31&53&111&\texttt{G1N35k5F}\\
5&35&1&1&7&0&1&3&6&13&31&53&111&\texttt{G1N35k5G}\\
5&39&2&1&3&0&1&3&7&13&27&53&107&\texttt{G1N39k5D}\\
5&39&2&1&3&0&1&3&7&13&27&53&107&\texttt{G1N39k5G}\\
\hline
6&21&2&4&3&0&1&3&5&13&21&45&85&\texttt{G1N21k6E}\\
6&27&1&6&3&0&1&3&5&15&21&47&85&\texttt{G0N27k6A}\\
6&33&2&2&3&0&1&3&5&11&21&43&85&\texttt{G1N33k6I}\\
6&39&4&2&3&0&1&3&5&11&21&43&85&\texttt{G1N39k6E}\\
6&121&1&3&4&0&1&2&6&12&22&44&86&\texttt{G0N121k6A}\\
6&225&1&3&4&0&1&2&6&12&22&44&86&\texttt{G0N225k6A}\\
6&225&1&3&4&0&1&2&6&12&22&44&86&\texttt{G0N225k6B}\\
6&243&1&4&3&0&1&3&5&13&21&45&85&\texttt{G0N243k6A}\\
6&243&1&2&11&0&1&3&7&11&29&43&93&\texttt{G0N243k6B}\\
\hline
7&3&1&1&3&0&1&3&7&15&31&63&127&\texttt{G1N3k7A}\\
7&11&1&1&3&0&1&3&7&15&31&63&127&\texttt{G1N11k7B}\\
7&19&1&1&3&0&1&3&7&15&31&63&127&\texttt{G1N19k7C}\\
7&35&1&3&5&0&1&3&5&9&17&33&65&\texttt{G1N35k7F}\\
7&35&1&3&5&0&1&3&5&9&17&33&65&\texttt{G1N35k7G}\\
7&39&2&2&4&0&1&2&4&8&16&32&64&\texttt{G1N39k7G}\\
\hline
\end{tabular}
}
\end{minipage}%
\begin{minipage}[t]{.8\linewidth}
\vspace{-77mm}
\hspace{1mm}
\scalebox{.7}{
\begin{tabular}{|ccl||cc||c|c|c|c|c|c|c|c||c|}
\hline
$k$ 	& $N$ &$d$&$\lambda^+$& $\lambda^-$& 0& 1 & 2 &3&4&5&6&7& label\\ \hline 
8&21&2&2&7&0&1&3&5&13&29&53&109&\texttt{G1N21k8E}\\
8&27&1&2&9&0&1&3&7&13&31&53&111&\texttt{G0N27k8A}\\
8&33&2&2&5&0&1&3&7&13&27&53&107&\texttt{G1N33k8I}\\
8&39&4&2&5&0&1&3&7&13&27&53&107&\texttt{G1N39k8E}\\
8&27&1&2&9&0&1&3&7&13&31&53&111&\texttt{G0N27k8A}\\
8&121&1&3&6&0&1&2&4&14&28&54&108&\texttt{G0N121k8A}\\
8&225&1&3&6&0&1&2&4&14&28&54&108&\texttt{G0N225k8A}\\
8&225&1&3&6&0&1&2&4&14&28&54&108&\texttt{G0N225k8B}\\
8&243&1&10&5&0&1&3&7&13&27&61&107&\texttt{G0N243k8A}\\
8&243&1&2&7&0&1&3&5&13&29&53&109&\texttt{G0N243k8B}\\
\hline
9&11&1&3&5&0&1&3&5&11&21&43&85&\texttt{G1N11k9B}\\
9&19&1&5&5&0&1&2&5&13&21&45&85&\texttt{G1N19k9C}\\
9&21&2&5&5&0&1&2&5&13&21&45&85&\texttt{G1N21k9D}\\
9&27&1&7&5&0&1&2&5&15&21&47&85&\texttt{G1N27k9C}\\
9&35&1&7&5&0&1&2&5&15&21&47&85&\texttt{G1N35k9F}\\
9&35&1&7&5&0&1&2&5&15&21&47&85&\texttt{G1N35k9G}\\
9&39&2&3&5&0&1&3&5&11&21&43&85&\texttt{G1N39k9D}\\
9&39&2&3&5&0&1&3&5&11&21&43&85&\texttt{G1N39k9G}\\
\hline
10&9&1&3&6&0&1&2&4&8&16&32&64&\texttt{G0N9k10A}\\
10&21&2&4&7&0&1&3&5&9&17&33&65&\texttt{G1N21k10H}\\
10&33&2&4&6&0&1&2&4&8&16&32&64&\texttt{G1N33k10H}\\
10&39&4&4&6&0&1&2&4&8&16&32&64&\texttt{G1N39k10H}\\
10&121&1&3&6&0&1&2&4&8&16&32&64&\texttt{G0N121k10A}\\
10&225&1&5&8&0&1&3&6&10&18&34&66&\texttt{G0N225k10A}\\
10&243&1&10&7&0&1&3&5&15&17&39&65&\texttt{G0N243k10A}\\
10&243&1&4&13&0&1&3&7&9&23&33&71&\texttt{G0N243k10B}\\
\hline
\end{tabular}
}
\end{minipage} 
\end{table}
\clearpage

\newpage
\begin{table}
    \caption{$\lambda$-invariants at $p=3$.}
    \label{table_p3}
    \begin{minipage}[t]{.54\linewidth}
 \scalebox{.7}{
\begin{tabular}{|ccl||cc||c|c|c|c|c|c|c||c|}
\hline
$k$ 	& $N$ &$d$&$\lambda^+$& $\lambda^-$& 0& 1 & 2 &3&4&5&6& label\\ \hline 
2&20&2&0&0&0&0&2&6&20&60&182&\texttt{G1N20k2B}\\
2&26&2&0&0&0&0&2&6&20&60&182&\texttt{G1N26k2C}\\
2&28&2&0&0&0&0&2&6&20&60&182&\texttt{G1N28k2C}\\
2&34&2&0&0&0&0&2&6&20&60&182&\texttt{G1N34k2C}\\
2&37&2&1&5&0&1&7&7&25&61&187&\texttt{G1N37k2G}\\ 
\hline
3&7&1&1&0&0&1&4&13&40&121&364&\texttt{G1N7k3A}\\
3&16&1&1&0&0&1&4&13&40&121&364&\texttt{G1N16k3A}\\
3&19&1&1&0&0&1&4&13&40&121&364&\texttt{G1N19k3C}\\
3&20&2&1&0&0&1&4&13&40&121&364&\texttt{G1N20k3E}\\
3&26&2&1&0&0&1&4&13&40&121&364&\texttt{G1N26k3C}\\
3&31&3&1&0&0&1&4&13&40&121&364&\texttt{G1N31k3F}\\
3&40&1&1&0&0&1&4&13&40&121&364&\texttt{G1N40k3E}\\
3&40&1&1&0&0&1&4&13&40&121&364&\texttt{G1N40k3F}\\
\hline
4&20&2&2&0&0&2&6&20&60&182&546&\texttt{G1N20k4B}\\
4&28&2&2&0&0&2&6&20&60&182&546&\texttt{G1N28k4E}\\
4&32&1&2&0&0&2&6&20&60&182&546&\texttt{G0N32k4A}\\
4&49&1&2&0&0&2&6&20&60&182&546&\texttt{G0N49k4D}\\
4&64&1&2&0&0&2&6&20&60&182&546&\texttt{G0N64k4A}\\
4&154&1&4&2&0&2&8&22&62&184&548&\texttt{G0N154k4D}\\ 
4&254&1&5&1&0&1&7&23&61&185&547&\texttt{G0N254k4A}\\
4&256&1&2&0&0&2&6&20&60&182&546&\texttt{G0N256k4A}\\
4&256&1&9&1&0&1&7&9&61&189&547&\texttt{G0N256k4B}\\
4&290&1&3&5&0&1&5&21&65&183&551&\texttt{G0N290k4A}\\
\hline
5&4&1&2&0&0&2&8&26&80&242&728&\texttt{G1N4k5A}\\
5&7&1&2&0&0&2&8&26&80&242&728&\texttt{G1N7k5B}\\
5&19&1&2&0&0&2&8&26&80&242&728&\texttt{G1N19k5C}\\
5&31&3&2&0&0&2&8&26&80&242&728&\texttt{G1N31k5D}\\
5&40&1&4&2&0&1&4&10&28&82&244&\texttt{G1N40k5F}\\
5&40&1&2&0&0&2&8&26&80&242&728&\texttt{G1N40k5G}\\
\hline
6&20&2&6&1&0&1&5&18&47&144&425&\texttt{G1N20k6B}\\
6&26&1&5&1&0&1&5&17&47&143&425&\texttt{G0N26k6A}\\
6&28&2&4&1&0&2&5&16&47&142&425&\texttt{G1N28k6F}\\ 
6&32&1&7&1&0&1&5&19&47&145&425&\texttt{G0N32k6A}\\
6&49&1&4&2&0&2&6&16&48&142&426&\texttt{G0N49k6B}\\
6&64&1&4&4&0&2&8&16&50&142&428&\texttt{G0N64k6A}\\
6&208&1&4&8&0&2&6&16&54&142&432&\texttt{G0N208k6A}\\
6&256&1&9&1&0&1&5&21&47&147&425&\texttt{G0N256k6A}\\
6&256&1&4&2&0&2&6&16&48&142&426&\texttt{G0N256k6B}\\ 
\hline
7&7&1&4&1&0&2&7&22&67&202&607&\texttt{G1N7k7C}\\
7&16&1&4&1&0&2&7&22&67&202&607&\texttt{G1N16k7A}\\
7&19&1&4&1&0&2&7&22&67&202&607&\texttt{G1N19k7C}\\
7&20&2&4&1&0&2&7&22&67&202&607&\texttt{G1N20k7E}\\ 
7&31&1&4&1&0&2&7&22&67&202&607&\texttt{G1N31k7D}\\
7&31&2&4&1&0&2&7&22&67&202&607&\texttt{G1N31k7E}\\
7&40&1&4&1&0&2&7&22&67&202&607&\texttt{G1N40k7E}\\
7&40&1&4&1&0&2&7&22&67&202&607&\texttt{G1N40k7F}\\
\hline
\end{tabular}
}
\end{minipage}%
\begin{minipage}[t]{.5\linewidth}
    \vspace{-200pt}
\scalebox{.7}{
\begin{tabular}{|ccl||cc||c|c|c|c|c|c|c||c|}
\hline
$k$ 	& $N$ &$d$&$\lambda^+$& $\lambda^-$& 0& 1 & 2 &3&4&5&6& label\\ \hline 
8&20&2&5&4&0&1&6&11&36&101&306&\texttt{G1N20k8C}\\
8&28&2&5&2&0&1&4&11&34&101&304&\texttt{G1N28k8E}\\
8&32&1&5&5&0&1&7&11&37&101&307&\texttt{G0N32k8A}\\
8&49&1&6&2&0&2&4&12&34&102&304&\texttt{G0N49k8B}\\
8&64&1&8&2&0&2&4&14&34&104&304&\texttt{G0N64k8A}\\ 
8&256&1&5&7&0&1&3&11&39&101&309&\texttt{G0N256k8A}\\
8&256&1&6&2&0&2&4&12&34&102&304&\texttt{G0N256k8B}\\
\hline
9&4&1&5&1&0&1&5&17&53&161&485&\texttt{G1N4k9A}\\
9&7&1&5&1&0&1&5&17&53&161&485&\texttt{G1N7k9B}\\
9&19&1&5&1&0&1&5&17&53&161&485&\texttt{G1N19k9C}\\
9&31&3&6&2&0&2&6&18&54&162&486&\texttt{G1N31k9D}\\
9&40&1&5&1&0&1&5&17&53&161&485&\texttt{G1N40k9E}\\
9&40&1&7&3&0&1&7&19&55&163&487&\texttt{G1N40k9F}\\
\hline
10&20&2&6&2&0&2&8&24&74&222&668&\texttt{G1N20k10C}\\
10&28&2&6&2&0&2&8&24&74&222&668&\texttt{G1N28k10E}\\
10&32&1&6&2&0&2&8&24&74&222&668&\texttt{G0N32k10A}\\
10&49&1&6&2&0&2&8&24&74&222&668&\texttt{G0N49k10A}\\
10&64&1&6&2&0&2&8&24&74&222&668&\texttt{G0N64k10A}\\
10&256&1&7&9&0&1&3&25&27&223&675&\texttt{G0N256k10A}\\
10&256&1&6&2&0&2&8&24&74&222&668&\texttt{G0N256k10B}\\
\hline
\end{tabular}
}
\end{minipage} 
\end{table}
\clearpage

\newpage
\begin{table}[!htb]
    \caption{$\lambda$-invariants at $p=5$.}
        \label{table_p5}
    \begin{minipage}[t]{.54\linewidth}
    %\vspace{-240pt}
 \scalebox{.7}{
\begin{tabular}{|ccl||cc||c|c|c|c|c|c||c|}
\hline
$k$ 	& $N$ &$d$&$\lambda^+$& $\lambda^-$& 0&1 & 2 &3&4&5& label\\ \hline
%%%%%%%%%%%%%%%%%%%%
2&18&2&0&0&0&0&0&4&20&104&\texttt{G1N18k2A}\\
2&21&2&0&0&0&0&0&4&20&104&\texttt{G1N21k2B}\\
2&24&2&0&0&0&0&0&4&20&104&\texttt{G1N24k2C}\\
2&28&2&0&1&$\infty$&0&0&5&20&105&\texttt{G1N28k2C}\\
2&38&2&0&0&0&0&0&4&20&104&\texttt{G1N38k2D}\\
2&39&4&0&0&0&0&0&4&20&104&\texttt{G1N39k2C}\\
2&39&4*&0&0&0&0&0&4&20&104&\texttt{G1N39k2C}\\
2&39&2&0&1&$\infty$&0&0&5&20&105&\texttt{G1N39k2I}\\
\hline
3&7&1&1&0&0&1&8&41&208&1041&\texttt{G1N7k3A}\\
3&8&1&1&0&0&1&8&41&208&1041&\texttt{G1N8k3A}\\
3&12&1&1&0&0&1&8&41&208&1041&\texttt{G1N12k3B}\\
3&21&2&1&0&0&1&8&41&208&1041&\texttt{G1N21k3E}\\
3&23&3&1&0&0&1&8&41&208&1041&\texttt{G1N23k3B}\\
3&27&1&1&0&0&1&8&41&208&1041&\texttt{G1N27k3C}\\
3&32&1&1&0&0&1&8&41&208&1041&\texttt{G1N32k3C}\\
3&39&2&1&0&0&1&8&41&208&1041&\texttt{G1N39k3D}\\
3&39&2&1&0&0&1&8&41&208&1041&\texttt{G1N39k3G}\\ 
\hline
4&9&1&2&0&0&2&12&62&312&1562&\texttt{G0N9k4A}\\
4&21&2&2&0&0&2&12&62&312&1562&\texttt{G1N21k4G}\\
4&24&2&2&0&0&2&12&62&312&1562&\texttt{G1N24k4C}\\
4&28&2&2&0&0&2&12&62&312&1562&\texttt{G1N28k4E}\\
4&39&4&2&0&0&2&12&62&312&1562&\texttt{G1N39k4H}\\
4&39&4*&2&0&0&2&12&62&312&1562&\texttt{G1N39k4H}\\
4&49&1&2&0&0&2&12&62&312&1562&\texttt{G0N49k4D}\\
4&66&1&2&0&0&2&12&62&312&1562&\texttt{G0N66k4B}\\
4&108&1&2&0&0&2&12&62&312&1562&\texttt{G0N108k4A}\\
4&108&1&2&0&0&2&12&62&312&1562&\texttt{G0N108k4B}\\
4&132&1&2&0&0&2&12&62&312&1562&\texttt{G0N132k4B}\\
4&144&1&2&0&0&2&12&62&312&1562&\texttt{G0N144k4A}\\
4&174&1&2&0&0&2&12&62&312&1562&\texttt{G0N174k4C}\\
4&182&1&2&0&0&2&12&62&312&1562&\texttt{G0N182k4D}\\
4&198&1&2&0&0&2&12&62&312&1562&\texttt{G0N198k4A}\\
4&222&1&2&0&0&2&12&62&312&1562&\texttt{G0N222k4B}\\
4&224&1&2&0&0&2&12&62&312&1562&\texttt{G0N224k4A}\\
4&224&1&2&0&0&2&12&62&312&1562&\texttt{G0N224k4B}\\
4&243&1&2&0&0&2&12&62&312&1562&\texttt{G0N243k4A}\\
4&243&1&2&0&0&2&12&62&312&1562&\texttt{G0N243k4B}\\
4&256&1&2&0&0&2&12&62&312&1562&\texttt{G0N256k4G}\\
4&256&1&2&0&0&2&12&62&312&1562&\texttt{G0N256k4H}\\
4&273&1&2&0&0&2&12&62&312&1562&\texttt{G0N273k4C}\\
\hline
5&7&1&3&0&0&3&16&83&416&2083&\texttt{G1N7k5B}\\
5&8&1&3&0&0&3&16&83&416&2083&\texttt{G1N8k5A}\\
5&12&1&3&0&0&3&16&83&416&2083&\texttt{G1N12k5B}\\
5&21&2&3&0&0&3&16&83&416&2083&\texttt{G1N21k5E}\\
5&23&3&3&0&0&3&16&83&416&2083&\texttt{G1N23k5B}\\
5&27&1&3&0&0&3&16&83&416&2083&\texttt{G1N27k5C}\\
5&32&1&3&0&0&3&16&83&416&2083&\texttt{G1N32k5D}\\
5&39&2&3&0&0&3&16&83&416&2083&\texttt{G1N39k5D}\\
5&39&2&3&0&0&3&16&83&416&2083&\texttt{G1N39k5G}\\ 
\hline
6&21&2&4&0&0&4&20&104&520&2604&\texttt{G1N21k6E}\\
6&24&2&4&0&0&4&20&104&520&2604&\texttt{G1N24k6E}\\
6&27&1&4&0&0&4&20&104&520&2604&\texttt{G0N27k6A}\\
6&28&2&5&0&0&1&20&105&520&2605&\texttt{G1N28k6F}\\
6&36&1&4&0&0&4&20&104&520&2604&\texttt{G0N36k6A}\\
6&39&4&4&0&0&4&20&104&520&2604&\texttt{G1N39k6E}\\
6&39&4*&4&0&0&4&20&104&520&2604&\texttt{G1N39k6E}\\
\hline
\end{tabular}
}
    \end{minipage}%
    \begin{minipage}{.5\linewidth}
    \hspace{1mm}
    \vspace{5.4mm}
\scalebox{.7}{
\begin{tabular}{|ccl||cc||c|c|c|c|c|c||c|}
\hline
$k$ 	& $N$ &$d$&$\lambda^+$& $\lambda^-$& 0&1 & 2 &3&4&5& label\\ \hline
%%%%%%%%%%%%%%%%%%%%
6&49&1&4&0&0&4&20&104&520&2604&\texttt{G0N49k6B}\\
6&108&1&4&0&0&4&20&104&520&2604&\texttt{G0N108k6A}\\
6&144&1&4&0&0&4&20&104&520&2604&\texttt{G0N144k6A}\\
6&243&1&4&0&0&4&20&104&520&2604&\texttt{G0N243k6A}\\
6&243&1&5&1&0&1&21&105&521&2605&\texttt{G0N243k6B}\\
6&256&1&4&0&0&4&20&104&520&2604&\texttt{G0N256k6C}\\
6&256&1&5&1&0&1&21&105&521&2605&\texttt{G0N256k6D}\\
\hline
7&3&1&4&0&0&4&24&124&624&3124&\texttt{G1N3k7A}\\
7&7&1&4&0&0&4&24&124&624&3124&\texttt{G1N7k7C}\\
7&8&1&4&0&0&4&24&124&624&3124&\texttt{G1N8k7A}\\
7&23&1&5&1&0&1&5&25&125&625&\texttt{G1N23k7B}\\
7&23&2&4&0&0&4&24&124&624&3124&\texttt{G1N23k7C}\\
7&23&2*&5&1&0&1&5&25&125&625&\texttt{G1N23k7C}\\
7&32&1&6&2&0&2&6&26&126&626&\texttt{G1N32k7D}\\
7&39&2&5&1&0&1&5&25&125&625&\texttt{G1N39k7G}\\
\hline
8&21&2&6&1&0&2&9&46&229&1146&\texttt{G1N21k8E}\\
8&24&2&6&1&0&2&9&46&229&1146&\texttt{G1N24k8E}\\
8&27&1&6&1&0&2&9&46&229&1146&\texttt{G0N27k8A}\\
8&28&2&6&1&0&2&9&46&229&1146&\texttt{G1N28k8E}\\
8&36&1&6&1&0&2&9&46&229&1146&\texttt{G0N36k8A}\\
8&39&4&6&1&0&2&9&46&229&1146&\texttt{G1N39k8E}\\
8&39&4*&6&2&0&2&10&46&230&1146&\texttt{G1N39k8E}\\
8&49&1&10&1&0&1&9&50&229&1150&\texttt{G0N49k8B}\\
8&66&1&6&1&0&2&9&46&229&1146&\texttt{G0N66k8A}\\
8&108&1&6&1&0&2&9&46&229&1146&\texttt{G0N108k8A}\\
8&144&1&6&1&0&2&9&46&229&1146&\texttt{G0N144k8A}\\
8&198&1&6&1&0&2&9&46&229&1146&\texttt{G0N198k8D}\\
8&243&1&10&1&0&1&9&50&229&1150&\texttt{G0N243k8A}\\
8&243&1&6&1&0&2&9&46&229&1146&\texttt{G0N243k8B}\\
8&256&1&6&1&0&2&9&46&229&1146&\texttt{G0N256k8C}\\
8&256&1&6&1&0&2&9&46&229&1146&\texttt{G0N256k8D}\\
\hline
9&7&1&7&1&0&3&13&67&333&1667&\texttt{G1N7k9B}\\
9&8&1&7&1&0&3&13&67&333&1667&\texttt{G1N8k9A}\\
9&12&1&8&1&0&4&13&68&333&1668&\texttt{G1N12k9B}\\
9&21&2&7&1&0&3&13&67&333&1667&\texttt{G1N21k9D}\\
9&23&3&7&1&0&3&13&67&333&1667&\texttt{G1N23k9B}\\
9&27&1&7&1&0&3&13&67&333&1667&\texttt{G1N27k9C}\\
9&32&1&7&1&0&3&13&67&333&1667&\texttt{G1N32k9D}\\
9&39&2&7&1&0&3&13&67&333&1667&\texttt{G1N39k9D}\\
9&39&2&7&1&0&3&13&67&333&1667&\texttt{G1N39k9G}\\
\hline
10&9&1&9&1&0&1&17&89&437&2189&\texttt{G0N9k10A}\\
10&21&2&8&1&0&4&17&88&437&2188&\texttt{G1N21k10H}\\
10&24&2&8&1&0&4&17&88&437&2188&\texttt{G1N24k10F}\\
10&28&2&8&1&0&4&17&88&437&2188&\texttt{G1N28k10E}\\
10&39&4&8&1&0&4&17&88&437&2188&\texttt{G1N39k10H}\\
10&39&4*&8&1&0&4&17&88&437&2188&\texttt{G1N39k10H}\\
10&49&1&8&2&0&4&18&88&438&2188&\texttt{G0N49k10A}\\
10&108&1&11&1&0&3&17&91&437&2191&\texttt{G0N108k10A}\\
10&108&1&8&2&0&4&18&88&438&2188&\texttt{G0N108k10B}\\
10&144&1&8&2&0&4&18&88&438&2188&\texttt{G0N144k10A}\\
10&243&1&8&4&0&4&20&88&440&2188&\texttt{G0N243k10A}\\
10&243&1&9&1&0&1&17&89&437&2189&\texttt{G0N243k10B}\\
10&256&1&11&1&0&3&17&91&437&2191&\texttt{G0N256k10C}\\
10&256&1&8&4&0&4&20&88&440&2188&\texttt{G0N256k10D}\\
\hline
\end{tabular}
}
\end{minipage} 
\end{table}
\clearpage
\newpage

\clearpage

\bibliographystyle{alpha}
\bibliography{references}

\end{document}